\documentclass[reqno,11pt]{article}
\usepackage{mathrsfs}
\usepackage{amssymb}
\usepackage{amsthm}
\usepackage{amsmath}
\usepackage{amsfonts}
\usepackage{mathtools}
\usepackage{enumerate}
\usepackage{bm}

\usepackage{graphicx}

\usepackage{bbm}
\usepackage{mathrsfs}
\usepackage{amssymb}
\usepackage[thicklines]{cancel}

\UseRawInputEncoding

\usepackage[usenames,dvipsnames]{xcolor}
\usepackage{soul}

\usepackage{txfonts}

\usepackage{psfrag}
\usepackage[numbers,sort&compress]{natbib}
\usepackage{setspace,color,wrapfig}
\usepackage[colorlinks,linkcolor=blue,citecolor=cyan]{hyperref}

\numberwithin{equation}{section}
\newtheorem{theorem}{Theorem}[section]
\newtheorem{lemma}[theorem]{Lemma}

\newtheorem{problem}[theorem]{Problem}
\theoremstyle{definition}

\theoremstyle{remark}

\begin{document}
\title{ Transonic shocks  for steady Euler flows with rotating effect  in two-dimensional almost flat nozzles}
\author{
Zihao Zhang\thanks{School of Mathematics, Jilin University, Changchun, Jilin Province, China, 130012. Email: zhangzihao@jlu.edu.cn}}
\date{}
\maketitle
\newcommand{\de}{{\mathrm{d}}}
\def\div{{\rm div\,}}
\def\curl{{\rm curl\,}}
\def\om{\omega}
\def\th{\theta}
\def\la{\lambda}
\newcommand{\ro}{{\rm rot}}
\newcommand{\sr}{{\rm supp}}
\newcommand{\sa}{{\rm sup}}
\newcommand{\va}{{\varphi}}
\newcommand{\md}{\mathcal{D}}
\newcommand{\me}{\mathcal{M}}
\newcommand{\ml}{\mathcal{V}}
\newcommand{\mg}{\mathcal{G}}
\newcommand{\mh}{\mathcal{H}}
\newcommand{\mf}{\mathcal{F}}
\newcommand{\ms}{\mathcal{S}}
\newcommand{\mt}{\mathcal{T}}
\newcommand{\mn}{\mathcal{N}}
\newcommand{\mb}{\mathcal{P}}
\newcommand{\mm}{\mathcal{B}}
\newcommand{\mj}{\mathcal{J}}
\newcommand{\mk}{\mathcal{K}}
\newcommand{\my}{\mathcal{U}}
\newcommand{\mw}{\mathcal{W}}
\newcommand{\mq}{\mathcal{Q}}
\newcommand{\ma}{\mathcal{A}}
\newcommand{\mc}{\mathcal{C}}
\newcommand{\mi}{\mathcal{I}}
\newcommand{\n}{\nabla}
\newcommand{\e}{\tilde}
\newcommand{\m}{\Omega}
\newcommand{\h}{\hat}
\newcommand{\x}{\bar}
 \newcommand{\q}{{\rm R}}
\newcommand{\p}{{\partial}}
\newcommand{\z}{{\varepsilon}}
\renewcommand\figurename{\scriptsize Fig}
\pagestyle{myheadings} \markboth{\hfill Transonic shocks  for steady Euler flows with rotating effect  \hfill}{\hfill  Transonic shocks  for steady Euler flows with rotating effect\hfill}\maketitle
\begin{abstract}
  We address the existence and stability of transonic shocks for the two-dimensional steady  rotating Euler system in an almost flat nozzle. Under the influence of the Coriolis force, we first establish a class of special
transonic shock solutions in a flat nozzle, whose  states depend  on the vertical variable. It is shown   that these solutions exist if and only if the upstream Mach number
satisfies certain conditions, while the shock position  is arbitrary.  We then determine the shock position and establish the existence of the transonic shock solution under small perturbations of the incoming supersonic flow, the exit pressure, and  the upper nozzle wall. The problem is formulated as a free boundary problem for a hyperbolic-elliptic mixed nonlinear system.  We   decompose the  hyperbolic and elliptic modes in terms of the deformation and vorticity,   and analyze the solvability condition to determine the admissible shock positions. Starting from the obtained  initial approximation  of the shock solution, a nonlinear iteration scheme can be constructed to derive a transonic shock solution   in which the shock front is close to the initial approximating position.
\end{abstract}
\begin{center}
\begin{minipage}{5.5in}
\end{minipage}
\end{center}
\tableofcontents
\section{Introduction  }\noindent
\par In this paper, we are concerned with the existence of transonic shocks for steady Euler flows with rotating effect in an almost flat nozzle, which is governed by the following  two-dimensional steady  compressible rotating Euler  system:
\begin{equation}\label{1-1}
\begin{cases}
\p_{x_1}(\rho u_1)+\p_{x_2}(\rho u_2)=0,\\
\p_{x_1}(\rho u_1^2+P)+\p_{x_2}(\rho u_1u_2)-\beta\rho u_2=0,\\
\p_{x_1}(\rho u_1u_2)+\p_{x_2}(\rho u_2^2+P)+\beta\rho u_1=0,\\
\p_{x_1}(\rho u_1 E+u_1 P)+\p_{x_2}(\rho u_2 E+u_2 P)=0,\\
\end{cases}
\end{equation}
  where ${\bf u}=(u_1,u_2)$ is the velocity, $\rho$ is the density, $P$ is the pressure, $E$ is the
 energy.  The term $ \rho {\bf u}^\bot=(-\rho u_2,\rho u_1)^T $ on the second and third equations in \eqref{1-1}
corresponds to the Coriolis force arising from the Earth's rotation. The constant $\beta>0 $ represents the speed of rotation, which is called the Coriolis parameter.  For ideal polytropic gas, the total energy and the pressure are given by
\begin{equation*}
 E=\frac1 2|{\bf u}|^2+\frac{ P}{(\gamma-1)\rho}
\ \ {\rm {and}}\ \ P(\rho,S)= e^{S}\rho^{\gamma},
\end{equation*}
where $S$ is  the entropy and $ \gamma> 1 $ is the adiabatic exponent.  We define the Bernoulli function $B$  and the Mach number $M $ by
\begin{equation*}
B=\frac{1}{2}|{\bf u}|^2+\frac{\gamma P}{(\gamma-1)\rho} \ \ {\rm{and}} \ \  M=\frac{|{\bf u}|}{{c(\rho, S)}},
\end{equation*}
where
$ c(\rho, S)=\sqrt{\p_\rho P(\rho,S)}$ is called the local sound speed. It is well-known that  the system \eqref{1-1} is hyperbolic for supersonic flows ($ M>1 $), and hyperbolic-elliptic mixed for subsonic flows ($ M<1 $).
\par    Let $ \mathbb{D}\subset\mathbb{R}^2  $ be an open and connected domain. Suppose that a non-self-intersecting $ C^1$-curve  $ \Gamma$ divides $ \mathbb{D} $ into two disjoint open subsets $ \mathbb{D}^\pm$  such that $ \mathbb{D}= \mathbb{D}^- \cup\Gamma \cup \mathbb{D}^+ $.
Assume that $ \bm U = (\rho, u_1, u_2,P)$ is a $C^1$ solution of the rotating Euler system \eqref{1-1} in each domain $ \mathbb{D}^+ $ and $ \mathbb{D}^- $ and is continuous up to the boundary $ \Gamma $ in each sub-domain. If   $ \bm U $ is a weak solution for \eqref{1-1} in domain $ \mathbb{D} $, by integration by parts, we see that  the  Rankine-Hugoniot conditions hold along $\Gamma$  almost everywhere:
\begin{equation}\label{1-2}
\begin{cases}
n_1[\rho u_1]+n_2[\rho u_2]=0,\\
n_1[\rho u_1^2+P]+n_2[\rho u_1u_2]=0,\\
n_1[\rho u_1u_2]+n_2[\rho u_2^2+P]=0,\\
n_1[ \rho u_1B]+n_2[\rho u_2B]=0,\\
\end{cases}
\end{equation}
where $ \mathbf{n} = (n_1, n_2) $  is the unit normal vector to $\Gamma$, and $[G](\mathbf{x}) = G_+(\mathbf{x})- G_-(\mathbf{x})$
denotes the jump across the  curve $\Gamma$ for a piecewise smooth function $ G $.
 Let $ \bm \tau= (\tau_{1},\tau_{2})$ { be} the unit tangential vector to $\Gamma$,  {then} $ \mathbf{n}\cdot{\bm \tau}= 0$. Taking the { inner} product of $(\eqref{1-2}_2,\eqref{1-2}_3)$ with $ \mathbf{n}$  and $ \bm \tau $ respectively, we have
\begin{equation}\label{1-3}
\rho(\mathbf{u}\cdot\mathbf{n})[\mathbf{u}\cdot \bm \tau]_{\Gamma}= 0, \quad
[\rho(\mathbf{u}\cdot\mathbf{n})^2+P]_{\Gamma}= 0.
\end{equation}
 Assume that $ \rho> 0 $  in $ \overline {\mathbb{D}} $, \eqref{1-3} is divided into  two subcases:
 \begin{itemize}
 \item    $ \mathbf{u}\cdot\mathbf{n}=0 $ and $ [P]=0 $ on $\Gamma$. In this case, the curve $ \Gamma $ is  a contact discontinuity;
     \item  $ \mathbf{u}\cdot\mathbf{n} \neq 0 $ and $[\mathbf{u}\cdot \bm \tau]_{\Gamma}=0 $. In this case,  the curve $ \Gamma $ is a shock.
 \end{itemize}
 \begin{figure}
  \centering
  \includegraphics[width=9cm,height=4.5cm]{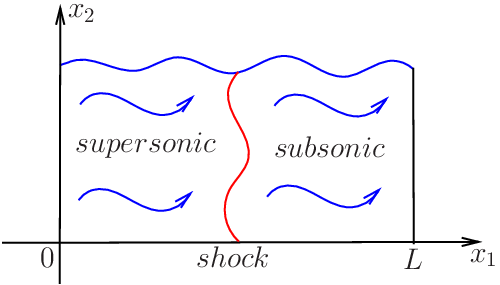}
  \caption{Transonic shock problem}
\end{figure}
\par The analysis of transonic shocks for the Euler equations can be traced back to the
 work   of Courant and Friedrichs in \cite{cf48}. They pointed out that the position of the shock front cannot be uniquely
determined unless additional conditions are imposed at the exit and the pressure
condition is  preferred. Over the past two decades, the stability of two types of transonic shock solutions to the steady Euler equations has been extensively studied.
The first type is   symmetric transonic shock solutons in divergent nozzles, such as radially symmetric solutions in divergent nozzles and spherically symmetric solutions in conical domains, for which the position of the shock front is uniquely determined by the exit pressure.  The authors in \cite{LXY09,LXY13} had proved the existence and
stability of the transonic shock solution in divergent sectors under general
perturbation of the nozzle wall and the exit pressure. Similar results for the
axisymmetric perturbations of the nozzle wall  and the pressure had been
established in \cite{LXY10,WXX21}. The  stability of spherically
symmetric subsonic flows and transonic shocks in a spherical shell was established  in \cite{LXY16} by requiring that the background solutions satisfy some ``Structural
Conditions". Recently, the authors in \cite{WX23,WX25} removed the ``Structural Conditions"
and established the existence and stability of cylindrical and spherical transonic shocks under three-dimensional perturbations of the incoming flow
and the exit pressure. In
\cite{WY25,WZZ25}, the authors  proved the structural stability of  transonic shock solutions in
flat nozzles under an external force. it was
shown that the external force has a stabilization effect on the transonic shock in the nozzle and the
shock position is uniquely determined. Furthermore, the authors in \cite{WXY21,WYX21} studied  radial symmetric spiral  flows with or without shock in an annulus. Due to nonzero angular velocity, it was found in \cite{WXY21} that besides the  supersonic-subsonic shock in a divergent nozzle, there exists a supersonic-supersonic shock solution, where the downstream state may change smoothly from supersonic to subsonic.
 \par The second type is   transonic  shock solutions in flat nozzles, with constant upstream supersonic and downstream subsonic states and an arbitrarily prescribed shock position. For the steady potential flow model, the authors in \cite{CGF03,CFM04,XY05, XY08} have investigated the existence, uniqueness, and stability of transonic shock solutions under various geometric configurations and boundary conditions. In \cite{XY05, XY08}, the well-posedness of steady potential flows were studied in general two-dimensional channels with variable cross section and three-dimensional slowly varying channels. For the full Euler equations, the existence and stability of the transonic shock for  2-D and 3-D Euler flows with the prescribed pressure at the exit up to a constant were studied in \cite{CS05,CS08, CY08,XYY09}. Both existence results are established under the assumption that the shock front passes through a given point.  Without such an artificial assumption, The authors in \cite{FX21}  established the existence and stability of transonic shock solutions
 in an almost flat finite nozzle with the
exit pressure, where the shock position was uniquely determined.  The
key idea  is to first solve a free boundary problem for the linearized Euler system to obtain the most desirable information on the location of the shock front. Then using this solution as an initial approximation,  a nonlinear
iteration scheme can be constructed which leads to a transonic shock solution with
the position of the shock front being close to the initial approximating position. The corresponding results for 2-D Euler flows with  gravity, 3-D  axisymmetric
 Euler flows with or without swirl and  2-D  MHD flows  were
also obtained in \cite{FG21,FG22,FG26-2,FG26-1,WY26}.
\par The  rotating Euler system can be regarded as the classical compressible Euler system with the Coriolis force.  The significance of rotational effects in fluids has been widely demonstrated in a variety of physical settings,  in particular in the geophysical fluids literature
 \cite{GS07,PJ87} or  for the $ \beta $-plane model \cite{EW17,PW18}. The existence of subsonic  flows with rotating effect in an infinitely long axisymmetric nozzle was obtained in \cite{ZZH22}. Most recently,  the authors in \cite{WSZ25} established the stability of supersonic contact discontinuity for the two-dimensional steady rotating Euler system in a finitely curved nozzle. However, there are no works on transonic shock solutions  for the steady rotating Euler system. Consequently, this paper aims to establish  the existence and stability  of transonic shock solutions for the two-dimensional rotating Euler flows  in an almost flat nozzle (see Fig 1), where the location of the shock front is determined by the boundary data. To this end,  our analysis relies on two main ingredients.
\par  First of all,  we  establish a class of special
transonic shock solutions to the steady rotating Euler system \eqref{1-1} in a flat nozzle under
the assumption that the states of the flow both ahead of and behind the shock
front depend only on the vertical variable. A primary challenge arising from the Coriolis force is the absence of standard background shock solutions, in contrast to steady Euler flows, for which trivial piecewise constant shock solutions are easily obtained. The construction of these  solutions is non-trivial, as the flow must simultaneously satisfies the prescribed boundary conditions and the Rankine-Hugoniot conditions across the shock front.  We demonstrate that   these transonic shock solutions exist if and only if the upstream Mach number satisfies the certain condition. Such necessary and sufficient conditions  are detailed in  Lemma \ref{bck}.
\par Based on the special transonic shock solutions, we will determine the location of the shock front and establish the existence of transonic shock solutions under small perturbations of the
incoming supersonic flow, the exit pressure, and the upper nozzle wall. Mathematically, it can be formulated as a free boundary value problem to the steady rotating Euler equations which is hyperbolic-elliptic mixed in subsonic regions. The previous works \cite{FX21,FG21,FG22,FG26-2,FG26-1}  utilized the characteristic decomposition   to reduce the steady Euler system  into a  elliptic system of the flow angle and
pressure. In contrast to \cite{FX21,FG21,FG22,FG26-2,FG26-1},  we decouple the hyperbolic and elliptic modes in terms of  the deformation tensor and vorticity developed in \cite{WX19} and reformulate  the steady rotating  Euler system into two transport equations for the Bernoulli quantity and  the entropy and  a deformation-curl system for the velocity field.  For the slightly curved nozzle in this paper, the corner singularity will be
transported along the trajectory,  the  $C^{\alpha}$ regularity for the flow in subsonic region is optimal.  Therefore, the streamline may not be
uniquely determined. We introduce the Lagrangian transformation to over this difficulty. To establish  the existence of transonic shock solutions in  the Lagrangian coordinates, one of the key  difficulties lies in obtaining information on the position of the
shock front since it can be arbitrary for the special solution. We focus on the contributions the Coriolis force and
the perturbations of the flat nozzle  in determining the
position of shock front.  Inspired by \cite{FX21}, a  free boundary problem for the linearized  deformation-curl first order elliptic system
with variable coefficients will be proposed.  In order to ensure such problem admits a solution, the boundary data must satisfy a solvability condition. The solvability condition analysis identifies a plausible mechanism for determining the admissible shock position, where the Coriolis force exerts a dominant  role. Once the initial approximation is obtained, a nonlinear iteration scheme can be constructed to derive a transonic shock solution.
\par This paper will be arranged as follows. In Section 2, we formulate the problem in detail and state the main result. In Section 3, we establish the well-posedness theory for boundary value problems of a  first order linear elliptic
 system  with variable coefficients,  which will be employed later in solving the linearized
problem for subsonic flows.
In Section 4,  with the help of the analysis in
Section 3, we establish  the existence of the initial approximation for the shock solution and  determine the approximate position of the shock front
by applying the solvability condition obtained in Section 3.
  Based on the initial approximation, Section 5 develops a nonlinear iteration scheme, and  demonstrates that it is well-defined and contractive, and consequently prove the main theorem.
 \section{Mathematical problem and main results}\noindent
 \par In this section, we give a detailed formulation of transonic shock flows to the steady rotating Euler system  and state the main result.
\subsection{The background transonic shock solutions}\noindent
 \par We  first construct a  class of special normal transonic shock solutions  in a finitely long flat nozzle.  The two-dimensional  flat  nozzle  of length $L $  is given by
\begin{equation*}
\bar\Omega= \{(x_1,x_2): 0<x_1<L, \ 0<x_2<1\}.
\end{equation*}
Let $\bar x_s\in(0,L)$  be an arbitrary  normal shock location, which divided the domain into two regions:
\begin{equation*}
\bar\Omega_{-}=\bar\Omega\cap\{ 0<x_1<\bar x_s\}, \quad \bar\Omega_{+}=\bar\Omega\cap\{ \bar x_s<x_1< L\}.
\end{equation*}
Due to the effect of   the Coriolis force, we will construct special normal transonic shock solutions
 of the form
 \begin{align}\label{1-4}
  &\bar {\bm U}_-(x_2)=\left(\bar \rho_-(x_2),\bar u_-(x_2),0,\bar P_-(x_2)\right),
     \ \ {\rm{in}}\ \  \bar\Omega_{-},\\\label{1-5}
      &\bar {\bm U}_+(x_2)=\left(\bar \rho_+(x_2),\bar u_+(x_2),0,\bar P_+(x_2)\right),\ \ {\rm{in}}\ \  \bar\Omega_{+},
      \end{align}
      which satisfy  the following Rankine-Hugoniot conditions:
      \begin{equation}\label{1-6}
\begin{cases}
\begin{aligned}
&[\bar\rho \bar u]=\bar\rho_+\bar u_+ -\bar\rho_- \bar u_-=0,\\
&[\bar\rho \bar u^2+ \bar P]=(\bar\rho_+ \bar u_+^2+ \bar P_+)-(\bar\rho_- \bar u_-^2+ \bar P_-)=0,\\
&[\bar B]=\bar B_+-\bar B_-=\bigg(\frac12\bar u_+^2+\frac{\gamma \bar P_+}{(\gamma-1)\bar\rho_+}\bigg)-\bigg(\frac12\bar u_-^2+\frac{\gamma \bar P_-}{(\gamma-1)\bar\rho_-}\bigg)=0.\\
\end{aligned}
\end{cases}
\end{equation}
\par Denote the  Mach number by \begin{equation*}
\bar M_-(x_2)=\frac{\bar u_-}{\sqrt{{\gamma \bar P_-}/{\bar\rho_-}}}, \ \ x_2\in[0,1].
\end{equation*}
 The existence of such special normal transonic shock solutions in the form of  \eqref{1-4} and \eqref{1-5} can be
 established in the following Lemma.
 \begin{lemma}\label{bck}
 For given $ \bar u_-\in C^3([0,1])$ with $ \displaystyle{\min_{[0,1]} \bar u_->0}$, $\bar M_-(1)>1 $  and  $\bar P_-(1)>0$, there exist special normal transonic shock solutions  $\left(\bar {\bm U}_\pm(x_2), \bar x_s\right) $ such that \eqref{1-1} and \eqref{1-6} are satisfied if and only if
 \begin{equation}\label{1-7}
 \bigg(\frac{1}{\bar M_-^2}\bigg)'(x_2)=\frac{\beta\gamma}{\bar u_-}\bigg(\frac{1}{\bar M_-^2}-1\bigg), \ \ x_2\in[0,1].
 \end{equation}
 \end{lemma}
 \begin{proof}
 The proof is divided into two steps.
 \par{ \bf Step 1.} It follows from \eqref{1-1} and \eqref{1-6} that for any $x_2\in[0,1]$,
 \begin{equation}\label{1-8}
 \begin{cases}
 \begin{aligned}
 &\bar P_-'(x_2)=-\beta(\bar\rho_- \bar u_-),  \\
 &\bar P_+'(x_2)=-\beta(\bar\rho_+ \bar u_+),  \\
 &\bar u_+(x_2)=\frac{2(\gamma-1)}{(\gamma+1)\bar u_-}\left(\frac{1}{2}\bar u_-^2+
\frac{\gamma \bar P_{-}}{(\gamma-1)\bar\rho_-}\right),\\
&\bar P_{+}(x_2)=\left(\bigg(1 + \frac{\gamma-1}{\gamma+1}\bigg)\bar M_-^2- \frac{\gamma-1}{\gamma+1}\right)\bar P_-=\frac{2 \bar \rho_-\bar u_-^2}{\gamma+1}- \frac{\gamma-1}{\gamma+1}\bar P_-,
\\
&\bar \rho_+(x_2)=\frac{\bar\rho_-\bar u_-}{\bar u_+}.
\end{aligned}
\end{cases}
\end{equation}
Then the fourth equation in \eqref{1-8} implies  that
\begin{equation}\label{1-9}
(\bar \rho_-\bar u_-^2)'(x_2)=-\beta\gamma\bar \rho_-\bar u_-,
 \ \ x_2\in[0,1].
\end{equation}
Note that
\begin{equation}\label{1-10}
\bigg(\frac{1}{\bar M_-^2}\bigg)'(x_2)=\frac{\gamma\bar P_-'}{\bar \rho_-\bar u^2_-}-\frac{\gamma\bar P_-(\bar \rho_-\bar u_-^2)'}{\bar \rho_-^2\bar u^4_-}, \ \ x_2\in[0,1].
\end{equation}
Thus \eqref{1-7} follows from the first equation in \eqref{1-8}, \eqref{1-9} and \eqref{1-10}.
 On the other hand, one has
\begin{equation}\label{1-11}
\begin{aligned}
\bar P_{+}'(x_2)&=\frac{2 (\bar \rho_-\bar u_-^2)'}{\gamma+1}- \frac{\gamma-1}{\gamma+1}\bar P_-'\\
&=\frac{2 }{\gamma+1}\frac{\bar \rho_-\bar u^2_-\bar P_-'}{\bar P_-}-\frac{2 }{\gamma+1}\bigg(\frac{\bar \rho_-^2\bar u^4_-}{\gamma\bar P_-}\bigg(\frac{1}{\bar M_-^2}\bigg)'\bigg)- \frac{\gamma-1}{\gamma+1}\bar P_-', \ \ x_2\in[0,1].
\end{aligned}
\end{equation}
Substituting \eqref{1-7}  into \eqref{1-11},  the second equation in \eqref{1-8} holds if \eqref{1-7} holds.
\par{ \bf Step 2.}  In the following, we prove the existence of solutions of the form $\left(\bar {\bm U}_\pm(x_2), \bar x_s\right) $. Set
$$d(x_2):=\frac{1}{\bar M_-^2}, \ \ x_2\in[0,1].$$
If $d(x_2)\in(0,1)$ for any $x_2\in[0,1]$, then it follows from \eqref{1-7} that
\begin{equation}\label{1-12}
d'(x_2)=\frac{\beta\gamma}{\bar u_-}\bigg(d-1\bigg)<0, \ \ x_2\in[0,1].
\end{equation}
This,  together with $d(1)<1 $ and $ \displaystyle{\min_{[0,1]} \bar u_->0}$, yields that
\begin{equation}\label{1-14}
d(1)<d(x_2)=1+(d(1)-1)e^{-\int_{x_2}^1\frac{\beta\gamma}{\bar u_-(s)}\de s}<1, \ \ x_2\in[0,1].
\end{equation}
Note that
\begin{equation}\label{1-15}
\bar \rho_-(x_2)=\frac{ \gamma\bar P_-}{d \bar u_-^2},\ \ x_2\in[0,1].
\end{equation}
By the first equation in  \eqref{1-8}, there hold
\begin{equation}\label{1-16}
\bar P_-'(x_2)=- \frac{ \beta\gamma\bar P_-}{d \bar u_-} \ \ {\rm{and}} \ \ \bar P_-(x_2)=  \bar P_-(1)e^{\int_{x_2}^1\frac{\beta\gamma}{(d \bar u_-)(s)}\de s}, \ \ x_2\in[0,1].
\end{equation}
 Substituting \eqref{1-16} into \eqref{1-15}, $\bar \rho_-(x_2)$ is obtained.  Finally,  the
 expression of  $ \bar {\bm U}_+(x_2) $ can be derived from the last three equations in \eqref{1-8}.
 \end{proof}
\subsection{The stability problem }\noindent
\par The two-dimensional finitely long curved  nozzle $ \Omega $  is  described by
\begin{equation}\label{2-1}
\Omega=\big\{(x_1,x_2):0<x_1<L,\ 0<x_2<1+\sigma g(x_1)\big\},
\end{equation}
 where $ \sigma $ is a small positive constant. Furthermore  $g \in C^{3,\alpha}([0,L])$  satisfies
  \begin{equation}\label{2-1-bo}
 g^{(i)}(0)=0, \ \ i=0,\cdots,3.
  \end{equation}
  The entrance,  exit,  and  nozzle walls of $\Omega $ are defined as
\begin{equation*}
\begin{aligned}
&\Gamma_{en}=\{(x_1,x_2):x_1=0,\ 0<x_2<1\}, \ \ \ \
 \Gamma_{ex}=\{(x_1,x_2):x_1=L,\ 0<x_2<1+\sigma g(x_1)\},\\
 &\Gamma_{w}^-=\{(x_1,x_2):0\leq x_1\leq L,\ x_2=0\}, \ \ \ \ \Gamma_{w}^+=\{(x_1,x_2):0\leq x_1\leq L,\ x_2=1+\sigma g(x_1)\}.
\end{aligned}
\end{equation*}
\par Notice that the density and the pressure can be represented as
\begin{equation}\label{2-2}
\begin{cases}
\begin{aligned}
\rho=\rho(S,B,|{\bf u}|^2 )=
\bigg(\frac{\gamma-1}{\gamma e^S}\bigg(B-\frac{1}{2}|{\bf u}|^2\bigg)\bigg)
^{\frac{1}{\gamma-1}},\\
P=P(S,B,|{\bf u}|^2 )=
\bigg(\frac{\gamma-1}{\gamma e^{\frac{S}{\gamma}}}\bigg(B-\frac{1}{2}|{\bf u}|^2 \bigg)\bigg)
^{\frac{\gamma}{\gamma-1}}.\\
\end{aligned}
\end{cases}
\end{equation}
Then in terms of the deformation tensor and  vorticity developed in \cite{WX19},  the steady rotating Euler system \eqref{1-1}  is  equivalent to the following system:
\begin{equation}\label{2-3}
\begin{cases}
\begin{aligned}
 &(c^2(\rho,S)- u_1^2) \p_{x_1} u_1 + (c^2(\rho,S)- u_2^2) \p_{x_2} u_2-u_1u_2(\p_{x_1}u_2+\p_{x_2}u_1)=0,\\
&u_1(\p_{x_1} u_2 -\p_{x_2}u_1+\beta) =\frac{ B-\frac{1}{2}(u_1^2+u_2^2)}{\gamma}\p_{x_2} S-\p_{x_2} B,\\
&u_1 \p_{x_1} S + u_2 \p_{x_2} S=0,\\
&u_1 \p_{x_1} B + u_2 \p_{x_2} B=0.
\end{aligned}
\end{cases}
\end{equation}
\par Let the  shock front be given by   $ x_1=\xi(x_2)$. Denote
\begin{equation*}
\Omega_+=\Omega\cap\{0<x_1<\xi(x_2)\}, \quad \Omega_-=\Omega\cap\{\xi(x_2)<x_1<L\},
\end{equation*}
and set
\begin{equation}\label{2-4}
{\bm V}(x_1,x_2)=
\begin{cases}
   {\bm {V}}_+(x_1,x_2):=(u_{1,+},u_{2,+},S_+,B_+)(x_1,x_2),\ \ {\rm{in}}\ \   \Omega_+,\\
  {\bm {V}}_-(x_1,x_2):=(u_{1,-},u_{2,-},S_-,B_-)(x_1,x_2),\ \ {\rm{in}}\ \   \Omega_-.\\
  \end{cases}
 \end{equation}
   Across the  shock front,  the following Rankine-Hugoniot conditions  are satisfied:
\begin{equation}\label{2-5}
\begin{cases}
\begin{aligned}
&[\rho u_1]-\xi'[\rho u_2]=0,\\
&[\rho u_1^2+P]-\xi'[\rho u_1u_2]=0,\\
&[\rho u_1u_2]-\xi'[\rho u_2^2+P]=0,\\
&[ B]=0.\\
\end{aligned}
\end{cases}
\end{equation}
In particular,  the background transonic background solutions are denoted by
\begin{equation}\label{2-6}
\bar {\bm {V}}_\pm(x_2):=(\bar u_{\pm},0,\bar S_\pm,\bar B_\pm),\ \ x_2\in[0,1],   \\
  \end{equation}
 where
 \begin{equation*}
 \bar S_\pm(x_2)=\ln\bigg(\frac{\bar P_\pm}{\bar \rho_\pm^\gamma}\bigg) \  \ {\rm{and}} \ \  \bar B_\pm(x_2)=\frac{1}{2}\bar u_\pm^2+\frac{\gamma \bar P_\pm}{(\gamma-1)\bar \rho_\pm},\ \ x_2\in[0,1].
 \end{equation*}
 Note that $\bar {\bm {V}}_\pm $ defined on $[0,1]$,  one can extend $\bar {\bm {V}}_\pm$ to a larger domain $ [0,2]$ as follows:
\begin{equation}\label{extension}
\bar {\bm {V}}_\pm^e(y_2)=
\begin{cases}
\bar {\bm {V}}_\pm(y_2),\ \ &y_2\in[0,1],\\
\displaystyle\sum_{k=1}^4 c_k\bar {\bm {V}}_\pm\bigg(1+\frac{1-y_2}{k}\bigg),\ \ &y_2\in[1,2],
\end{cases}
\end{equation}
where the constants $c_k$ ($k=1,2,3,4$) satisfy the  algebraic relations
\begin{equation}\label{vandermon}
\sum_{k=1}^4 c_k= 1,\ \ \ -\sum_{k=1}^4 \frac{c_k}{k} =1,\ \ \ \sum_{k=1}^4 \frac{c_k }{k^2}=1, \ \ -\sum_{k=1}^4 \frac{c_k}{k^3} =1.
\end{equation}
 It is easy to see that $\bar {\bm {V}}_\pm^e$ belong to $C^3$ as long as $\bar {\bm {V}}_\pm\in C^3$. For ease of notations, we still denote these extended functions by $\bar {\bm {V}}_\pm$.
 \par Under the above setting, we now present a detailed formulation of the transonic shock problem.
\begin{problem}\label{problem}Given functions   $(u_{1,en},u_{2,en},S_{en},B_{en},P_{ex})(x_2)$
satisfying
 \begin{equation}\label{2-bo}
 u_{2,en}(0)=u_{2,en}(1)=0,
 \end{equation}
 find a  the transonic shock solution $\bigg( {\bm V}_-,{\bm V}_+,\xi(x_2) \bigg) $ to the steady  rotating Euler system \eqref{2-3} in $\m$ satisfying the following properties.
\begin{enumerate}[ \rm (1)]
\item The location of the shock is given by
\begin{equation}\label{2-7}
\Gamma_s=\big\{(x_1,x_2): x_1=\xi(x_2),\ 0<x_2<x_2^\ast\big\},
\end{equation}
where  $(\xi(x_2^\ast),x_2^\ast)$ is the intersection point of the shock profile with
 the upper nozzle wall $\Gamma_w^+$.
 \item ${\bm {V}}_-$ satisfies the steady  rotating Euler system \eqref{2-3} in the domain $ {\Omega}_- $  with the boundary conditions:
     \begin{eqnarray}\label{2-7-a}
     \begin{cases}
     \begin{aligned}
 &{\bm {V}}_-(0,x_2)=\bar{\bm {V}}_-(x_2)+\sigma {\bm {V}}_{en}(x_2),  \  \ &{\rm{on}} \ \ \Gamma_{en},\\
 &\frac{u_{2,-}}{u_{1,-}}(x_1,1+\sigma g(x_1))=\sigma g'(x_1),\  &{\rm{on}}  \ \ \Gamma_{w,-}^+,\\
 &u_{2,-}(x_1,0)=0,\  &{\rm{on}}  \ \ \Gamma_{w,-}^-,
 \end{aligned}
 \end{cases}
 \end{eqnarray}
 where $\Gamma_{w,-}^\pm=\p\m_-\cap \Gamma_{w}^\pm$ and  ${\bm {V}}_{en}:=(u_{1,en},u_{2,en},S_{en},B_{en})$.

 \item ${\bm {V}}_+$ satisfies  the steady  rotating Euler system \eqref{2-3} in the domain $ {\Omega}_+$,
     the slip boundary conditions on the walls of the nozzle
     \begin{equation}\label{2-9}
     \begin{cases}
      \begin{aligned}
    &\frac{u_{2,+}}{u_{1,+}}(x_1,1+\sigma g(x_1))=\sigma g'(x_1),\ \ &{\rm{on}} \ \ \Gamma_{w,+}^+,\\
    &u_{2,+}(x_1,0)=0,\  &{\rm{on}}  \ \ \Gamma_{w,-}^+,
     \end{aligned}
    \end{cases}
     \end{equation}
     and the exit pressure
\begin{equation}\label{2-10}
P(L,x_2)=\bar P_+(x_2)+\sigma P_{ex}(x_2), \ \ {\rm{on}} \ \ \Gamma_{ex},
\end{equation}
where  $\Gamma_{w,+}^\pm=\p\m_+\cap \Gamma_{w}^\pm$.
\item The solutions ${\bm {V}}_\pm$ satisfy the R-H  condition \eqref{2-5} on the shock front.
    \end{enumerate}
 \end{problem}

 \subsection{The Lagrangian transformation}\noindent
\par Define $  (\tilde y_1,\tilde y_2)=(x_1,\tilde y_2(x_1,x_2)) $ such that
\begin{equation}\label{3-1}
\begin{cases}
\begin{aligned}
&\frac{\p \tilde y_2}{\p x_1}=-\rho_-u_{2,-}, \
\frac{\p \tilde y_2}{\p x_2}=\rho_-u_{1,-}, &{\rm{if}}\
(x_1,x_2)\in \overline{\m_-},\\
&\frac{\p \tilde y_2}{\p x_1}=-\rho_+u_{2,+}, \
\frac{\p \tilde y_2}{\p x_2}=\rho_+u_{1,+}, & {\rm{if}}\
(x_1,x_2)\in \overline{\m_+},\\
&\tilde y_2(0,0)=0,\quad \tilde y_2(L,0)=0.\\
\end{aligned}
\end{cases}
\end{equation}
On the  nozzle walls $ \Gamma_w^\pm $, one  derives
\begin{equation*}
\frac{\de} {\de x_1}\tilde y_2(x_1,0)=0 \quad {\rm{and}}\quad
\frac{\de} {\de x_1}\tilde y_2(x_1,1+\sigma g(x_1))=0.
\end{equation*}
Thus for  any $  x_1\in[0,L] $, we  assume $ \tilde y_2(x_1,0)=0 $. Then one has
\begin{equation*}
\begin{cases}
\tilde y_2(x_1,1+\sigma g(x_1))=m, &x_1\in[0,L_\ast],\\
\tilde y_2(x_1,1+\sigma g(x_1))=m_1, &  x_1\in[L_\ast,L],
\end{cases}
\end{equation*}
where $ m $ and $ m_1 $ are two constants to be determined, and $ ( L_\ast, 1+\sigma g(L_\ast))$ is the intersecting point of the shock front $ \ms $ with the nozzle wall $ \Gamma_w^+ $. Using the first equation in \eqref{1-1}, one obtains
\begin{equation*}
\frac{\de \tilde y_2} {\de x_2}(\xi(x_2)+ 0,x_2)=\frac{\de \tilde y_2} {\de x_2}(\xi(x_2)- 0,x_2).
\end{equation*}
This yields $ m_1=m $, which can be computed as follows:
\begin{equation}\label{3-2}
 m=\int_{0}^{1}(\rho_- u_{1,-})(0,s)\de s=\int_{0}^{1}(\bar \rho_-\bar u_-+\sigma\rho_{en}u_{1,en})(s)\de s>0.
 \end{equation}
Here
$$\rho_{en}(y_2)=
\bigg(\frac{\gamma-1}{\gamma e^{S_{en}}}\bigg(B_{en}-\frac{1}{2}(u_{1,en}^2+u_{2,en}^2\bigg)\bigg)
^{\frac{1}{\gamma-1}},\ \ y_2\in[0,1].$$
\par Define the Lagrangian transformation   as
\begin{equation}\label{3-3}
 y_1=x_1,\ \
y_2=Y_2(x_1,x_2)=\frac{\bar m}{m}\tilde y_2(x_1,x_2),\ \ (x_1,x_2)\in\m,
\end{equation}
where
\begin{equation*}
\bar m=\int_{0}^{1}(\bar \rho_-\bar u_-)(s)\de s>0.
 \end{equation*}
 If $ (\rho_\pm,u_{1,\pm}, u_{2,\pm},P_\pm) $ are close to the background solution
$ (\bar\rho_\pm,\bar u_\pm,0,\bar P_\pm) $, there exists a positive constant $ C_1 $, depending on the background solution, such that
 \begin{equation}\label{2-3-3}
 \frac{\p(y_1,y_2)}{\p(x_1,x_2)}=\frac {\bar m}m \bigg|
\begin{matrix} 1& 0 \\ -\rho u_2 & \rho u_1\end{matrix}\bigg|=\frac {\bar m\rho u_1}m \geq \frac {C_1\bar m}{ m}>0.
\end{equation}
 Hence the  Lagrangian transformation is invertible. Under this  transformation,
 The domain $ \Omega $ becomes
\begin{equation*}
\mn=\{(y_1,y_2): 0<y_1<L,\ 0<y_2 <\bar m\}.
\end{equation*}
 The     nozzles walls $ \Gamma_w^\pm$ are transformed  into
\begin{equation*}
\begin{aligned}
\Sigma_w^-=\{(y_1,y_2): 0\leq y_1\leq L,\ y_2=0\}, \quad
\Sigma_w^+=\{(y_1,y_2): 0\leq y_1\leq L,\ y_2=\bar m\}.
\end{aligned}
\end{equation*}
The entrance and exit  of $ \mn $ are defined as
\begin{equation*}
\begin{aligned}
\Sigma_{en}=\{(y_1,y_2):  y_1=0,\ 0<y_2<\bar m\},\quad
\Sigma_{ex}=\{(y_1,y_2):  y_1=L,\ 0<y_2<\bar m\}.
\end{aligned}
\end{equation*}
\par Denote  the inverse transformation of  \eqref{3-3}  by
 \begin{equation}\label{2-3-3-a}
  x_1=y_1, \ \ x_2=X_2(y_1,y_2), \ \  (y_1,y_2)\in {\mn}.
 \end{equation}
Then it derives from \eqref{3-1} that
\begin{equation*}
\begin{cases}
\begin{aligned}
  &\frac{\p X_2}{\p y_1}=\frac{ u_2}{ u_1},\ \  \frac{\p X_2}{\p y_2}=\frac{ m}{\bar m\rho u_1}, \ \  {\rm{in}} \ \ \mn,\\
  &X_2(y_1,0)=0,\quad \ y_1\in[0,L].
  \end{aligned}
  \end{cases}
  \end{equation*}
  Then $x_2$ can be represented as
\begin{equation}\label{3-7}
x_2=X_2(y_1,y_2)
= \frac m{\bar m}\int_{0}^{y_2}\frac{1}{(\rho u_{1})(y_1,s)}\de s, \ \  (y_1,y_2)\in {\mn}.
\end{equation}
  Define
  \begin{equation*}
  (\e\rho, \e u_1, \e u_2, \e P)(y_1,y_2)= (\rho,  u_1, u_2,  P)(x_1,X_2(y_1,y_2)),\ \ (y_1,y_2)\in\mn.
    \end{equation*}
  Then  the    system \eqref{1-1}   can be rewritten as
\begin{equation}\label{2-7-n}
\begin{cases}
\begin{aligned}
&\p_{y_1}\left(\frac{1}{\e\rho  \e u_1}\right)-\frac{\bar m}{m}\p_{y_2}\left(\frac{ \e u_2}{ \e u_1}\right)
=0,\\
&\p_{y_1}  \bigg(\e u_1+\frac{\e P}{\e\rho \e u_1}\bigg)-\frac{\bar m}{m}\p_{y_2}\bigg(\frac{\e P\e u_2}{\e u_1}\bigg)=0,\\
&\p_{y_1} \e u_2+\frac{\bar m}{m}\p_{y_2}\e P+\beta=0,\\
&\p_{y_1}\e B=0.
\end{aligned}
\end{cases}
\end{equation}
 In the Lagrangian coordinates,  the density $\e\rho $ and the pressure  $\e P$ can be represented as
\begin{equation}\label{3-9}
\begin{cases}
\begin{aligned}
\e\rho=\e\rho(\e S,\e B,|\e {\bf u}|^2 )=
\bigg(\frac{\gamma-1}{\gamma e^{\e S}}\bigg(\e B-\frac{1}{2}|\e{\bf u}|^2\bigg)\bigg)
^{\frac{1}{\gamma-1}},\\
\e P=\e P(\e S,\e B,|\e {\bf u}|^2 )=
\bigg(\frac{\gamma-1}{\gamma e^{\frac{\e S}{\gamma}}}\bigg(\e B-\frac{1}{2}|\e{\bf u}|^2 \bigg)\bigg)
^{\frac{\gamma}{\gamma-1}}.\\
\end{aligned}
\end{cases}
\end{equation}
Substituting \eqref{3-9} into \eqref{2-3}, one gets
\begin{equation}\label{3-12}
\begin{cases}
\begin{aligned}
& (c^2(\e\rho,\e S)
 - \e u_1^2) \bigg(\p_{y_1}-\frac{\bar m}{m}\e\rho \e u_2\p_{y_2}\bigg) \e u_1 + (c^2(\e\rho,\e S)- \e u_2^2) \frac{\bar m}{m}\e\rho \e u_1 \p_{y_2}\e u_2\\
 &=\e u_1 \e u_2\bigg(\p_{y_1}\e u_2-\frac{\bar m}{m}\e\rho \e u_2\p_{y_2}\e u_2+\frac{\bar m}{m}\e \rho \e u_1\p_{y_2}\e u_1\bigg),\\
&\e u_1\bigg(\p_{y_1} \e u_2-\frac{\bar m}{m}\e \rho \e u_2\p_{y_2}\e u_2 -\frac{\bar m}{m}\e\rho \e u_1 \p_{y_2}\e u_1+\beta\bigg)\\
 &=\frac{ (\e B-\frac{1}{2}(\e u_1^2+\e u _2^2)}{\gamma } \e\rho \e u_1 \frac{\bar m}{m}\p_{y_2}\e S-\frac{\bar m}{m}\e\rho \e u_1 \p_{y_2} \e B,\\
& \p_{y_1} \e S=\p_{y_1} \e B=0.
\end{aligned}
\end{cases}
\end{equation}
Set
\begin{equation*}
  \e{\bm {V}}(y_1,y_2):=(\e u_{1},\e u_{2},\e S,\e B), \ \e M_1(y_1,y_2)=\frac{\e u_1}{c(\e\rho,\e S)}, \  \e M_2(y_1,y_2)=\frac{\e u_2}{c(\e\rho,\e S)},\  (y_1,y_2)\in {\mn}.
\end{equation*}
Then it follows from \eqref{3-12}  that
\begin{equation}\label{3-13}
\begin{cases}
\begin{aligned}
& (1
 - \e M_1^2)\p_{y_1}\e u_1-\e M_1\e M_2\p_{y_1}\e u_2-\frac{\bar m}{m}\e\rho \e u_2\p_{y_2} \e u_1+\frac{\bar m}{m}\e\rho \e u_1 \p_{y_2}\e u_2=0,\\
 &\p_{y_1} \e u_2-\frac{\bar m}{m}\e \rho \e u_2\p_{y_2}\e u_2 -\frac{\bar m}{m}\e\rho \e u_1 \p_{y_2}\e u_1+\beta
 =\frac{ e^{\e S}\e\rho^{\gamma}}{\gamma -1}   \frac{\bar m}{m}\p_{y_2}\e S-\frac{\bar m}{m}\e\rho \p_{y_2} \e B,\\
& \p_{y_1} \e S=\p_{y_1} \e B=0.
\end{aligned}
\end{cases}
\end{equation}
Furthermore, the first two equations in \eqref{3-13} can be rewritten as
\begin{equation}\label{3-13-de}
\ma_1 (\e{\bm V})\p_{y_1}\e {\bf u}+\ma_2 (\e{\bm V})\p_{y_2}\e {\bf u}+\ma_3 (\e{\bm V})=0,
\end{equation}
where
\begin{equation*}
\ma_1 (\e{\bm V})=\bigg(\begin{array}{lll}
1- \e M_1^2 &  - \e M_1\e M_2 \\
 0&  \quad 1
\end{array}\bigg),\ \
\ma_2 (\e{\bm V})=\frac{\bar m}{m}\bigg(\begin{array}{lll}
-\e\rho \e u_2 &  \e\rho \e u_1 \\
  -\e\rho \e u_1&   -\e\rho \e u_2
\end{array}\bigg),
\end{equation*}
and $\ma_3 (\e{\bm V})=\bigg(0,-\beta
 +\frac{ \e\rho^{\gamma}}{\gamma -1}   \frac{\bar m}{m}\p_{y_2}\e S-\frac{\bar m}{m}\e\rho \p_{y_2} \e B\bigg)^T $. The two solutions of
 \begin{equation*}
 {\rm det}\left(\ma_1 (\e{\bm V})-\lambda\ma_2 (\e{\bm V})\right)=0
 \end{equation*}
 can be explicitly given by
 \begin{equation}\label{3-13-ch}
 \lambda_\pm(\e{\bm V})=\frac{m}{\bar m}\frac{-\e u_2\pm \e u_1\sqrt{\e M_1^2+\e M_2^2-1}}{\e \rho \e {\bf u}^2}.
 \end{equation}
For supersonic flows, \(\lambda_{\pm}\) are real, which implies that the system \eqref{3-13-de} is hyperbolic. In contrast, for subsonic flows, \(\lambda_{\pm}\) are a pair of conjugate complex numbers, which implies that the system \eqref{3-13-de} is elliptic.
\par The shock front in the Lagrangian coordinates is expressed as $ y_1=\psi(y_2)$. Denote
\begin{equation*}
\mn_+=\mn\cap\{0<y_1<\psi(y_2)\}, \quad \mn_-=\mn\cap\{\psi(y_2)<x_1<L\},
\end{equation*}
and set
\begin{equation}\label{2-4-l}
\e{\bm {V}}_\pm(y_1,y_2):=(\e u_{1,\pm},\e u_{2,\pm},\e S_\pm,\e B_\pm)(y_1,y_2),\ \ {\rm{in}}\ \   \mn_\pm.\\
  \end{equation}
 The background transonic background solutions in the  Lagrangian coordinates are
\begin{equation}\label{2-6-l}
(\h \rho_{\pm},\h u_{\pm},0,\h P_\pm)(y_2)=(\bar \rho_{\pm},\bar u_\pm,0,\bar P_\pm)
 ( \bar X_2(y_2)), \ \  y_2\in[0,\bar  m],   \\
  \end{equation}
  where
 \begin{equation*}
 \begin{aligned}
  \bar X_2(y_2)=\int_{0}^{y_2}\frac{1}{(\bar \rho_-\bar u_-)(s)}\de s=\int_{0}^{y_2}\frac{1}{(\bar \rho_+\bar u_+)(s)}\de s, \ \  y_2\in[0,\bar m].\\
\end{aligned}
 \end{equation*}
 Set
 \begin{equation}\label{2-6-l-b}
 \h {\bm {V}}_\pm(y_2):=(\h u_\pm,0,\h S_\pm,\h B_\pm), \ \ y_2\in [0,\bar m],
 \end{equation}
 where
 \begin{equation*}
 \h S_\pm(y_2)=\ln\bigg(\frac{\h P_\pm}{\h \rho_\pm^\gamma}\bigg) \  \ {\rm{and}} \ \  \h B_\pm(y_2)=\frac{1}{2}\h u_\pm^2+\frac{\gamma \h P_\pm}{(\gamma-1)\h \rho_\pm},\ \ y_2\in [0,\bar m].
 \end{equation*}
  The Rankine-Hugoniot conditions in \eqref{2-5} becomes
\begin{equation}\label{3-8-r}
\begin{cases}
\begin{aligned}
&\left[\frac{1}{\e\rho \e u_1}\right]+\frac{\bar m\psi'}m\left[\frac{ \e u_2}{  \e u_1}\right]
=0,\\
&  \bigg[\e u_1+\frac{\e P}{\e\rho \e u_1}\bigg]+\frac{\bar m\psi'}{m}\bigg[\frac{\e P\e u_2}{\e u_1}\bigg]=0,\\
&[\e u_2]-\frac{\bar m\psi'}m[\e P]=0,\\
&[\e B]=0.
\end{aligned}
\end{cases}
\end{equation}
One can rewrite the third  equation in \eqref{3-8-r} as
\begin{equation}\label{3-10}
G_0(\e{\bm V}_-,\e{\bm V}_+)=[\e u_2]-\frac{\bar m\psi'}m[\e P]=0.
\end{equation}
Using  \eqref{3-10}, we can eliminate the quantity $ \psi^{\prime} $ in the first two equations of \eqref{3-8-r} to obtain
\begin{equation}\label{3-11}
\begin{cases}
\begin{aligned}
&G_1(\e{\bm V}_-,\e{\bm V}_+)=\left[\frac{1}{\e\rho \e u_1}\right]+\frac{[\e u_2]}{[\e P]}\left[\frac{ \e u_2}{  \e u_1}\right]
=0,\\
 &G_2(\e{\bm V}_-,\e{\bm V}_+)=\bigg[\e u_1+\frac{\e P}{\e\rho \e u_1}\bigg]+\frac{[\e u_2]}{[\e P]}\bigg[\frac{\e P\e u_2}{\e u_1}\bigg]=0.\\
\end{aligned}
\end{cases}
\end{equation}
Then the transonic shock problem \eqref{problem} in the  Lagrangian coordinates can be reformulated as follows.
\begin{problem}\label{problem-1} Given functions $(u_{1,en},u_{2,en},S_{en},B_{en},P_{ex}) (x_2) $  satisfying \eqref{2-bo}, find a   transonic shock solution $\bigg( \e{\bm V}_-,\e{\bm V}_+,\psi(y_2) \bigg) $ to the steady  rotating Euler system \eqref{3-13} in $\mn$ satisfying the following properties.
\begin{enumerate}[ \rm (1)]
\item The location of the shock is given by
\begin{equation}\label{2-7-l}
\Sigma_s=\big\{(y_1,y_2): y_1=\psi(y_2),\ 0<y_2<\bar m\big\}.
\end{equation}
\item $\e{\bm {V}}_-$ satisfies the steady  rotating Euler system \eqref{3-13} in the domain $ {\mn}_- $  with the boundary conditions:
     \begin{eqnarray}\label{2-7-a-l}
     \begin{cases}
     \begin{aligned}
 &\e{\bm {V}}_-(0,y_2)=\h{\bm {V}}_-(y_2)+\sigma \e{\bm {V}}_{en}(y_2),  \  \ &{\rm{on}} \ \ \Sigma_{en},\\
 &\frac{\e u_{2,-}}{\e u_{1,-}}(y_1,\bar m)=\sigma g'(y_1),\  &{\rm{on}}  \ \ \Sigma_{w,-}^+,\\
 &\e u_{2,-}(y_1,0)=0,\  &{\rm{on}}  \ \ \Sigma_{w,-}^-,
 \end{aligned}
 \end{cases}
 \end{eqnarray}
 where $\Sigma_{w,-}^\pm=\p\mn_-\cap \Sigma_{w}^\pm$ and
 \begin{eqnarray*}
 \begin{aligned}
  \e{\bm {V}}_{en}(y_2):&=(\e u_{1,en},\e u_{2,en},\e S_{en},\e B_{en})(y_2)\\
  &=
  ( u_{1,en}, u_{2,en}, S_{en}, B_{en})\bigg(\frac m{\bar m}\int_{0}^{y_2}\frac{1}{(\bar \rho_-\bar u_-+\sigma\rho_{en}u_{1,en})(0,s)}\de s\bigg),\ \ y_2\in [0,\bar m].
  \end{aligned}
  \end{eqnarray*}
 \item $\e{\bm {V}}_+$ satisfies  the steady  rotating Euler system \eqref{3-13} in the domain $ {\mn}_+$,
     the slip boundary conditions on the walls of the nozzle
     \begin{equation}\label{2-9-l}
     \begin{cases}
      \begin{aligned}
    &\frac{\e u_{2,+}}{\e u_{1,+}}(y_1,\bar m)=\sigma g'(y_1),\ \ &{\rm{on}} \ \ \Sigma_{w,+}^+,\\
    &\e u_{2,+}(y_1,0)=0,\  &{\rm{on}}  \ \ \Sigma_{w,-}^+,
     \end{aligned}
    \end{cases}
     \end{equation}
     and the exit pressure
\begin{equation}\label{2-10-l}
\e P_+(L,y_2)=\h P_+(y_2)+\sigma \e P_{ex}(y_2), \ \ {\rm{on}} \ \ \Gamma_{ex},
\end{equation}
where  $\Gamma_{w,+}^\pm=\p\m_+\cap \Gamma_{w}^\pm$ and
\begin{eqnarray*}
  \e P_{ex}(y_2)=P_{ex}\bigg(\frac m{\bar m}\int_{0}^{y_2}\frac{1}{(\rho_+ u_{1,+})(L,s)}\de s\bigg),\ \ y_2\in [0,\bar m].
  \end{eqnarray*}
\item The solutions $\e{\bm {V}}_\pm$ satisfy the R-H  condition \eqref{3-8-r} on the shock front.
    \end{enumerate}
 \end{problem}
 \subsection{The main result}\noindent
 \par Before we state the  main result,  some weighted H\"{o}lder spaces and their norms in the rectangle $ \mn $ are first introduced. For $ \mathbf{y}=(y_1,y_2) $ and $ \h{\mathbf{y}}=(\h y_1,\h y_2)\in \mn $, set
\begin{equation*}
 \delta_{\mathbf{y}}:=\min(y_2,\bar m-y_2)\quad {\rm{and}} \quad
  \delta_{\textbf{y},\h{\textbf{y}}}:
  =\min(\delta_{\mathbf{y}},\delta_{\h{\mathbf{y}}}).
  \end{equation*}
 For any positive integer $ k $, $ \alpha\in(0,1) $ and $ \kappa\in \mathbb{R} $  define the standard H\"{o}lder norms by
    \begin{equation*}
    \begin{aligned}
     &\|u\|_{k;\mn}:=\sum_{0\leq|\upsilon|\leq k}\sup_{\mathbf{y}\in \mn}|D^{\upsilon}u( {\mathbf{y}})|; \\
    & [u]_{ k,\alpha;\mn}:=\sum_{|\upsilon|=k}\sup_{\mathbf{y},\h{\mathbf{y}}\in \mn,\mathbf{y}\neq \h {\mathbf{y}}}
  \frac{|D^{\upsilon}u(\mathbf{y})-D^{\upsilon}u(\h {\mathbf{y}})|}{|\mathbf{y}-\h {\mathbf{y}}|^{\alpha}};\\
  &\|u\|_{ k,\alpha;\mn}:=\|u\|_{k;\mn}+ [u]_{ k,\alpha;\mn},
  \end{aligned}
  \end{equation*}
   where $D^{\upsilon}$ denotes $\p_{y_1}^{\upsilon_1}\p_{y_2}^{\upsilon_2} $ for a multi-index $\upsilon= (\upsilon_1,\upsilon_2)$ with $\upsilon_1,\upsilon_2\in \mathbb{Z}_+$ and $|\upsilon|=\sum_{j=1}^2\upsilon_j$.
  Define the weighted H\"{o}lder norms by
   \begin{equation*}
  \begin{aligned}
  \|u\|_{k,0;\mn}^{(\kappa)}\ &:=\sum_{0\leq|\upsilon|\leq k} \delta_{\mathbf{y}}^{\max(|\upsilon|+\kappa,0)}|D^{\upsilon}u(\textbf{y})|, \\
  [u]_{k,\alpha;\mn}^{(\kappa)}\ &:=\sum_{|\upsilon|=k}\delta_{\mathbf{y},\h{\mathbf{y}}}^{\max(k+\alpha+\kappa,0)}
  \frac{|D^{\upsilon}u(\textbf{y})-D^{\upsilon}u(\h{\textbf{y}})|}
  {|\textbf{y}-\h{\textbf{y}} |^{\alpha}};\\
  \|u\|_{k,\alpha;\mn}^{(\kappa)}\ &:=\|u\|_{k,0;\mn}^{(\kappa)}
  +[u]_{m,\alpha;\mn}^{(\kappa)}.\\
  \end{aligned}
  \end{equation*}
  $ C_{k,\alpha}^{(\kappa)}(\mn) $ denotes the completion of the set of all smooth functions whose $\|\cdot\|_{k,\alpha;\mn}^{(\kappa)}$ norms are finite.
  Furthermore, for a vector function $ {\bf u}=(u_1,u_2,\cdots,u_n) $, define
\begin{equation*}
   \|{\bf u}\|_{k,\alpha;\mn}^{(\kappa)}:=\sum_{i=1}^{n}\| u_i\|_{k,\alpha;\mn}^{(\kappa)}.
\end{equation*}
  If $u $ is a function defined on
 $(0,\bar m) $,
the weighted  H\"{o}lder norm $ \|{ u}\|_{k,\alpha;(0,\bar m)}^{(\kappa )} $  can be defined similarly.
\par The main theorem of this paper can be stated as follows.
\begin{theorem}\label{th1}
 Fix $ \alpha\in(\frac12,1) $. For given  functions $ g\in C^{3,\alpha}([0,L])$ satisfying \eqref{2-1-bo}, $(u_{1,en},u_{2,en},S_{en},B_{en},$\\$P_{ex})\in \left(C^{2,\alpha}([0,1])\right)^4\times C^{2,\alpha}([0,1+\sigma g(L)])$  satisfying \eqref{2-bo}
 and either the condition \eqref{5-19} or \eqref{5-19-le} in Lemma \ref{le6},  there exist  positive constants $\sigma_{0} $ and $\mc_0$  depending only on the background solution and the boundary data such that for any $0<\sigma<\sigma_0$, Problem \ref {problem-1} admits a transonic shock solution $\left( \e{\bm V}_-,\e{\bm V}_+,\psi(y_2) \right) $, which satisfies
 \begin{equation}\label{3-27}
  \| \e{\bm V}_--\h{\bm V}_-\|_{2,\alpha;\overline{\mn_-}}
    +\|\e{\bm V}_+-\h{\bm V}_+\|_{1,\alpha;\mn^-}^{(-\alpha)}+ \|\psi' \|_{1,\alpha;(0,\bar m)}^{(-\alpha)}+|\psi(\bar m)-\bar \psi|
   \leq \mc_0\sigma.
   \end{equation}
   Here $ \bar \psi$ is the approximate shock position obtained in Lemma \ref{le6}.
\end{theorem}

\section{Solving a first order linear elliptic
 system  with variable coefficients}\noindent
 \par  To establish the existence of the transonic  shock solution, one of the key steps is to solve the boundary value problem of the linearized rotating  Euler system for  subsonic flows behind the shock front, which is hyperbolic-elliptic mixed. To this end,  we aim to establish  the existence theorem to such problem in this section, which will later be applied in the proof of  Theorem \ref{th1}.
\par  Let $ L_1 $  and $L_2$ be two positive constants and
\begin{equation}\label{4-1}
\md=\{(y_1,y_2): L_1<y_1<L_2,\ 0<y_2 <\bar m\}
\end{equation}
be a rectangle with the boundaries
\begin{equation*}
\begin{aligned}
&\Gamma_{1}=\{(y_1,y_2):y_1=L_1,\ 0<y_2<\bar m\}, \ \ \ \
 \Gamma_{2}=\{(y_1,y_2):y_1=L_2,\ 0<y_2<\bar m\},\\
 &\Gamma_{3}=\{(y_1,y_2):L_1\leq x_1\leq L_2,\ y_2=0\}, \ \ \ \Gamma_{4}=\{(y_1,y_2):L_1\leq x_1\leq L_2,\ y_2=\bar m\}.
\end{aligned}
\end{equation*}
Consider the following boundary value problem:
\begin{equation}\label{4-2}
\begin{cases}
\p_{y_1}(\lambda_1(y_2)v_1)+\p_{y_2}(\lambda_2(y_2)v_2)=H_1,\ &{\rm{in}} \ \md,\\
\p_{y_1}(\lambda_3(y_2)v_2)-\p_{y_2}(\lambda_4(y_2)v_1)=H_2,\ &{\rm{in}} \ \md,\\
v_1(L_1,y_2)=h_1(y_2), \ &{\rm{on}} \ \Gamma_{1},\\
v_1(L_2,y_2)=h_2(y_2), \ &{\rm{on}} \ \Gamma_{2},\\
v_2(y_1,0)=0, \ &{\rm{on}} \ \Gamma_{3},\\
v_2(y_1,\bar m)=h_3(y_1), \ &{\rm{on}} \ \Gamma_{4},\\
\end{cases}
\end{equation}
where $\lambda_i\in C^3([0,\bar m])$ satisfy $\lambda_i>0 $  for $ i=1,2,3,4$.
\begin{theorem}\label{th2}
Fix $ \alpha\in(\frac12,1) $. For given  functions $(H_1,H_2)\in \left(C_{0,\alpha}^{(1-\alpha)}(\md)\right)^2$,  $(h_1,h_2)\in \left(C_{1,\alpha}^{(-\alpha)}(0,\bar m)\right)^2 $ and $ h_3\in C^{0,\alpha}([L_1,L_2])$, there exists a unique solution $(v_1,v_2)$ to the boundary value problem \eqref{4-2} if and only if
\begin{equation}\label{4-3}
\int_0^{\bar m}(\lambda_1h_2-\lambda_1h_1)(y_2)\de y_2+\int_{L_1}^{L_2}\lambda_2(\bar m)h_3(y_1)\de y_1=\iint_\md H_1 \de y_1 \de y_2.
\end{equation}
Furthermore, the solution $(v_1,v_2)$ satisfies the following estimate:
\begin{equation}\label{4-4}
\|(v_1,v_2)\|_{1,\alpha;\md}^{(-\alpha)}\leq \mc\bigg(\|(H_1,H_2)\|_{0,\alpha;\md}^{(1-\alpha)}
+\|(h_1,h_2)\|_{1,\alpha;(0,\bar m)}^{(-\alpha)}+\|h_3\|_{0,\alpha;[L_1,L_2]}\bigg),
\end{equation}
where $\mc>0 $ is a positive constant depending only on $(\lambda_1,\lambda_2,\lambda_3,\lambda_4,L_1,L_2,\bar m)$.
\end{theorem}
\begin{proof}
 The proof  is divided into three steps.
  \par  { \bf Step 1}:
We decompose the problem \eqref{4-2} into two boundary value problems with different inhomogeneous terms as follows. Let $(v_1,v_2)=(\hat v_1,\hat v_2)+(\check v_1,\check v_2)$, where $ (\hat v_1,\hat v_2) $ satisfies
\begin{equation}\label{4-5}
\begin{cases}
\p_{y_1}(\lambda_1(y_2)\hat v_1)+\p_{y_2}(\lambda_2(y_2)\hat v_2)=H_1,\ &{\rm{in}} \ \md,\\
\p_{y_1}(\lambda_3(y_2)\hat v_2)-\p_{y_2}(\lambda_4(y_2)\hat v_1)=0,\ &{\rm{in}} \ \md,\\
\hat v_1(L_1,y_2)=h_1(y_2), \ &{\rm{on}} \ \Gamma_{1},\\
\hat v_1(L_2,y_2)=h_2(y_2), \ &{\rm{on}} \ \Gamma_{2},\\
\hat v_2(y_1,0)=0, \ &{\rm{on}} \ \Gamma_{3},\\
\hat v_2(y_1,\bar m)=h_3(y_1), \ &{\rm{on}} \ \Gamma_{4},\\
\end{cases}
\end{equation}
and
$ (\check v_1,\check v_2) $ satisfies
\begin{equation}\label{4-6}
\begin{cases}
\p_{y_1}(\lambda_1(y_2)\check v_1)+\p_{y_2}(\lambda_2(y_2)\check v_2)=0,\ &{\rm{in}} \ \md,\\
\p_{y_1}(\lambda_3(y_2)\check v_2)-\p_{y_2}(\lambda_4(y_2)\check v_1)=H_2,\ &{\rm{in}} \ \md,\\
\check v_1(L_1,y_2)=0, \ &{\rm{on}} \ \Gamma_{1},\\
\check v_1(L_2,y_2)=0, \ &{\rm{on}} \ \Gamma_{2},\\
\check v_2(y_1,0)=0, \ &{\rm{on}} \ \Gamma_{3},\\
\check v_2(y_1,\bar m)=0, \ &{\rm{on}} \ \Gamma_{4}.\\
\end{cases}
\end{equation}
\par  { \bf Step 2}: In this step, we are going to solve \eqref{4-5}. The second equation in \eqref{4-5} implies that there exists a potential function $ \h \phi $ so that
 \begin{equation}\label{4-7}
 \p_{y_2}\h \phi=\lambda_3(y_2)\hat v_2, \ \ \p_{y_1}\h \phi=\lambda_4(y_2)\hat v_1, \ \ \h \phi(L_1,0)=0.
 \end{equation}
  Substituting \eqref{4-7} into the first equation in \eqref{4-5} yields that
  \begin{equation}\label{4-8}
\begin{cases}
\begin{aligned}
&\p_{y_1}\bigg(\frac{\lambda_1(y_2)}{\lambda_4(y_2)}\p_{y_1}\h \phi\bigg)
+\p_{y_2}\bigg(\frac{\lambda_2(y_2)}{\lambda_3(y_2)}\p_{y_2}\h\phi\bigg)=H_1,\ &{\rm{in}} \ \md,\\
&\p_{y_1}\h \phi(L_1,y_2)=(\lambda_4h_1)(y_2), \ &{\rm{on}} \ \Gamma_{1},\\
&\p_{y_1}\h \phi(L_2,y_2)=(\lambda_4h_2)(y_2), \ &{\rm{on}} \ \Gamma_{2},\\
&\p_{y_2}\h \phi(y_1,0)=0, \ &{\rm{on}} \ \Gamma_{3},\\
&\p_{y_2}\h \phi(y_1,\bar m)=\lambda_3(\bar m)h_3(y_1), \ &{\rm{on}} \ \Gamma_{4}.\\
\end{aligned}
\end{cases}
\end{equation}
 Define the following function space
\begin{eqnarray}\nonumber
\mathbb{H}=\left\{\vartheta\in H^1(\md):\iint_\md \vartheta(y_1,y_2) \de y_1 \de y_2=0
\right\}.
\end{eqnarray}
We have the following weak formulation for \eqref{4-8}: for any $\vartheta\in \mathbb{H}$, there holds
\begin{eqnarray}\label{4-9}
\mathcal{I}(\h \phi,\vartheta)=\mb(\vartheta),
\end{eqnarray}
where
\begin{equation*}
 \begin{aligned}
 \mathcal{I}(\h \phi,\vartheta)&=\iint_{\md}
 \bigg(\frac{\lambda_1(y_2)}{\lambda_4(y_2)}\p_{y_1}\h \phi\p_{y_1}\vartheta
 +\frac{\lambda_2(y_2)}{\lambda_3(y_2)}\p_{y_2}\h \phi\p_{y_2}\vartheta\bigg)\de y_1 \de y_2,\\
  \mb(\vartheta)&=\int_0^{\bar m}\bigg((\lambda_1h_2)(y_2)\vartheta(L_2,y_2)-(\lambda_1h_1)(y_2)
  \vartheta(L_1,y_2)\bigg)\de y_2\\
 &\quad +\int_{L_1}^{L_2}\lambda_2(\bar m)h_3(y_1) \vartheta(y_1,\bar m)\de y_1-\iint_\md H_1\vartheta \de y_1 \de y_2.
\end{aligned}
\end{equation*}
Note that for  $ \alpha\in(\frac12,1) $, one has
 \begin{equation*}
 \begin{aligned}
\iint_{\md}|H_1|^2\de y_1\de y_2
\leq \bigg(\|H_1\|_{0,\alpha;\md}^{(1-\alpha}\bigg)^2
\int_{0}^{\bar m}\int_{L_1}^{L_2}
\delta_{\mathbf{y}}^{2\alpha-2}\de y_1\de y_2
\leq C\bigg(\|H_1\|_{0,\alpha;\mn}^{(1-\alpha)}\bigg)^2.
\end{aligned}
\end{equation*}
 Then it follows from  the H\"{o}lder inequality and the trace theorem that
 \begin{equation}\label{4-f}
 |\mb(\vartheta)|
 \leq C\bigg(\|H_1\|_{0,\alpha;\md}^{(1-\alpha)}
+\|(h_1,h_2)\|_{1,\alpha;(0,\bar m)}^{(-\alpha)}+\|h_3\|_{0,\alpha;[L_1,L_2]}\bigg)\|\vartheta\|_{H^1(\md)}.
  \end{equation}
  Thus the linear functional $\mb(\vartheta)$ on $H^1(\md)$ is continuous. Next, we verify that the bilinear operator $\mathcal{I}$ is bounded  and  coercive in $\mathbb{H}\times \mathbb{H}$. Note that $ \lambda_i >0$ for $ i=1,2,3,4$. Then it follows from the  H\"{o}lder inequality and the Poincar\'{e} inequality that
  \begin{equation*}
  \begin{aligned}
  |\mathcal{I}(\h \phi,\vartheta)|&=\left|\iint_{\md}
 \bigg(\frac{\lambda_1(y_2)}{\lambda_4(y_2)}\p_{y_1}\h \phi\p_{y_1}\vartheta
 +\frac{\lambda_2(y_2)}{\lambda_3(y_2)}\p_{y_2}\h \phi\p_{y_2}\vartheta\bigg)\de y_1 \de y_2 \right|\\
 &\leq \bigg(\bigg\|\frac{\lambda_1}{\lambda_4}\bigg\|_{C^{0}[0,\bar m]}+\bigg\|\frac{\lambda_2}{\lambda_3}\bigg\|_{C^{0}[0,\bar m]}\bigg)\h \phi\|_{H^1(\md)}\|\vartheta\|_{H^1(\md)},\\
  \mathcal{I}(\h \phi,\h \phi)&=\iint_{\md}\bigg(\frac{\lambda_1(y_2)}{\lambda_4(y_2)}
  (\p_{y_1}\h \phi)^2+
   \frac{\lambda_2(y_2)}{\lambda_3(y_2)}(\p_{y_2}\h \phi)^2
  \bigg)\de y_1\de y_2 \\
  &\geq  \min\bigg\{\min_{[0,\bar m]}\frac{\lambda_1}{\lambda_4},\min_{[0,\bar m]}\frac{\lambda_2}{\lambda_3}\bigg\}\|\n\h \phi\|_{L^2(\md)}^2
  \geq  C\min\bigg\{\min_{[0,\bar m]}\frac{\lambda_1}{\lambda_4},\min_{[0,\bar m]}\frac{\lambda_2}{\lambda_3}\bigg\}\| \h \phi\|_{H^1(\md)}^2.
  \end{aligned}
  \end{equation*}
   By the Lax-Milgram theorem, there exists a unique weak solution $\h \phi\in\mathbb{H}$ to \eqref{4-8}. Furthermore, the solution satisfies
\begin{equation}\label{4-j}
\|\h\phi\|_{H^1(\md)}\leq C\bigg(\|H_1\|_{0,\alpha;\md}^{(1-\alpha)}
+\|(h_1,h_2)\|_{1,\alpha;(0,\bar m)}^{(-\alpha)}+\|h_3\|_{0,\alpha;[L_1,L_2]}\bigg).
    \end{equation}
    \par In the following, we  improve the regularity of the weak solution $ \h \phi $ to \eqref{4-8}. For $ \mathbf{y}_0=(y_{1,0},y_{2,0})  \in \overline{\md}$ and $ \nu\in \mathbb{R } $ with $ 0<\nu<\frac{1}{10}\min\{1,L_2-L_1,\bar m\} $, set
 \begin{equation*}
  \begin{aligned}
  &B_{\nu}(\mathbf{y}_0):=\{\mathbf{y}=(y_1,y_2)\in \mathbb{R}^2:
  |\mathbf{y}_0-\mathbf{y}|<\nu\}, \quad D_{\nu}(\mathbf{y}_0):= B_{\nu}(\mathbf{y}_0)\cap \md.\\
  \end{aligned}
 \end{equation*}
 Note that there exists  a constant $ \varsigma\in(0,1/10) $ such that
 \begin{equation*}
 \varsigma\leq \frac{|D_{\nu}(\mathbf{y}_0)|}{|B_{\nu}(\mathbf{y}_0)|}
 \leq \frac{1}{\varsigma}.
 \end{equation*}
 Hence we  will follow the proof of \cite[Theorem 3.8]{HL11} to obtain
 \begin{equation}\label{4-16}
\begin{aligned}
\iint_{D_{\nu}(\mathbf{y}_0)}|\n\h\phi
|^2  \de y_1 \de y_2
\leq C\nu^{2\alpha} \bigg(\|H_1\|_{0,\alpha;\md}^{(1-\alpha)}
+\|(h_1,h_2)\|_{1,\alpha;(0,\bar m)}^{(-\alpha)}+\|h_3\|_{0,\alpha;[L_1,L_2]}\bigg)^2
\end{aligned}
\end{equation}
for any $ \mathbf{y}_0  \in \overline{\md} $. Once \eqref{4-16} is derived, one can apply \cite[Theorem 3.1]{HL11} to get
   \begin{equation}\label{4-17}
  \|\h\phi\|_{0,\alpha;\overline{\md}}\leq  C\bigg(\|H_1\|_{0,\alpha;\md}^{(1-\alpha)}
+\|(h_1,h_2)\|_{1,\alpha;(0,\bar m)}^{(-\alpha)}+\|h_3\|_{0,\alpha;[L_1,L_2]}\bigg).
  \end{equation}
  \par Now, we prove \eqref{4-17} only for the case $ \mathbf{y}_0  \in \Gamma_{1}\cap \Gamma_{3}$,  since the other cases can be treated in the same way.
   Fix $ \mathbf{y}_0=(L_1,0)\in \Gamma_{1}\cap \Gamma_{3} $ and $ \mu\in \mathbb{R} $ with $ 0<\mu\leq\frac{1}{10}\min\{1,L_2-L_1,\bar m\} $. Let $ \h\phi_0 $ be a weak solution of the following problem:
  \begin{equation}\label{4-18}
\begin{cases}
\begin{aligned}
 &\p_{y_1}\bigg(\frac{\lambda_1}{\lambda_4}(0)\p_{y_1}\h \phi_0\bigg)
+\p_{y_2}\bigg(\frac{\lambda_2}{\lambda_3}(0)\p_{y_2}\h\phi_0\bigg)
  =0,  &{\rm{in}} \ \ D_{\mu}(\mathbf{y}_0),\\
 &\p_{y_1}\h\phi_0(L_0,y_2)=0, & {\rm{on}} \ \  \p D_{\mu}(\mathbf{y}_0)\cap\Gamma_{1},\\
 &\p_{y_2}\h\phi_0(y_1,0)=0,&{\rm{on}} \ \ \p D_{\mu}(\mathbf{y}_0)\cap \Gamma_{3},\\
  &\h\phi_0(y_1,y_2)=\h\phi, &{\rm{on}} \ \ \p D_{\mu}(\mathbf{y}_0)\cap \md.
  \end{aligned}
  \end{cases}
\end{equation}
Then $ \h\phi_h=\h\phi-\h\phi_0 $ satisfies
\begin{equation}\label{4-19}
\begin{aligned}
&\iint_{ D_{\mu}(\mathbf{y}_0)}\bigg(\frac{\lambda_1}{\lambda_4}(0)\p_{y_1}
\h\phi_h\p_{y_1}\vartheta+\frac{\lambda_2}{\lambda_3}(0)\p_{y_2}
\h\phi_h\p_{y_2}\vartheta\bigg)\de y_1 \de y_2
  =I_1+I_2+I_3,
\end{aligned}
\end{equation}
 where $\vartheta\in H^1(D_{\mu}(\mathbf{y}_0))$ with $\vartheta=0 \ {\rm{on}}\  \p D_{\mu}(\mathbf{y}_0) \cap\md $ and
\begin{equation*}
\begin{aligned}
&I_1=-\iint_{D_{\mu}(\mathbf{y}_0)}\bigg(\bigg(\frac{\lambda_1}{\lambda_4}(y_2)
-\frac{\lambda_1}{\lambda_4}(0)\bigg)\p_{y_1}
\h\phi\p_{y_1}\vartheta+\bigg(\frac{\lambda_2}{\lambda_3}(y_2)
-\frac{\lambda_2}{\lambda_3}(0)\bigg)\p_{y_2}
\h\phi\p_{y_2}\vartheta\bigg)\de y_1 \de y_2,\\
&I_2=-\int_{ \p D_{\mu}(\mathbf{y}_0)\cap\Gamma_{1}}(\lambda_1h_1)(y_2)
  \vartheta(L_0,y_2)\de y_2,\\
 &I_3= -\iint_{D_{\mu}(\mathbf{y}_0)}(H_1\vartheta)(y_1,y_2) \de y_1 \de y_2.
 \end{aligned}
\end{equation*}
 By using  the  H\"{o}lder inequality, the Poincar\'{e} inequality and the trace theorem, one obtains
\begin{equation}\label{4-20}
\begin{cases}
\begin{aligned}
&I_1\leq C\mu \bigg(\iint_{D_{\mu}(\mathbf{y}_0)}
|\n\h\phi|^2\de y_1 \de y_2\bigg)^{\frac12}\bigg(\iint_{D_{\mu}(\mathbf{y}_0)}
|\n\vartheta|^2\de y_1 \de y_2\bigg)^{\frac12},\\
&I_2\leq C\mu\|h_1\|_{1,\alpha;(0,\bar m)}^{(-\alpha)}\bigg(\iint_{D_{\mu}(\mathbf{y}_0)}
|\n\vartheta|^2\de y_1 \de y_2\bigg)^{\frac12},\\
&I_3\leq C\mu\bigg(\bigg(\|H_1\|_{0,\alpha;\md}^{(1-\alpha)}\bigg)^2
\mu^{2\alpha}\bigg)^{\frac12}\bigg(\iint_{D_{\mu}(\mathbf{y}_0)}
|\n\vartheta|^2\de y_1 \de y_2\bigg)^{\frac12}.
\end{aligned}
\end{cases}
\end{equation}
Taking the test function $ \vartheta=\h \phi_h $ to \eqref{4-19} and combining \eqref{4-20} yield
\begin{equation}\label{4-22}
\begin{aligned}
\iint_{D_{\mu}(\mathbf{y}_0)}|\n\h\phi_h
|^2  \de y_1 \de y_2
&\leq C\mu^{2}\iint_{D_{\mu}(\mathbf{y}_0)}
|\n\h\phi|^2\de y_1 \de y_2\\
&\quad+C\mu^{2\alpha}\bigg(\|H_1\|_{0,\alpha;\md}^{(1-\alpha)}
+\|(h_1,h_2)\|_{1,\alpha;(0,\bar m)}^{(-\alpha)}+\|h_3\|_{0,\alpha;[L_1,L_2]}\bigg)^2.
  \end{aligned}
  \end{equation}
   By \cite[Corollary 3.11]{HL11}, for $0<\nu<\mu\leq \frac{1}{10}\min\{1,L_2-L_1,\bar m\} $, it holds that
   \begin{equation}\label{4-26}
  \begin{aligned}
  &\iint_{D_{\nu}(\mathbf{y}_0)}|\n\h\phi
|^2  \de y_1 \de y_2
\leq C\bigg(\frac{\nu}{\mu}\bigg)^2\iint_{D_{\mu}(\mathbf{y}_0)}|\n\h\phi
|^2\de y_1 \de y_2\\
&\quad+C\mu^{2}\iint_{D_{\mu}(\mathbf{y}_0)}
|\n\h\phi|^2\de y_1 \de y_2+C\mu^{2\alpha}\bigg(\|H_1\|_{0,\alpha;\md}^{(1-\alpha)}
+\|(h_1,h_2)\|_{1,\alpha;(0,\bar m)}^{(-\alpha)}+\|h_3\|_{0,\alpha;[L_1,L_2]}\bigg)^2.\\
\end{aligned}
  \end{equation}
Then it follows from \eqref{4-j} and \eqref{4-26} that
  \begin{equation}\label{4-m}
  \begin{aligned}
  &\iint_{D_{\nu}(\mathbf{y}_0)}|\n\h\phi
|^2  \de y_1 \de y_2
\leq C\bigg(\frac{\nu}{\mu}\bigg)^2\iint_{D_{\mu}(\mathbf{y}_0)}|\n\h\phi
|^2\de y_1 \de y_2\\
&\quad+C\mu^{2\alpha}\bigg(\|H_1\|_{0,\alpha;\md}^{(1-\alpha)}
+\|(h_1,h_2)\|_{1,\alpha;(0,\bar m)}^{(-\alpha)}+\|h_3\|_{0,\alpha;[L_1,L_2]}\bigg)^2.\\
\end{aligned}
  \end{equation}
   By \cite[Lemma 3.4]{HL11}, there holds
\begin{equation}\label{4-28}
  \begin{aligned}
  \iint_{D_{\nu}(\mathbf{y}_0)}|\n\h\phi
|^2  \de y_1 \de y_2
\leq  C\nu^{2\alpha}\bigg(\|H_1\|_{0,\alpha;\md}^{(1-\alpha)}
+\|(h_1,h_2)\|_{1,\alpha;(0,\bar m)}^{(-\alpha)}+\|h_3\|_{0,\alpha;[L_1,L_2]}\bigg)^2. \\
\end{aligned}
  \end{equation}
  Hence the proof of \eqref{4-16} is completed. With the $C^{0,\alpha}$ estimate, one can  apply the Schauder estimate in \cite[Theorem 4.6]{LG13} to obtain
\begin{equation}\label{4-17-fu}
  \|\h\phi\|_{2,\alpha;{\md}}^{(-1-\alpha)}\leq  C\bigg(\|H_1\|_{0,\alpha;\md}^{(1-\alpha)}
+\|(h_1,h_2)\|_{1,\alpha;(0,\bar m)}^{(-\alpha)}+\|h_3\|_{0,\alpha;[L_1,L_2]}\bigg).
  \end{equation}
  Therefore,  there exists a unique solution  $ \h\phi $ to the problem \eqref{4-8} if and only if
\begin{equation}\label{4-3-a}
\int_0^{\bar m}(\lambda_1h_2-\lambda_1h_1)(y_2)\de y_2+\int_{L_1}^{L_2}\lambda_2(\bar m)h_3(y_1)\de y_1=\iint_\md H_1 \de y_1 \de y_2.
\end{equation} Furthermore,  $ \h\phi $ satisfies the estimate \eqref{4-17-fu}. Then it follows from \eqref{4-7}  that
\begin{equation}\label{4-4-de}
\|(\hat v_1,\hat v_2)\|_{1,\alpha;\md}^{(-\alpha)}\leq C\bigg(\|H_1\|_{0,\alpha;\md}^{(1-\alpha)}
+\|(h_1,h_2)\|_{1,\alpha;(0,\bar m)}^{(-\alpha)}+\|h_3\|_{0,\alpha;[L_1,L_2]}\bigg).
\end{equation}
The above estimate constant $ C > 0 $   depends only on $(\lambda_1,\lambda_2,\lambda_3,\lambda_4,L_1,L_2,\bar m)$.
\par  { \bf Step 3}:  In this step, we will solve \eqref{4-6}.
The first equation in \eqref{4-6} implies that there exists a potential function $ \check \phi $ such that
 \begin{equation}\label{4-7-fu}
 \p_{y_2}\check \phi=-\lambda_1(y_2)\check v_1, \ \ \p_{y_1}\check \phi=\lambda_2(y_2)\check v_2,\ \ \check \phi(L_1,0)=0.
 \end{equation}
  Substituting \eqref{4-7-fu} into the second equation in \eqref{4-6} yields that
\begin{equation}\label{4-6-fu}
\begin{cases}
\begin{aligned}
&\p_{y_1}\bigg(\frac{\lambda_3(y_2)}{\lambda_2(y_2)}\p_{y_1}\check \phi\bigg)+\p_{y_2}\bigg(\frac{\lambda_4(y_2)}{\lambda_1(y_2)}\p_{y_2}\check \phi\bigg)=H_2,\ &{\rm{in}} \ \md,\\
&\check \phi(L_1,y_2)=0, \ &{\rm{on}} \ \Gamma_{1},\\
&\check \phi(L_2,y_2)=0, \ &{\rm{on}} \ \Gamma_{2},\\
&\check \phi(y_1,0)=0, \ &{\rm{on}} \ \Gamma_{3},\\
&\check \phi(y_1,\bar m)=0, \ &{\rm{on}} \ \Gamma_{4}.\\
\end{aligned}
\end{cases}
\end{equation}
By the Lax-Milgram theorem, there exists a unique weak solution $\check\phi\in H_0^1(\md)$ to \eqref{4-6-fu}. Furthermore, the solution satisfies
\begin{equation}\label{4-29}
\|\check \phi\|_{H^1(\md)}\leq C\|H_2\|_{L^2(\md)}\leq C\|H_2\|_{0,\alpha;\md}^{(1-\alpha)}.
\end{equation}
Then one can follow the analogous arguments as in { \bf Step 2} to obtain that
\begin{equation}\label{4-16-a}
\begin{aligned}
\iint_{D_{\nu}(\mathbf{y}_0)}|\n\h\phi-(\n\check\phi)_{\mathbf{y}_0,\nu}
|^2  \de y_1 \de y_2
\leq C\nu^{2+2\alpha} \bigg(\|H_2\|_{0,\alpha;\md}^{(1-\alpha)}
\bigg)^2
\end{aligned}
\end{equation}
for any $ \mathbf{y}_0  \in \overline{\md} $. Here
$$ (\n\check\phi)_{\mathbf{y}_0,\nu}:=\frac{1}{|D_{\nu}(\mathbf{y}_0)|}
  \iint_{D_{\nu}(\mathbf{y}_0)} |\n\check\phi|\de y_1 \de y_2.$$
With the estimate \eqref{4-16-a}, it follows from \cite[Theorem 3.1]{HL11} that
   \begin{equation}\label{4-17-11}
  \|\check\phi\|_{1,\alpha;\overline{\md}}\leq  C\|H_2\|_{0,\alpha;\md}^{(1-\alpha)}
.
  \end{equation}
  \par Next, we utilize the  scaling argument to obtain   the $ C_{2,\alpha}^{(-1-\alpha)}(\md) $ estimate  for $\check\phi  $. For any fixed point $\mathbf{y}_\ast\in {\md} \setminus(\overline{\Gamma_3}\cup \overline{\Gamma_4}) $, define $ 2d={\rm{dist}}(\mathbf{y}_\ast, \overline{\Gamma_3}\cup \overline{\Gamma_4})$ and a scaled function
  \begin{equation*}
 \check \phi^{(\mathbf{y}_\ast)}(\mathbf{z}):=\frac{1}{d^{1+\alpha}}( \check \phi(\mathbf{y}_\ast+d\mathbf{z})-\check \phi(\mathbf{y}_\ast)-d\n \check \phi(\mathbf{y}_\ast) \cdot \mathbf{z})
  \end{equation*}
  for $ \mathbf{z}\in \{\mathbf{z}\in B_1(0):\mathbf{y}_\ast+d\mathbf{z}\in \md\}=:\mm_1(\mathbf{y}_\ast)$. It follows from \eqref{4-17} that
  \begin{equation}\label{4-17-ut}
  \|\check \phi^{(\mathbf{y}_\ast)}\|_{0,0;\overline{\mm_1(\mathbf{y}_\ast)}}\leq  C\|H_2\|_{0,\alpha;\md}^{(1-\alpha)}
.
  \end{equation}
  Substituting $ \check \phi^{(\mathbf{y}_\ast)}(\mathbf{z}) $  into \eqref{4-6-fu} gives
  \begin{equation}\label{4-6-fu-cg}
\p_{z_1}\bigg(\frac{\lambda_3}{\lambda_2}(\mathbf{y}_\ast+d\mathbf{z}) \p_{z_1}\check \phi^{(\mathbf{y}_\ast)}\bigg)+\p_{z_2}\bigg(\frac{\lambda_4}{\lambda_1}
(\mathbf{y}_\ast+d\mathbf{z})\p_{z_2}\check \phi^{(\mathbf{y}_\ast)}\bigg)=H_2^{(\mathbf{y}_\ast)},\ {\rm{in}} \ \mm_{3/4}(\mathbf{y}_\ast),\\
 \end{equation}
 where
 \begin{equation*}
 \begin{aligned}
 H_2^{(\mathbf{y}_\ast)}(z)&=d^{1-\alpha}H_2(\mathbf{y}_\ast+d\mathbf{z})
 -\p_{z_1}\bigg(\frac{\bigg(\frac{\lambda_3}{\lambda_2}(\mathbf{y}_\ast+d\mathbf{z})
 -\frac{\lambda_3}{\lambda_2}(\mathbf{y}_\ast)\bigg)\p_{y_1}
 \check \phi(\mathbf{y}_\ast)}{d^\alpha}\bigg)\\
 &\quad-\p_{z_2}\bigg(\frac{\bigg(\frac{\lambda_4}{\lambda_1}(\mathbf{y}_\ast+d\mathbf{z})
 -\frac{\lambda_4}{\lambda_1}(\mathbf{y}_\ast)\bigg)\p_{y_2}
 \check \phi(\mathbf{y}_\ast)}{d^\alpha}\bigg).
 \end{aligned}
 \end{equation*}
 Then it follows from \eqref{4-17-11} that
 \begin{equation}\label{4-17-11-fu-sc}
  \|H_2^{(\mathbf{y}_\ast)}\|_{0,\alpha;\overline{\mm_{3/4}(\mathbf{y}_\ast)}}\leq  C\|H_2\|_{0,\alpha;\md}^{(1-\alpha)}.
  \end{equation}
  By applying
   the standard Schauder  estimate in \cite[Section 6.3 and 6.7]{GT98}, there holds
   \begin{equation}\label{4-30}
 \|\check \phi^{(\mathbf{y}_\ast)}\|_{2,\alpha;\overline{\mm_{1/2}(\mathbf{y}_\ast)}} \leq  C\|H_2\|_{0,\alpha;\md}^{(1-\alpha)}.
  \end{equation}
  Therefore, re-scaling back gives
  \begin{equation}\label{4-35-5}
   \|\check \phi\|_{2,\alpha;\md}^{(-1-\alpha)} \leq  C\|H_2\|_{0,\alpha;\md}^{(1-\alpha)}.
   \end{equation}
  Then it follows from \eqref{4-7-fu}  that
  \begin{equation}\label{4-4-de-fu}
\|(\check v_1,\check v_2)\|_{1,\alpha;\md}^{(-\alpha)}\leq C\|H_2\|_{0,\alpha;\md}^{(1-\alpha)}.
\end{equation}
The above  constant $ C > 0 $   depends only on $(\lambda_1,\lambda_2,\lambda_3,\lambda_4,L_1,L_2,\bar m)$. This estimate, together with \eqref{4-4-de}, yields \eqref{4-4}. The proof of Theorem \ref{th2} is completed.
\end{proof}
\section{The initial approximating location of the shock front}\noindent
\par Let $\psi(y_2)=\bar \psi $, where  $\bar \psi $ is an unknown constant to be determined. The initial approximating location of the shock front is given by
\begin{equation}\label{5-1}
\dot\Sigma_s=\big\{(y_1,y_2): y_1=\bar\psi,\ 0<y_2<\bar m\big\}.
\end{equation}
Then $\dot\Sigma_s $ divides the domain $\mn$ into two parts $\dot\mn_-$ and $\dot\mn_+$ as
\begin{equation*}
\begin{aligned}
\dot\mn_-=\{(y_1,y_2): 0<y_1<\bar \psi ,\ 0<y_2 <\bar m\},\\
\dot\mn_+=\{(y_1,y_2): \bar \psi<y_1<L ,\ 0<y_2 <\bar m\}.\\
\end{aligned}
\end{equation*}
Set ${\bf{w}}_\pm=(w_{1,\pm},w_{2,\pm},w_{3,\pm},w_{4,\pm}) $ and $$(\e u_{1,\pm},\e u_{2,\pm},\e S_{\pm},\e B_{\pm})=(\h u_{\pm},0,\h S_{\pm},\h B_{\pm})+
(w_{1,\pm},w_{2,\pm},w_{3,\pm},w_{4,\pm}),\ \ {\rm{in}} \ \ \dot\mn_\pm.$$
To derive the equations for ${\bf{w}}_\pm$, one notes that
$$-\h\rho_{\pm} \h u_{\pm} \h u_{\pm}'+\beta
 =\frac{ e^{\h S_{\pm}}\h\rho_{\pm}^{\gamma}}{\gamma -1}   \h S_{\pm}'-\h\rho \e B_{\pm}', \ \ {\rm{in}} \ \ \dot\mn_\pm.$$ Then it follows from \eqref{3-13} that
\begin{equation}\label{5-2-p}
\begin{cases}
\begin{aligned}
& (1
 - \h M_{\pm}^2)\p_{y_1} w_{1,\pm}-\h \rho_{\pm}\h u_{\pm}'w_{2,\pm} +\h\rho_{\pm} \h u_{\pm} \p_{y_2}w_{2,\pm}=F_{1}(y_1,y_2,{\bf{w}}_\pm),\ \ &{\rm{in}} \ \dot\mn_\pm,\\
 &\p_{y_1} w_{2,\pm}-\h\rho_{\pm} \h u_{\pm} \p_{y_2}w_{1,\pm}+\bigg(-\h\rho_{\pm} \h u_{\pm}'+
 \frac{ \beta\h u_{\pm}}{\h c_{\pm}^2}+\frac{\h\rho_{\pm} \h u_{\pm} \h S_{\pm}'}{\gamma }\bigg)w_{1,\pm}=\frac{ e^{\h S_{\pm}}\h\rho_{\pm}^{\gamma}}{\gamma -1}\p_{y_2}w_{3,\pm}\\
 &-\frac{\beta}{\gamma-1}w_{3,\pm}
-\h\rho_{\pm}\p_{y_2}  w_{4,\pm}+\bigg(\frac{\beta}{\h c_{\pm}^2}+\frac{\h\rho_{\pm}  \h S_{\pm}'}{\gamma }\bigg)w_{4,\pm} +F_{2}(y_1,y_2,{\bf{w}}_\pm),\ \ &{\rm{in}} \ \dot\mn_\pm,\\
& \p_{y_1}  w_{3,\pm}=\p_{y_1} w_{4,\pm}=0,\ \ &{\rm{in}} \ \dot\mn_\pm,\\
\end{aligned}
\end{cases}
\end{equation}
where
\begin{equation*}
\begin{aligned}
&\h M_{\pm}^2(y_2)=\frac{\h u_{\pm}^2}{\h c_{\pm}^2},\ \  \h c_{\pm}^2(y_2)=\gamma e^{\h S_{\pm}}\h \rho_{\pm}^{\gamma-1},\ \ y_2\in[0,\bar m],\\
&F_{1}(y_1,y_2,{\bf{w}}_\pm)=(\e M_{1,\pm}^2-\h M_{\pm}^2)\p_{y_1} w_{1,\pm}+\e M_{1,\pm}\e M_{2,\pm}\p_{y_1}w_{2,\pm}+\bigg(\frac{\bar m}{m}-1\bigg)\e\rho_{\pm} w_{2,\pm}\p_{y_2} (\h u_{\pm}+w_{1,\pm})\\
&\quad-\bigg(\frac{\bar m}{m}-1\bigg)\e\rho_{\pm} (\h u_{\pm}+w_{1,\pm}) \p_{y_2} w_{2,\pm}
+\e\rho_{\pm} w_{2,\pm}\p_{y_2} (\h u_{\pm}+w_{1,\pm})-\h \rho_{\pm} \h u_{\pm}'w_{2,\pm}\\
&\quad-\e\rho_{\pm} (\h u_{\pm}+w_{1,\pm}) \p_{y_2}\e w_{2,\pm}
+\h \rho_{\pm} \h u_{\pm} \p_{y_2}\e w_{2,\pm},\ \ (y_1,y_2)\in \dot\mn_\pm,\\
&F_{2}(y_1,y_2,{\bf{w}}_\pm)=\frac{\bar m}{m}\e \rho_{\pm} w_{2,\pm}\p_{y_2}w_{2,\pm} +\bigg(\frac{\bar m}{m}-1\bigg)\e\rho (\h u_{\pm}\p_{y_2}w_{1,\pm}+w_{1,\pm}\h u_{\pm}'+w_{1,\pm}\p_{y_2}w_{1,\pm})\\
&\quad+\e\rho (\h u_{\pm}\p_{y_2}w_{1,\pm}+w_{1,\pm}\h u_{\pm}'+w_{1,\pm}\p_{y_2}w_{1,\pm})-\h\rho_{\pm} \h u_{\pm} \p_{y_2}w_{1,\pm}
 -\h\rho_{\pm} \h u_{\pm}' w_{1,\pm}\\
&\quad +\frac{ e^{(\h S_\pm+w_{3,\pm)}}\e\rho_{\pm}^{\gamma}}{\gamma -1}   \bigg(\frac{\bar m}{m}-1\bigg)\p_{y_2}w_{3,\pm}-\bigg(\frac{\bar m}{m}-1\bigg)\e\rho_{\pm} \p_{y_2} w_{4,\pm}+\frac{ e^{(\h S_\pm+w_{3,\pm)}} \e\rho_{\pm}^{\gamma}}{\gamma -1}\p_{y_2}w_{3,\pm}-\e\rho_{\pm} \p_{y_2}w_{4,\pm}\\
  &\quad-\frac{ e^{\h S_\pm}\h\rho_{\pm}^{\gamma}}{\gamma -1}\p_{y_2}w_{3,\pm}+\h \rho_\pm\p_{y_2}w_{4,\pm}+\beta\bigg(\frac{\bar m}{m}-1\bigg)\bigg(\frac{\e\rho_{\pm}}{\h\rho_{\pm}}-1\bigg)
  +\beta\bigg(\frac{\e\rho_{\pm}}{\h\rho_{\pm}}-1+\frac{\h u_{\pm}}{\h c_{\pm}^2}w_1+\frac{w_{3,\pm}}{\gamma-1}-\frac{w_{4,\pm}}{\h c_{\pm}^2}\bigg)\\
  &\quad+\bigg(\frac{\bar m}{m}-1\bigg)\frac{\h S_{\pm}'}{\gamma-1}\e \rho_{\pm}(e^{(\h S_\pm+w_{3,\pm)}}\e \rho_{\pm}^{\gamma-1}-e^{\h S_\pm}\h \rho_{\pm}^{\gamma-1})+\frac{\h S_{\pm}'}{\gamma-1}\e \rho_{\pm}(e^{(\h S_\pm+w_{3,\pm})}\e \rho_{\pm}^{\gamma-1}-e^{\h S_\pm}\h \rho_{\pm}^{\gamma-1})\\
 &\quad -\frac{\h S_{\pm}'}{\gamma-1}\h \rho_{\pm}(e^{(\h S_\pm+w_{3,\pm})}\e \rho_{\pm}^{\gamma-1}-e^{\h S_\pm}\h \rho_{\pm}^{\gamma-1})
  +\frac{\h S_{\pm}'}{\gamma-1}\h \rho_{\pm}\bigg(e^{(\h S_\pm+w_{3,\pm})}
  \e \rho_{\pm}^{\gamma-1}-e^{\h S_\pm}\h \rho_{\pm}^{\gamma-1}\bigg)\\
 &\quad +\frac{\h\rho_{\pm} \h u_{\pm} \h S_{\pm}'}{\gamma }w_{1,\pm}-\frac{\h\rho_{\pm}\h S_{\pm}'}{\gamma}w_{4,\pm}, \ \ (y_1,y_2)\in \dot\mn_\pm.\\
   \end{aligned}
\end{equation*}

 \par Let $ \dot{\bm V}_-=(\dot u_{1,-},\dot u_{2,-},\dot S_-,\dot B_-)^T$
 defined in $\dot\mn_-$ satisfies the linearized rotating Euler system at the supersonic state $ \h{\bm V}_-$ below, which will
 yield an initial approximation for the supersonic flow ahead of the shock front:
\begin{equation}\label{5-2}
\begin{cases}
\begin{aligned}
& (1
 - \h M_{-}^2)\p_{y_1}\dot u_{1,-}-\h \rho_-\h u_-'  \dot u_{2,-}+\h\rho_- \h u_- \p_{y_2}\dot u_{2,-}=0,\ \ &{\rm{in}} \ \dot\mn_-,\\
 &\p_{y_1} \dot u_{2,-} -\h\rho_- \h u_- \p_{y_2}\dot u_{1,-}+\bigg(-\h\rho_- \h u_-'+
 \frac{ \beta\h u_-}{\h c_-^2}+\frac{\h\rho_- \h u_- \h S_-'}{\gamma }\bigg)\dot u_{1,-}\\
 &=\frac{ e^{\h S_-}\h\rho_-^{\gamma}}{\gamma -1}\p_{y_2}\dot  S_--\frac{\beta}{\gamma-1}\dot  S_-
 -\h\rho_{-}\p_{y_2}  \dot  B_-+\bigg(\frac{\beta}{\h c_{-}^2}+\frac{\h\rho_{-}  \h S_{-}'}{\gamma }\bigg)   \dot  B_-,\ \ &{\rm{in}} \ \dot\mn_-,\\
& \p_{y_1} \dot S_-=\p_{y_1} \dot B_-=0,\ \ &{\rm{in}} \ \dot\mn_-.\\
\end{aligned}
\end{cases}
\end{equation}
Let $ \dot{\bm V}_+=(\dot u_{1,+},\dot u_{2,+},\dot S_+,\dot B_+)^T$
 defined in $\dot\mn_+$ satisfies the linearized rotating Euler system at the subsonic state $ \h{\bm V}_+$ below, which will
be given the initial approximation for the subsonic flow
behind the shock front:
\begin{equation}\label{5-3}
\begin{cases}
\begin{aligned}
& (1
 - \h M_{+}^2)\p_{y_1}\dot u_{1,+}-\h \rho_+\h u_+'  \dot u_{2,+}+\h\rho_+ \h u_+ \p_{y_2}\dot u_{2,+}=0,\ \ &{\rm{in}} \ \dot\mn_+,\\
 &\p_{y_1} \dot u_{2,+} -\h\rho_+ \h u_+ \p_{y_2}\dot u_{1,+}+\bigg(-\h\rho_+ \h u_+'+
 \frac{ \beta\h u_+}{\h c_+^2}+\frac{\h\rho_+ \h u_+ \h S_+'}{\gamma }\bigg)\dot u_{1,+}\\
 &=\frac{e^{\h S_{+}} \h\rho_+^{\gamma}}{\gamma -1}\p_{y_2}\dot  S_+-\frac{\beta}{\gamma-1}\dot  S_+
 -\h\rho_{+}\p_{y_2}  \dot  B_++\bigg(\frac{\beta}{\h c_{+}^2}+\frac{\h\rho_{+}  \h S_{+}'}{\gamma }\bigg)   \dot  B_+,\ \ &{\rm{in}} \ \dot\mn_+,\\
& \p_{y_1} \dot S_+=\p_{y_1} \dot B_+=0,\ \ &{\rm{in}} \ \dot\mn_+.
\end{aligned}
\end{cases}
\end{equation}

\par The free surface
 will be determined simultaneously with the linear shock solution $(\dot{\bm V}_\pm,\dot \psi')$.  Linearizing the Rankine-Hugoniot conditions \eqref{3-10}-\eqref{3-11} around $\dot{\bm V}_\pm$
on the free surface yields
\begin{equation}\label{5-shock}
\begin{cases}
\begin{aligned}
&G_0(\e{\bm V}_-,\e{\bm V}_+)=a_{0,+}\dot{\bm V}_++a_{0,-}\dot{\bm V}_--\frac{\bar m \dot \psi'}m[\h P]+O(|\dot{\bm V}_\pm|^2),\\
&G_1(\e{\bm V}_-,\e{\bm V}_+)=a_{1,+}(y_2)\dot{\bm V}_++a_{1,-}(y_2)\dot{\bm V}_-+O(|\dot{\bm V}_\pm|^2),
\\
 &G_2(\e{\bm V}_-,\e{\bm V}_+)=a_{2,+}(y_2)\dot{\bm V}_++a_{2,-}(y_2)\dot{\bm V}_-+O(|\dot{\bm V}_\pm|^2),\\
\end{aligned}
\end{cases}
\end{equation}
where the coefficients $a_{i,\pm}$ $ (i=0,1,2)$ are given by
\begin{equation*}
\begin{aligned}
&a_{0,\pm}=\pm(0,1,0,0)^T,\\
&a_{1,\pm}(y_2)=\pm \frac1{\h \rho_\pm\h u_\pm}\bigg(\frac{\h M_{\pm}^2-1}{\h u_\pm},0,\frac1{\gamma-1},-\frac1{\h c_\pm^2}\bigg)^T,\ \ &y_2\in[0,\bar m],\\
&a_{2,\pm}(y_2)=\pm \bigg(\frac{\h M_{\pm}^2-1}{\gamma\h M_\pm^2},0,0,\frac{\gamma-1}{\gamma\h u_\pm}\bigg)^T,\ \ &y_2\in[0,\bar m].\\
\end{aligned}
\end{equation*}
\subsection{The solution  $ \dot{\bm V}_-$ in $ \mn$}\noindent
\par Since $\h M_->1$, it is obvious that $ \dot{\bm V}_-$ is governed by a hyperbolic system. Then it can
be solved in $ \mn$ by applying the classical theory for initial-boundary value problems of the first order hyperbolic
system. It follows from  \eqref{5-2} that $ \dot{\bm V}_-$ satisfies the following problem
\begin{equation}\label{5-4}
\begin{cases}
\begin{aligned}
& \p_{y_1}( b_{1,-}(y_2)\dot u_{1,-}) +\p_{y_2}( b_{2,-}(y_2)\dot u_{2,-})=0,  \ \ &{\rm{in}} \ \mn,\\
 &\p_{y_1} ( b_{3,-}(y_2)\dot u_{2,-}) -\p_{y_2}( b_{4,-}(y_2)\dot u_{1,-})
 =b_{3,-}(y_2)\bigg(\frac{ e^{\h S_-}\h\rho_-^{\gamma}}{\gamma -1}\p_{y_2}\dot  S_-\\
 & -\frac{\beta}{\gamma-1}\dot  S_--\h\rho_{-}\p_{y_2}  \dot  B_-+\bigg(\frac{\beta}{\h c_{-}^2}+\frac{\h\rho_{-}  \h S_{-}'}{\gamma}\bigg)   \dot  B_-\bigg),\ \ &{\rm{in}} \ \mn,\\
& \p_{y_1} \dot S_-=\p_{y_1} \dot B_-=0,\ \ &{\rm{in}} \ \mn,\\
\end{aligned}
\end{cases}
\end{equation}
with the following boundary conditions:
\begin{equation}\label{5-5}
\begin{cases}
\dot{\bm {V}}_-(0,y_2)=\sigma \e{\bm {V}}_{en}(y_2),  \  \ &y_2\in[0,\bar m],\\
 {\dot u_{2,-}}(y_1,\bar m)=\sigma{\h u_{-}}(\bar m)g'(y_1),\  &y_1\in[0,L],\\
 \dot u_{2,-}(y_1,0)=0,\  &y_1\in[0,L],
 \end{cases}
 \end{equation}
where
\begin{equation*}
\begin{aligned}
& b_{1,-}(y_2)=\frac{1
 - \h M_{-}^2}{\h\rho_- \h u_-} b_{2,-}, \ \  b_{2,-}(y_2)=\exp \left(\int_{0}^{y_2}-
  \frac{\h u_-'}{ \h u_-}(s)\de s\right),\  \  y_2\in[0,\bar m],\\
  &  b_{3,-}(y_2)=\frac{
  1}{\h\rho_- \h u_-}b_{4,-},\ \ b_{4,-}(y_2)=\exp \bigg( \int_{0}^{y_2}\bigg(\frac{ \h u_-'}{ \h u_-}-
 \frac{ \beta}{\h \rho_-\h c_-^2}-\frac{ \h S_-'}{\gamma}\bigg)(s)\de s\bigg),\ \  y_2\in[0,\bar m].\\
 \end{aligned}
\end{equation*}
\begin{lemma}\label{le2}
Assume that \eqref{2-1-bo} and \eqref{2-bo} hold. There exists a unique solution  $ \dot{\bm V}_-$ to \eqref{5-4} and \eqref{5-5}, which satisfies
\begin{equation}\label{5-6}
 \| \dot{\bm V}_-\|_{2,\alpha;\overline{\mn}}\leq C\sigma\left(\| \e{\bm {V}}_{en}\|_{2,\alpha;[0,\bar m]}+\| \h u_-(\bar m)g'\|_{2,\alpha;[0,L]}\right)\leq \mc_-\sigma.
 \end{equation}
 for some constant $\mc_->0 $ depending only on $(\h {\bm V}_-,L,\bar m)$ and $\alpha$.
 Furthermore, it holds that
 \begin{equation}\label{5-7}
 (\dot S_-, \dot B_-)(y_1,y_2)=\sigma(\e S_{en},\e B_{en})(y_2),\ \ (y_1,y_2)\in \mn,
 \end{equation}
 and
 \begin{equation}\label{5-8}
 \int_0^{\bar m}b_{1,-}(s)\dot u_{1,-}(y_1,s)\de s=\sigma\int_0^{\bar m}( b_{1,-}\e u_{1,en})(s)\de s-\sigma  b_{2,-}(\bar m){\h u_{-}}(\bar m)g(y_1), \ \ y_1\in[0,L].
 \end{equation}
\end{lemma}
\begin{proof}
The unique existence   of the supersonic flow to \eqref{5-4} and \eqref{5-5} can be established via the characteristic method (see \cite{LY85}). Next, we verify \eqref{5-8}. The first equation in \eqref{5-4} implies that there exists a potential function $ \dot \phi_-$ such that
\begin{equation*}
\p_{y_2}\dot \phi_-= b_{1,-}\dot u_{1,-}, \ \ \p_{y_1}\dot \phi_-=- b_{2,-}\dot u_{2,-}, \ \ \dot \phi_-(0,0)=0, \ \ {\rm{in}} \ \ \mn.
\end{equation*}
By the boundary conditions in \eqref{5-5},  one gets
\begin{equation*}
\begin{aligned}
&\dot \phi_-(0,y_2)=\sigma\int_0^{y_2}( b_{1,-}\e u_{1,en})(s)\de s,\\
&\int_0^{\bar m} b_{1,-}(s)\dot u_{1,-}(y_1,s)\de s=\int_0^{\bar m}\p_{y_2}\dot \phi_-(y_1,s)\de s
=\dot \phi_-(y_1,\bar m)-\dot \phi_-(y_1,0)\\
&=\sigma\int_0^{\bar m}( b_{1,-}\e u_{1,en})(s)\de s+\int_0^{y_1}\bigg(\p_{y_1}\dot \phi_-(s,\bar m)-\p_{y_1}\dot \phi_-(s,0)\de s\bigg)\\
&=\sigma\int_0^{\bar m}( b_{1,-}\e u_{1,en})(s)\de s-\int_0^{y_1}b_{2,-}(\bar m)\dot u_{2,-}(y_1,\bar m)\de s\\
&=\sigma\int_0^{\bar m}( b_{1,-}\e u_{1,en})(s)\de s-\sigma  b_{2,-}(\bar m){\h u_{-}}(\bar m)g(y_1).
\end{aligned}
\end{equation*}
\end{proof}
\subsection{The determination of  $ \dot{\bm V}_+$ and $ \bar \psi$}\noindent
$ \dot{\bm V}_+$ satisfies the following problem
\begin{equation}\label{5-9}
\begin{cases}
\begin{aligned}
& \p_{y_1}( b_{1,+}(y_2)\dot u_{1,+}) +\p_{y_2}( b_{2,+}(y_2)\dot u_{2,+})=0,\ \ &{\rm{in}} \ \dot\mn_+,\\
 &\p_{y_1} ( b_{3,+}(y_2)\dot u_{2,+}) -\p_{y_2}( b_{4,+}(y_2)\dot u_{1,+})=b_{3,+}(y_2)\bigg(\frac{ e^{S_+}\h\rho_+^{\gamma}}{\gamma -1}\p_{y_2}\dot  S_+\\
 & -\frac{\beta}{\gamma-1}\dot  S_+-\h\rho_{-}\p_{y_2}  \dot  B_++\bigg(\frac{\beta}{\h c_{+}^2}+\frac{\h\rho_{+}  \h S_{+}'}{\gamma }\bigg)   \dot  B_+\bigg),\ \ &{\rm{in}} \ \dot\mn_+,\\
& \p_{y_1} \dot S_+=\p_{y_1} \dot B_+=0,\ \ &{\rm{in}} \ \dot\mn_+,\\
\end{aligned}
\end{cases}
\end{equation}
where
\begin{equation*}
\begin{aligned}
& b_{1,+}(y_2)=\frac{1
 - \h M_{+}^2}{\h\rho_+ \h u_+} b_{2,+}, \ \  b_{2,+}(y_2)=\exp \left(\int_{0}^{y_2}-
  \frac{\h u_+'}{ \h u_+}(s)\de s\right),\  \  y_2\in[0,\bar m],\\
  & b_{3,+}(y_2)=\frac{
  1}{\h\rho_+ \h u_+}b_{4,+},\ \  b_{4,+}(y_2)=\exp \bigg( \int_{0}^{y_2}\bigg(\frac{ \h u_+'}{ \h u_+}-
 \frac{ \beta}{\h \rho_+\h c_+^2}-\frac{ \h S_+'}{\gamma}\bigg)(s)\de s\bigg),\ \  y_2\in[0,\bar m].\\
 \end{aligned}
\end{equation*}
On the nozzle walls, \eqref{5-9} subject to the boundary conditions:
\begin{equation}\label{5-10}
\begin{cases}
{\dot u_{2,+}}(y_1,\bar m)=\sigma{\h u_{+}}(\bar m)g'(y_1),\  &y_1\in[\bar\psi,L],\\
 \dot u_{2,+}(y_1,0)=0,\  &y_1\in[\bar\psi,L].
 \end{cases}
 \end{equation}
\par  With the supersonic state $ \dot{\bm V}_-$ being specified, the boundary conditions for
 $ \dot{\bm V}_+$ at the free line $ y_1=\bar \psi$ are completely determined by the linearized
 Rankine-Hugoniot conditions in \eqref{5-shock} in and the upstream flow. Thus we
 impose the following boundary condition on the free line $ y_1=\bar \psi$:
 \begin{equation}\label{5-shock-1}
\begin{cases}
a_{1,+}(y_2)\dot{\bm V}_+(\bar \psi,y_2)+a_{1,-}(y_2)\dot{\bm V}_-(\bar \psi,y_2)
=0,\  \  &y_2\in[0,\bar m],\\
 a_{2,+}(y_2)\dot{\bm V}_+(\bar \psi,y_2)+a_{1,-}(y_2)\dot{\bm V}_-(\bar \psi,y_2)=0, \  \  &y_2\in[0,\bar m].\\
\end{cases}
\end{equation}
These relations ultimately determine the boundary values of $ \dot u_{1,+}$ and $\dot S_+$ at
$ y_1=\bar \psi$.
\begin{lemma}
On the free boundary $\dot\Sigma_s $, it holds that
\begin{equation}\label{5-shock-2}
\begin{cases}
{\dot u_{1,+}}(\bar\psi,y_2)=\mathfrak {a}_1(y_2){\dot u_{1,-}}(\bar\psi,y_2),\  \  &y_2\in[0,\bar m],\\
{\dot S_{+}}(\bar\psi,y_2)=\mathfrak {a}_2(y_2){\dot u_{1,-}}(\bar\psi,y_2)+\sigma \e S_{en}(y_2),\  \  &y_2\in[0,\bar m],\\
\end{cases}
\end{equation}
where
\begin{equation*}
\mathfrak {a}_1(y_2)=\frac{ \h M_{+}^2}{\h M_{-}^2}\frac{ \h M_{-}^2-1}{\h M_{+}^2-1},\  \ \mathfrak {a}_2(y_2)=\frac{(\gamma-1)\h M_{-}^2-1}{\h P_+ \h u_-}[\h P], \ \ y_2\in[0,\bar m].
\end{equation*}

\end{lemma}
\begin{proof}
It follows from
 \eqref{5-shock-1}  that
\begin{equation}\label{5-12}
\bigg(\begin{array}{ccc}
\frac{\h M_{+}^2-1}{\h u_+} & \frac1{\gamma-1} \\
\frac{\h M_{+}^2-1}{\gamma\h M_+^2} &  0
 \end{array}\bigg)\bigg(\begin{array}{ccc}{\dot u_{1,+}}\\
 {\dot S_{+}}\end{array}\bigg)=\bigg(\begin{array}{ccc}
\frac{\h M_{-}^2-1}{\h u_-} & \frac1{\gamma-1} \\
\frac{\h M_{-}^2-1}{\gamma\h M_-^2} &  0
 \end{array}\bigg)\bigg(\begin{array}{ccc}{\dot u_{1,-}}\\
 {\dot S_{-}}\end{array}\bigg).
 \end{equation}
Note that
 \begin{equation*}
 {\rm{det}}\bigg(\begin{array}{ccc}
\frac{\h M_{\pm}^2-1}{\h u_\pm} & \frac1{\gamma-1} \\
\frac{\h M_{\pm}^2-1}{\gamma\h M_\pm^2} &  0
 \end{array}\bigg)=\frac{1-\h M_{\pm}^2}{(\gamma-1)\gamma\h M_\pm^2}\neq 0.
 \end{equation*}
 Then one gets
 \begin{equation}\label{5-13}
 \begin{aligned}
\bigg(\begin{array}{ccc}{\dot u_{1,+}}\\
 {\dot S_{+}}\end{array}\bigg)&=\bigg(\begin{array}{ccc}
\frac{\h M_{+}^2-1}{\h u_+} & \frac1{\gamma-1} \\
\frac{\h M_{+}^2-1}{\gamma\h M_+^2} &  0
 \end{array}\bigg)^{-1}\bigg(\begin{array}{ccc}
\frac{\h M_{-}^2-1}{\h u_-} & \frac1{\gamma-1} \\
\frac{\h M_{-}^2-1}{\gamma\h M_-^2} &  0
 \end{array}\bigg)\bigg(\begin{array}{ccc}{\dot u_{1,-}}\\
 {\dot S_{-}}\end{array}\bigg)
 =\bigg(\begin{array}{ccc}\frac{ \h M_{+}^2}{\h M_{-}^2}\frac{ \h M_{-}^2-1}{\h M_{+}^2-1},& 0\\
 \frac{(\gamma-1)\h M_{-}^2-1}{\h P_+ \h u_-}[\h P],& 1\end{array}\bigg)\bigg(\begin{array}{ccc}{\dot u_{1,-}}\\
 {\dot S_{-}}\end{array}\bigg).\\
 \end{aligned}
 \end{equation}
 Furthermore,  \eqref{4-8} implies that
 \begin{equation}\label{5-14}
{\dot S_{-}}(\bar \psi,y_2)=\sigma \e S_{en}(y_2),\ \  y_2\in[0,\bar m].\\
 \end{equation}
 This, together with \eqref{5-13}, yields \eqref{5-shock-2}.
\end{proof}
It follows from the fourth equation in \eqref{3-8-r} and the last two equations in \eqref{5-9} that
\begin{equation}\label{5-14-y-hy}
\begin{cases}
\p_{y_1} \dot S_+= \p_{y_1}\dot B_+=0,\ \ &{\rm{in}} \ \dot\mn_+,\\
{\dot B_{+}}(\bar \psi,y_2)={\dot B_{-}}(\bar \psi,y_2),\  \  &y_2\in[0,\bar m],\\
{\dot S_{+}}(\bar\psi,y_2)=\mathfrak {a}_2(y_2){\dot u_{1,-}}(\bar\psi,y_2)+\sigma \e S_{en}(y_2),\  \  &y_2\in[0,\bar m].\\
\end{cases}
\end{equation}
Thus there holds
\begin{equation}\label{5-14-y}
\begin{cases}
{\dot B_{+}}(y_1,y_2)={\dot B_{+}}(\bar \psi,y_2)={\dot B_{-}}(\bar \psi,y_2)=\sigma \e B_{en}(y_2), &(y_1,y_2)\in \dot\mn_+,\\
{\dot S_{+}}(y_1,y_2)={\dot S_{+}}(\bar \psi,y_2)=\mathfrak {a}_2(y_2){\dot u_{1,-}}(\bar\psi,y_2)+\sigma \e S_{en}(y_2),&(y_1,y_2)\in \dot\mn_+.\\
\end{cases}
\end{equation}
 On the exit of the nozzle, the  pressure is prescribed as
\begin{equation}\label{5-10-l}
\e P_+(L,y_2)=\h P_+(y_2)+\sigma  P_{ex}\bigg(\frac m{\bar m}\int_{0}^{y_2}\frac{1}{(\rho_+ u_{1,+})(L,s)}\de s\bigg),\ \  y_2\in [0,\bar m].
\end{equation}
It follows from \eqref{3-9} that
\begin{equation}\label{5-15-ex}
\e P_+=\h P_+-{\h \rho_+ \h u_+}\dot u_{1,+}-\frac{\h P_+}{\gamma-1}
\dot S_{+}+\h \rho_+\dot B_{+}+O(|\dot{\bm V}_+|^2).
\end{equation}
Substituting \eqref{5-14-y} into the above equation, the leading order satisfies
\begin{equation}\label{5-15}
\begin{aligned}
\dot u_{1,+}(L,y_2)&=- \frac{\sigma}{(\h \rho_+ \h u_+)(y_2)}P_{ex}\bigg(\int_{0}^{y_2}\frac{1}{(\bar\rho_+ \bar u_+)(s)}\de s\bigg)+\mathfrak {a}_3(y_2){\dot u_{1,-}}(\bar\psi,y_2)\\
&\quad-\frac{\sigma(\h P_+\e  S_{en})(y_2)}{(\gamma-1)(\h \rho_+ \h u_+)(y_2)}+\frac{\sigma\e  B_{en}(y_2)}{ \h u_+(y_2)} , \ \  y_2\in [0,\bar m],
\end{aligned}
\end{equation}
where
\begin{equation*}
\mathfrak {a}_3(y_2)=-\frac{\h M_{-}^2-1}{ \h \rho_+ \h u_+\h u_-}[\h P], \ \  y_2\in [0,\bar m].
\end{equation*}
Then $(\dot u_{1,+},\dot u_{2,+})$ satisfies the following problem:
\begin{equation}\label{5-16}
\begin{cases}
\begin{aligned}
& \p_{y_1}( b_{1,+}(y_2)\dot u_{1,+}) +\p_{y_2}( b_{2,+}(y_2)\dot u_{2,+})=0,\ \ &{\rm{in}} \ \dot\mn_+,\\
 &\p_{y_1} ( b_{3,+}(y_2)\dot u_{2,+}) -\p_{y_2}( b_{4,+}(y_2)\dot u_{1,+})=b_{3,+}(y_2)\bigg(\frac{ e^{\h S_+}\h\rho_+^{\gamma}}{\gamma -1}\p_{y_2}\dot  S_+\\
 & -\frac{\beta}{\gamma-1}\dot  S_+-\h\rho_{-}\p_{y_2}  \dot  B_++\bigg(\frac{\beta}{\h c_{+}^2}+\frac{\h\rho_{+}  \h S_{+}'}{\gamma }\bigg)   \dot  B_+\bigg),\ \ &{\rm{in}} \ \dot\mn_+,\\
&{\dot u_{2,+}}(y_1,\bar m)=\sigma{\h u_{+}}(\bar m)g'(y_1),\  &y_1\in[\bar\psi,L],\\
 &\dot u_{2,+}(y_1,0)=0,\  &y_1\in[\bar\psi,L],\\
 &{\dot u_{1,+}}(\bar\psi,y_2)=\mathfrak {a}_1(y_2){\dot u_{1,-}}(\bar\psi,y_2),\  \  &y_2\in[0,\bar m],\\
 &\dot u_{1,+}(L,y_2)=- \frac{\sigma}{(\h \rho_+ \h u_+)(y_2)}P_{ex}\bigg(\int_{0}^{y_2}\frac{1}{(\bar\rho_+ \bar u_+)(s)}\de s\bigg)+\mathfrak {a}_3(y_2){\dot u_{1,-}}(\bar\psi,y_2)\\
 &\qquad\qquad\qquad-\frac{\sigma(\h P_+\e  S_{en})(y_2)}{(\gamma-1)(\h \rho_+ \h u_+)(y_2)}+\frac{\sigma\e  B_{en}(y_2)}{ \h u_+(y_2)}, \ \  &y_2\in [0,\bar m].
\end{aligned}
\end{cases}
\end{equation}

\begin{lemma}\label{le5}
Fix $ \alpha\in(\frac12,1) $.  There exists a unique solution  $ (\dot u_{1,+},\dot u_{2,+})$ to \eqref{5-16}
if and only if
\begin{equation}\label{5-3-sub}
\begin{aligned}
&\int_0^{\bar m}\bigg(\big(\mathfrak {a}_3-\mathfrak {a}_1\big)(y_2)b_{1,+}(y_2){\dot u_{1,-}}(\bar\psi,y_2)-b_{1,+}(y_2)\bigg( \frac{\sigma }{(\h \rho_+ \h u_+)(y_2)}P_{ex}\bigg(\int_{0}^{y_2}\frac{1}{(\bar\rho_+ \bar u_+)(s)}\de s\bigg)\\
&\qquad-\frac{\sigma(\h P_+\e  S_{en})(y_2)}{(\gamma-1)(\h \rho_+ \h u_+)(y_2)}+\frac{\sigma\e  B_{en}(y_2)}{ \h u_+(y_2)}\bigg)\bigg)\de y_2
+\sigma (b_{2,+}{\h u_{+}})(\bar m)\big(g(L)-g(\bar \psi)\big)=0.\\
\end{aligned}
\end{equation}
In addition,
\begin{equation}\label{5-6-sub}
 \|(\dot u_{1,+},\dot u_{2,+})\|_{1,\alpha;\dot\mn_+}^{(-\alpha)}\leq C\sigma\left(\| \e{\bm {V}}_{en}\|_{2,\alpha;[0,\bar m]}+\| g'\|_{2,\alpha;[0,L]}\right),
 \end{equation}
 where the constant $ C>0 $ depends only on $(\h {\bm V}_\pm,L,\bar m)$ and $\alpha$.
 \end{lemma}
 Applying Theorem \ref{th2}, it suffices to prove that there exists a $ \bar\psi$ such that \eqref{5-3-sub} holds. Set
 \begin{equation*}
\begin{aligned}
&J_1(\bar\psi)=\frac1\sigma\int_0^{\bar m}\bigg(\big(\mathfrak {a}_3-\mathfrak {a}_1\big)(y_2)b_{1,+}(y_2){\dot u_{1,-}}(\bar\psi,y_2)\bigg)\de y_2+ (b_{2,+}{\h u_{+}})(\bar m)\big(g(L)-g(\bar \psi)\big),\\
&J_2=\int_0^{\bar m}\bigg(b_{1,+}(y_2)\bigg(\frac{ P_{ex}\bigg(\int_{0}^{y_2}\frac{1}{(\bar\rho_+ \bar u_+)(s)}\de s\bigg)}{(\h \rho_+ \h u_+)(y_2)}+\frac{(\h P_+\e  S_{en})(y_2)}{(\gamma-1)(\h \rho_+ \h u_+)(y_2)}-\frac{\e  B_{en}(y_2)}{ \h u_+(y_2)}\bigg)\bigg)\de y_2.
\end{aligned}
\end{equation*}
A direct computation yields that
\begin{equation*}
\begin{aligned}
J_1'(\bar\psi)&=\frac1\sigma\int_0^{\bar m}\bigg(\big(\mathfrak {a}_3-\mathfrak {a}_1\big)(y_2)b_{1,+}(y_2)\p_{\bar\psi}{\dot u_{1,-}}(\bar\psi,y_2)\bigg)\de y_2- (b_{2,+}{\h u_{+}})(\bar m)g'(\bar\psi)\\
&=\frac1\sigma\int_0^{\bar m}\bigg(\frac{b_{1,+}(\mathfrak {a}_3-\mathfrak {a}_1)}{b_{1,-}}(y_2)\p_{\bar\psi}(b_{1,-}(y_2){\dot u_{1,-}}(\bar\psi,y_2))\bigg)\de y_2- (b_{2,+}{\h u_{+}})(\bar m)g'(\bar\psi)\\
&=-\frac1\sigma\int_0^{\bar m}\bigg(\frac{b_{1,+}(\mathfrak {a}_3-\mathfrak {a}_1)}{b_{1,-}}(y_2)\p_{y_2}\big(b_{2,-}(y_2){\dot u_{2,-}}(\bar\psi,y_2)\big)\bigg)\de y_2- (b_{2,+}{\h u_{+}})(\bar m)g'(\bar\psi),\\
&=\frac1\sigma\int_0^{\bar m}\bigg(\bigg(\frac{b_{1,+}(\mathfrak {a}_3-\mathfrak {a}_1)}{b_{1,-}}\bigg)'(y_2)b_{2,-}(y_2){\dot u_{2,-}}(\bar\psi,y_2)\bigg)\de y_2- (b_{2,+}{\h u_{+}})(\bar m)g'(\bar\psi)\\
&\quad-\bigg(\frac{b_{1,+}(\mathfrak {a}_3-\mathfrak {a}_1)}{b_{1,-}}\bigg)(\bar m)(b_{2,-}\h u_-)(\bar m)g'(\bar\psi),\\
J_1''(\bar\psi)&=\int_0^{\bar m}\bigg(\big(\mathfrak {a}_3-\mathfrak {a}_1\big)(y_2)b_{1,+}(y_2)\frac{\p_{\bar\psi}^2{\dot u_{1,-}}(\bar\psi,y_2)}{\sigma}\bigg)\de y_2- (b_{2,+}{\h u_{+}})(\bar m)g''(\bar\psi).\\
\end{aligned}
\end{equation*}

\begin{lemma}\label{le6} Assume that \eqref{2-1-bo} and \eqref{2-bo} hold.
Let
$$\mathfrak {C}=\mc_-\int_0^{\bar m}\left|\big(\mathfrak {a}_3-\mathfrak {a}_1\big)(y_2)b_{1,+}(y_2)\right|\de y_2+(b_{2,+}{\h u_{+}})(\bar m)\|g''\|_{L^{\infty}[0,L]}>0,$$
where the   constant $\mc_- $ is derived from the estimate  \eqref{5-6}.
\begin{enumerate}[ \rm (i)]
\item If
\begin{equation*}
\int_0^{\bar m}\bigg(\bigg(\frac{b_{1,+}(\mathfrak {a}_3-\mathfrak {a}_1)}{b_{1,-}}\bigg)'(y_2)(b_{2,-}{\e u_{2,en}})(y_2)\bigg)\de y_2>0 ,
\end{equation*}
set
\begin{equation}\label{5-18}
0<L_+<\min \{L_\ast,L\}
\end{equation}
with
\begin{equation*}
L_\ast=\frac{1}{\mathfrak {C}}\int_0^{\bar m}\bigg(\bigg(\frac{b_{1,+}(\mathfrak {a}_3-\mathfrak {a}_1)}{b_{1,-}}\bigg)'(y_2)(b_{2,-}{\e u_{2,en}})(y_2)\bigg)\de y_2.
\end{equation*}
Then if
\begin{equation}\label{5-19}
J_1(0)<J_2<J_1(0)+ J_\ast(L_+),
\end{equation}
where
\begin{equation*}
\begin{aligned}
J_1(0)&=\int_0^{\bar m}\bigg(\big(\mathfrak {a}_3-\mathfrak {a}_1\big)(y_2)(b_{1,+}\e u_{1,en})(y_2)\bigg)\de y_2+ (b_{2,+}{\h u_{+}})(\bar m)g(L),\\
J_\ast(L_+)&=L_+\bigg(\int_0^{\bar m}\bigg(\bigg(\frac{b_{1,+}(\mathfrak {a}_3-\mathfrak {a}_1)}{b_{1,-}}\bigg)'(y_2)(b_{2,-}{\e u_{2,en}})(y_2)\bigg)\de y_2-\frac{L_+}2\mathfrak {C}\bigg)>0,
\end{aligned}
\end{equation*}
there exists a $ \bar\psi$ such that
\begin{equation}\label{5-20}
J_1(\bar \psi)=J_2.
\end{equation}
\item  If
\begin{equation*}
\int_0^{\bar m}\bigg(\bigg(\frac{b_{1,+}(\mathfrak {a}_3-\mathfrak {a}_1)}{b_{1,-}}\bigg)'(y_2)(b_{2,-}{\e u_{2,en}})(y_2)\bigg)\de y_2<0 ,
\end{equation*}
set
\begin{equation}\label{5-18-le}
0<L_-<\min \{L_\star,L\}
\end{equation}
with
\begin{equation*}
L_\star=-\frac{1}{\mathfrak {C}}\int_0^{\bar m}\bigg(\bigg(\frac{b_{1,+}(\mathfrak {a}_3-\mathfrak {a}_1)}{b_{1,-}}\bigg)'(y_2)(b_{2,-}{\e u_{2,en}})(y_2)\bigg)\de y_2.
\end{equation*}
Then if
\begin{equation}\label{5-19-le}
J_1(0)+ J_\star(L_-)<J_2<J_1(0),
\end{equation}
where
\begin{equation*}
\begin{aligned}
J_1(0)&=\int_0^{\bar m}\bigg(\big(\mathfrak {a}_3-\mathfrak {a}_1\big)(y_2)((b_{1,+}\e u_{1,en})(y_2)\bigg)\de y_2+ (b_{2,+}{\h u_{+}})(\bar m)g(L),\\
J_\star(L_-)&=L_-\bigg(\int_0^{\bar m}\bigg(\bigg(\frac{b_{1,+}(\mathfrak {a}_3-\mathfrak {a}_1)}{b_{1,-}}\bigg)'(y_2)(b_{2,-}{\e u_{2,en}})(y_2)\bigg)\de y_2+\frac{L_-}2\mathfrak {C}\bigg)<0,
\end{aligned}
\end{equation*}
 there exists a $ \bar\psi$ such that
\begin{equation}\label{5-20-le}
J_1(\bar \psi)=J_2.
\end{equation}
\end{enumerate}
\end{lemma}
\begin{proof} We  only prove the  case $\rm (i)$. The case $\rm (ii)$ can be treated similarly.
By applying the first boundary condition in \eqref{5-5}, one obtains
\begin{equation}\label{5-17}
\begin{aligned}
J_1'(\bar\psi)&=J_1'(0)+J_1''(\psi_\ast)\bar\psi\\&=\int_0^{\bar m}\bigg(\bigg(\frac{b_{1,+}(\mathfrak {a}_3-\mathfrak {a}_1)}{b_{1,-}}\bigg)'(y_2)b_{2,-}(y_2){\e u_{2,en}}(y_2)\bigg)\de y_2\\
&\quad+\bar\psi\bigg(\int_0^{\bar m}\bigg(\big(\mathfrak {a}_3(y_2)-\mathfrak {a}_1(y_2)\big)b_{1,+}(y_2)\frac{\p_{\bar\psi}^2{\dot u_{1,-}}(\bar\psi,y_2)}{\sigma}\bigg)\de y_2- (b_{2,+}{\h u_{+}})(\bar m)g''(\bar\psi)\bigg)\\
&\geq \int_0^{\bar m}\bigg(\bigg(\frac{b_{1,+}(\mathfrak {a}_3-\mathfrak {a}_1)}{b_{1,-}}\bigg)'(y_2)b_{2,-}(y_2){\e u_{2,en}}(y_2)\bigg)\de y_2\\
&\quad-\bar\psi\bigg(\mc_-\int_0^{\bar m}\left|(\mathfrak {a}_3(y_2)-\mathfrak {a}_1(y_2))b_{1,+}(y_2)\right|\de y_2+(b_{2,+}{\h u_{+}})(\bar m)\|g''\|_{L^{\infty}[0,L]}\bigg).
\end{aligned}
\end{equation}
Here $\psi_\ast\in(0,\bar\psi)$. Then it follows from \eqref{5-18} that for any $\bar\psi\in(0,L_+)$,
\begin{equation}\label{5-21}
J_1'(\bar\psi)>0.
\end{equation}
Therefore,
\begin{equation*}
\begin{aligned}
\min_{(0,L_+)}J_1(\bar\psi)&=J_1(0)=\int_0^{\bar m}\bigg(\big(\mathfrak {a}_3(y_2)-\mathfrak {a}_1(y_2)\big)(b_{1,+}\e u_{1,en})(y_2)\bigg)\de y_2+ (b_{2,+}{\h u_{+}})(\bar m)g(L),\\
\max_{(0,L_+)}J_1(\bar\psi)&=J_1(L_+)=\int_0^{L_+}J_1'(\bar \psi)\de\bar \psi+J_1(0)
=L_+\int_0^{\bar m}\bigg(\bigg(\frac{b_{1,+}(\mathfrak {a}_3-\mathfrak {a}_1)}{b_{1,-}}\bigg)'(y_2)b_{2,-}(y_2){\e u_{2,en}}(y_2)\bigg)\de y_2+J_1(0)\\
&\quad+\int_0^{L_+}\bar \psi\bigg(\int_0^{\bar m}\bigg(\big(\mathfrak {a}_3(y_2)-\mathfrak {a}_1(y_2)\big)b_{1,+}(y_2)\frac{\p_{\bar\psi}^2{\dot u_{1,-}}(\bar\psi,y_2)}{\sigma}\bigg)\de y_2- (b_{2,+}{\h u_{+}})(\bar m)g''(\bar\psi)\bigg)\de\bar \psi\\
&\geq J_1(0)+ J_\ast(L_+).
\end{aligned}
\end{equation*}
Thus if \eqref{5-19} holds, there exists a unique $ \bar\psi$ such that
\eqref{5-20}
holds.
\end{proof}
Collecting the above results, we state the main result in this section below.
\begin{theorem}\label{th4} Fix $ \alpha\in(\frac12,1) $. Under the assumptions in Lemma \ref{le6}, there exists a unique solution $ (\dot{\bm V}_\pm;\dot \psi',\bar \psi)$ satisfying the following properties.
\begin{enumerate}[ \rm (1)]
\item The position of the free boundary $ y_1=\bar \psi $ is determined by \eqref{5-3-sub}.
  \item $ \dot{\bm V}_-$ solves the the linearized rotating Euler system \eqref{5-4}  with the boundary condition \eqref{5-5}.
    \item  Behind  $ y_1=\bar \psi $, $ \dot{\bm V}_+$ satisfies the linear boundary value problems \eqref{5-14-y-hy} and \eqref{5-16}.
        \item The shock front $ \dot\psi'(y_2) $ is determined by
        \begin{equation}\label{5-22}
     \dot \psi'(y_2)=  \frac { m }{\bar m[\h P](y_2)}\bigg(a_{0,+}\dot{\bm V}_+(\bar \psi,y_2)+a_{0,-}\dot{\bm V}_-(\bar \psi,y_2)\bigg),\ \ y_2\in[0,\bar m].
     \end{equation}
     Furthermore, it holds that
     \begin{equation}\label{5-6-sub-all}
 \| \dot{\bm V}_-\|_{2,\alpha;\overline{\mn}}+\|\dot{\bm V}_+\|_{1,\alpha;\dot\mn_+}^{(-\alpha)}+\|\dot \psi'\|_{1,\alpha;(0,\bar m)}^{(-\alpha)}\leq C\sigma\left(\| \e{\bm {V}}_{en}\|_{2,\alpha;[0,\bar m]}+\| g'\|_{2,\alpha;[0,L]}\right),
 \end{equation}
 where $ C$ is a positive constant  depending only on $(\h {\bm V}_\pm,L,\bar m)$.
\end{enumerate}
\end{theorem}
\section{The nonlinear iteration scheme}\noindent
\par Based on the solution  $ (\dot{\bm V}_\pm;\dot \psi',\bar \psi)$ established in Theorem \ref{th4}, we design an iteration scheme to resolve the nonlinear shock problem.
\subsection{The supersonic solution  in $ \mn$}\noindent
\par Note that the steady rotating Euler system is hyperbolic in the  supersonic region.   Then we can apply the theory in \cite{LY85} to obtain the following unique existence result.
\begin{theorem}\label{6-1}
Assume that \eqref{2-1-bo} holds. Then there exists a positive constant $ \sigma_1<1$ depending only on $(\h {\bm V}_-,L,\bar m)$ such that for any $0<\sigma< \sigma_1 $, the initial boundary problem
\begin{equation}\label{6-2}
\begin{cases}
\begin{aligned}
&\ma_1 (\e{\bm V}_-)\p_{y_1}\e {\bf u}_-+\ma_2 (\e{\bm V}_-)\p_{y_2}\e {\bf u}_-+\ma_3 (\e{\bm V}_-)=0, \ \ &{\rm{in}} \ \mn,\\
& \p_{y_1} \e S_-=\p_{y_1} \e B_-=0, \ \ &{\rm{in}} \ \mn,\\
&\e{\bm {V}}_-(0,y_2)=\h{\bm {V}}_-(y_2)+\sigma \e{\bm {V}}_{en}(y_2),  \  \ &y_2\in [0,\bar m],\\
 &\frac{\e u_{2,-}}{\e u_{1,-}}(y_1,\bar m)=\sigma g'(y_1),\  & y_1\in[0,L],\\
 &\e u_{2,-}(y_1,0)=0,\  &   y_1\in[0,L],
 \end{aligned}
 \end{cases}
\end{equation}
has a unique solution $ \e{\bm V}_-\in C^{2,\alpha}(\overline{\mn})$ such that
\begin{equation}\label{6-3}
\| \e{\bm {V}}_--\h{\bm {V}}_-\|_{2,\alpha;\overline{\mn}}\leq C(\h{\bm {V}}_-,L,\bar m)\left(\| \e{\bm {V}}_{en}\|_{2,\alpha;[0,\bar m]}+\| g'\|_{2,\alpha;[0,L]}\right)\sigma.
\end{equation}Finally, let $ \ddot{\bm V}_-:=\e{\bm {V}}_--\h{\bm {V}}_-$. Then there holds
 \begin{equation}\label{6-4}
\| \ddot{\bm V}_--\dot{\bm V}_-\|_{1,\alpha;\overline{\mn}}\leq C(\h{\bm {V}}_-,L,\bar m)\sigma^2.
\end{equation}
\end{theorem}
\begin{proof}
 It follows from the second equation in \eqref{6-2}  that
\begin{equation}\label{6-3-hy}
(\e S_-, \e B_-)(y_1,y_2)=(\h S_-+\sigma  \e S_{en}, \h B_-+ \sigma\e B_{en})(y_2),\ \  (y_1,y_2)\in\mn.
\end{equation}
Then \eqref{6-2} can be simplified as \begin{equation}\label{6-2-si}
\begin{cases}
\begin{aligned}
&\ma_1 (\e{\bm V}_-)\p_{y_1}\e {\bf u}_-+\ma_2 (\e{\bm V}_-)\p_{y_2}\e {\bf u}_-+\ma_3 (\e{\bm V}_-)=0, \ \ &{\rm{in}} \ \mn,\\
&(\e u_{1,-},\e u_{2,-})(0,y_2)=(\h u_-+\sigma \e{u}_{1,en},\e{u}_{2,en})(y_2),  \  \ &y_2\in [0,\bar m],\\
 &\frac{\e u_{2,-}}{\e u_{1,-}}(y_1,\bar m)=\sigma g'(y_1),\  & y_1\in[0,L],\\
 &\e u_{2,-}(y_1,0)=0,\  &   y_1\in[0,L].
 \end{aligned}
 \end{cases}
\end{equation}
The unique  existence of the solution $ \e{\bf u}_-\in C^{2,\alpha}(\overline{\mn})$  to the problem \eqref{6-2-si}   can be obtained   by employing the characteristic method and the standard Picard iteration (see \cite{LY85}).
Next, a direct computation yields that
\begin{equation*}
\begin{cases}
\begin{aligned}
& (1
 - \h M_{-}^2)\p_{y_1}(\ddot u_{1,-}-\dot u_{1,-})-\h \rho_-\h u_-' (\ddot u_{2,-}-\dot u_{2,-})\\
 & +\h\rho_- \h u_- \p_{y_2}(\ddot u_{2,-}-\dot u_{2,-})
 =F_{1}(y_1,y_2,\ddot{\bm V}_-),&{\rm{in}} \ \mn,\\
 &\p_{y_1} (\ddot u_{2,-}-\dot u_{2,-}) -\h\rho_- \h u_- \p_{y_2}\dot (\ddot u_{1,-}-\dot u_{1,-})+\bigg(-\h\rho_- \h u_-'+
 \frac{ \beta\h u_-}{\h c_-^2}\\
 & +\frac{\h\rho_- \h u_- \h S_-'}{\gamma }\bigg)(\ddot u_{1,-}-\dot u_{1,-})=F_{2}(y_1,y_2,\ddot{\bm V}_-)
 ,&{\rm{in}} \ \mn,\\
&(\ddot u_{1,-}-\dot u_{1,-})(0,y_2)=(\ddot u_{2,-}-\dot u_{2,-})(0,y_2)=0,  &y_2\in [0,\bar m],\\
 &(\ddot u_{2,-}-\dot u_{2,-})(y_1,\bar m)-\sigma g'(y_1(\ddot u_{1,-}-\dot u_{1,-})(y_1,\bar m)=\sigma g'(y_1)\dot u_{1,-}(y_1,\bar m), &y_1\in[0,L],\\
 &(\ddot u_{2,-}-\dot u_{2,-})(y_1,0)=0, &  y_1\in[0,L],
 \end{aligned}
 \end{cases}
\end{equation*}
where $F_1$ and $F_2$  are defined as in \eqref{5-2-p} with ${\bf{w}}_-$ replaced by $\ddot{\bm V}_-$. Then it holds that
\begin{equation}\label{6-5}
\begin{aligned}
\| \ddot{\bm V}_--\dot{\bm V}_-\|_{1,\alpha;\overline{\mn}}&
\leq C\bigg(\| F_1(\ddot{\bm V}_-)\|_{1,\alpha;\overline{\mn}}+\| F_2(\ddot{\bm V}_-)\|_{1,\alpha;\overline{\mn}}+\sigma\| \dot u_{1,-}\|_{1,\alpha;\overline{\mn}}\bigg)\\
& \leq C(\h{\bm {V}}_-,L,\bar m)\sigma^2.
\end{aligned}
\end{equation}
\end{proof}
\subsection{The shock front and subsonic solution}\noindent
\par The shock front $ \Sigma_s$ is given by
\begin{equation}\label{6-6}
 \Sigma_s=\{(y_1,y_2):y_1=\psi(y_2)=\bar\psi+\ddot\psi(y_2), \ 0<y_2<\bar m\}.
 \end{equation}
 Then the region of the subsonic flow behind it is
 \begin{equation*}
 \mn_+=\{(y_1,y_2): \psi(y_2)<y_1<L ,\ 0<y_2 <\bar m\}.
 \end{equation*}
 In the subsonic domain, $(\e{\bm V}_+,\psi)$ solves the following problem:
 \begin{equation}\label{6-7}
\begin{cases}
\begin{aligned}
&\ma_1 (\e{\bm V}_+)\p_{y_1}\e {\bf u}_++\ma_2 (\e{\bm V}_+)\p_{y_2}\e {\bf u}_++\ma_3 (\e{\bm V}_+)=0, \ \ &{\rm{in}} \ \mn_+,\\
& \p_{y_1} \e S_+=\p_{y_1} \e B_+=0, \ \ &{\rm{in}} \ \mn_+,\\
&[\e B]=0,  \ \
 G_i(\e{\bm V}_-,\e{\bm V}_+)=0, \ i=1,2,3, \ \ &{\rm{on}} \ \Sigma_s,\\
 &\e P_+(L,y_2)=\h P_+(y_2)+\sigma  P_{ex}\bigg(\frac m{\bar m}\int_{0}^{y_2}\frac{1}{(\e\rho_+ \e u_{1,+})(L,s)}\de s\bigg), \  & y_2\in[0,\bar m],\\
 &\frac{\e u_{2,+}}{\e u_{1,+}}(y_1,\bar m)=\sigma g'(y_1),\  & y_1\in[0,L],\\
 &\e u_{2,+}(y_1,0)=0,\  &   y_1\in[0,L].
 \end{aligned}
 \end{cases}
\end{equation}
The second and third equations in \eqref{6-7} yield that
 \begin{equation}\label{6-9-be}
 \e B_+(y_1,y_2)=\e B_+(\psi(y_2),y_2)=\e B_-(\psi(y_2),y_2)=(\h B_-+ \sigma\e B_{en})(y_2),\ \  (y_1,y_2)\in\mn_+.
 \end{equation}
It is well-known that the system \eqref{6-7} is a nonlinear free boundary problem, where the unknown shock surface is a part of the boundary and should be determined
with the subsonic flow simultaneously. To fix the shock
front,  we introduce a coordinates transformation   to reduce the free boundary value problem into a fixed boundary value problem.
\par Define the following coordinates transformation
 \begin{equation*}
 \begin{cases}
 \begin{aligned}
& z_1=\frac{L-\bar \psi}{L- \psi(y_2)}(y_1-\psi(y_2))+\bar \psi,\\ &z_2=y_2,\end{aligned}
 \end{cases}
 \end{equation*}
 with its inverse
  \begin{equation*}
  \begin{cases}
   \begin{aligned}
 &y_1:=Y_1(z_1,z_2;\psi)=z_1+\frac{L- z_1}{L-\bar \psi}(\psi(z_2)-\bar \psi),\\ &z_2=y_2.\end{aligned}
  \end{cases}
 \end{equation*}
 Then the shock front $\Sigma_s $ and domain  $\mn_+$ are changed to be
 \begin{equation*}
 \begin{cases}
 \Sigma_s^z:=\{(z_1,z_2):z_1=\bar\psi, \ 0<z_2<\bar m\},\\
 \mn_+^z:=\{(z_1,z_2):\bar\psi<z_1<L, \ 0<z_2<\bar m\}.\\
 \end{cases}
 \end{equation*}
 Indeed,   $ \mn_+^z = \dot\mn_+$. Thus the equations in  \eqref{5-9} hold in $ \mn_+^z$.
 \par Set
 \begin{equation}\label{6-8}
\e{\bm {V}}_+^z(z_1,z_2):=(\e u_{1,+}^z,\e u_{2,+}^z,\e S_+^z,\e B_+^z)(z_1,z_2)=(\e u_{1,+},\e u_{2,+},\e S_+,\e B_+)(Y_1(z_1,z_2;\psi),y_2) \ \ {\rm{in}}\ \   \mn_+^z,\\
  \end{equation} where $\e B_+ $ is given in \eqref{6-9-be}.
  A direct computation yields that
  \begin{equation*}
  \begin{cases}
  \begin{aligned}
&  \p_{y_1}=\frac{L-\bar \psi}{L- \psi(y_2)}\p_{z_1}, \\ &\p_{y_2}=\p_{z_2}-\frac{L-z_1}{L- \psi(y_2)}\psi'(y_2)\p_{z_1}.\end{aligned}
  \end{cases}
  \end{equation*}
  Then the free boundary value problem \eqref{6-7} can be rewritten as \begin{equation}\label{6-9}
\begin{cases}
\begin{aligned}
& \left(1
 - (\e M_{1,+}^z)^2\right)\p_{z_1}\e u_{1,+}^z-\e M_{1,+}^z\e M_{2,+}^z\p_{z_1}\e u_{2,+}^z-\frac{\bar m}{m}\e\rho_+^z \e u_{2,+}^z\p_{z_2} \e u_{1,+}^z\\
 &+\frac{\bar m}{m}\e\rho_+^z \e u_{1,+}^z \p_{z_2}\e u_{2,+}^z=\mf_1(\e{\bm {V}}_+^z,\psi),\ \ &{\rm{in}} \ \mn_+^z,\\
 &\p_{z_1} \e u_{2,+}^z-\frac{\bar m}{m}\e \rho_+^z \e u_{2,+}^z\p_{z_2}\e u_{2,+}^z -\frac{\bar m}{m}\e\rho_+^z \e u_{1,+}^z\p_{z_2}\e u_{1,+}^z+\beta\\
 &=\frac{ e^{\e S_+^z}(\e\rho_+^z)^{\gamma}}{\gamma -1}   \frac{\bar m}{m}\p_{z_2}\e S_+^z-\frac{\bar m}{m}\e\rho_+^z  (\h B_-+ \sigma\e B_{en})'+\mf_2(\e{\bm {V}}_+^z,\psi),\ \ &{\rm{in}} \ \mn_+^z,\\
& \p_{z_1} \e S_+^z=0, \ \ &{\rm{in}} \ \mn_+^z,\\
 &G_i(\e{\bm V}_-,\e{\bm V}_+^z)=0, \ i=1,2,3, \ \ &{\rm{on}} \ \Sigma_s^z,\\
 &\e P_+^z(L,z_2)=\h P_+(z_2)+\sigma  P_{ex}\bigg(\frac m{\bar m}\int_{0}^{z_2}\frac{1}{(\e\rho_+^z \e u_{1,+}^z)(L,s)}\de s\bigg), \  & y_2\in[0,\bar m],\\
 &\frac{\e u_{2,+}^z}{\e u_{1,+}^z}(z_1,\bar m)=\sigma g'(Y_1(z_1,z_2;\psi)),\  & y_1\in[0,L],\\
 &\e u_{2,+}^z(z_1,0)=0,\  &   y_1\in[0,L],
 \end{aligned}
 \end{cases}
\end{equation}
where
\begin{equation*}
\begin{aligned}
\mf_1(\e{\bm {V}}_+^z,\psi)&=-\frac{\psi(z_2)-\bar \psi}{L-\psi(z_2)}\bigg(\left(1
 - (\e M_{1,+}^z)^2\right)\p_{z_1}\e u_{1,+}^z-\e M_{1,+}^z\e M_{2,+}^z\p_{z_1}\e u_{2,+}^z\bigg)\\
 &\quad+\frac{L-z_1}{L- \psi(y_2)}\psi'(y_2)\bigg(-\frac{\bar m}{m}\e\rho_+^z \e u_{2,+}^z\p_{z_1} \e u_{1,+}^z+\frac{\bar m}{m}\e\rho_+^z \e u_{1,+}^z \p_{z_1}\e u_{2,+}^z\bigg),\\
 \mf_2(\e{\bm {V}}_+^z,\psi)&=-\frac{\psi(z_2)-\bar \psi}{L-\psi(z_2)}\p_{z_1} \e u_{2,+}^z+\frac{L-z_1}{L- \psi(y_2)}\psi'(y_2)\bigg(-\frac{\bar m}{m}\e \rho_+^z \e u_{2,+}^z\p_{z_1}\e u_{2,+}^z \\
&\quad-\frac{\bar m}{m}\e\rho_+^z \e u_{1,+}^z\p_{z_1}\e u_{1,+}^z-\frac{e^{\e S_+^z} (\e\rho_+^z)^{\gamma}}{\gamma -1}   \frac{\bar m}{m}\p_{z_1}\e S_+^z\bigg).\\
 \end{aligned}
 \end{equation*}

 \subsection{The linearized system and iteration scheme}\noindent
 \par It should be noted the free boundary $ y_1=\psi(z_2)$ will be determined by two independent information: the shape of the shock-front $\psi'$ will be determined by the R-H conditions, and an exact position on the nozzle $\psi_\sharp := \psi(\bar m)$ where the shock-front passes through will be determined by the solvability condition for the existence of the solution $\e{\bm {V}}_+$. Therefore, we shall rewrite $\psi$  as below:
\begin{equation}\label{6-10}
\psi(z_2)=\bar\psi+\ddot\psi_\sharp-\int_{z_2}^{\bar m}\ddot\psi'(s)\de s,
 \end{equation}
 where \begin{equation*}
 \ddot\psi_\sharp:=\ddot\psi(\bar m)=\psi_\sharp-\bar\psi.
 \end{equation*}
 Set $\e{\bm {V}}_+^z=\h{\bm {V}}_++\ddot {\bm {V}}_+$. Given $(\ddot {\bm {V}}_+,\ddot\psi')$ which is close to the initial approximation $ (\dot{\bm V}_+;\dot \psi')$ constructed in Theorem \ref{th4}, we are going to determine $\ddot\psi_\sharp$ and $(\ddot {\bm {V}}_\ast,\ddot\psi_\ast')$ by solving the corresponding linearized problem, where
 $$\ddot {\bm {V}}_+=(\ddot u_{1,+},\ddot u_{2,+},\ddot S_+,\sigma\e B_{en}) \ \ {\rm{and}} \ \  \ddot {\bm {V}}_\ast=(\ddot u_{1,\ast},\ddot u_{2,\ast},\ddot S_\ast,\sigma\e B_{en}).$$
 \par  By  \eqref{5-shock} and \eqref{5-shock-1},  the  boundary condition on the fixed shock front $ z_1=\bar \psi $ is prescribed as
 \begin{equation}\label{6-12}
\bigg(\begin{array}{ccc}
\frac{\h M_{+}^2-1}{\h u_+} & \frac1{\gamma-1} \\
\frac{\h M_{+}^2-1}{\gamma\h M_+^2} &  0
 \end{array}\bigg)\bigg(\begin{array}{ccc}{\ddot u_{1,\ast}}\\
 {\ddot S_{\ast}}\end{array}\bigg)=\bigg(\begin{array}{ccc}{g_2}\\
 {g_1}\end{array}\bigg),
 \end{equation}
 where
 \begin{equation*}
 \bigg(\begin{array}{ccc}{g_2}\\
 {g_1}\end{array}\bigg)=
\bigg(\begin{array}{ccc}
\frac{\h M_{+}^2-1}{\h u_+} & \frac1{\gamma-1} \\
\frac{\h M_{+}^2-1}{\gamma\h M_+^2} &  0
 \end{array}\bigg)\bigg(\begin{array}{ccc}{\ddot u_{1,+}}\\
 {\ddot S_{+}}\end{array}\bigg)-\bigg(\begin{array}{ccc}{G_1(\e{\bm V}_-(\psi(z_2),z_2),\e{\bm V}_+)}\\
 {G_2(\e{\bm V}_-(\psi(z_2),z_2),\e{\bm V}_+)}\end{array}\bigg).
 \end{equation*}
 Note that
 \begin{equation*}
 {\rm{det}}\bigg(\begin{array}{ccc}
\frac{\h M_{\pm}^2-1}{\h u_\pm} & \frac1{\gamma-1} \\
\frac{\h M_{\pm}^2-1}{\gamma\h M_\pm^2} &  0
 \end{array}\bigg)=\frac{1-\h M_{\pm}^2}{(\gamma-1)\gamma\h M_\pm^2}\neq 0.
 \end{equation*}
 There holds
 \begin{equation}\label{6-13}
 \begin{cases}
 \ddot u_{1,\ast}(\bar \psi,z_2)=g_{2,\ast}(z_2,\ddot {\bm {V}}_+,\ddot {\bm {V}}_-,
 \ddot\psi';\ddot\psi_\sharp),  &z_2\in[0,\bar m],\\
 \ddot S_{\ast}(\bar \psi,z_2)=g_{1,\ast}(z_2,\ddot {\bm {V}}_+,\ddot {\bm {V}}_-,
 \ddot\psi';\ddot\psi_\sharp),  &z_2\in[0,\bar m],\\
 \end{cases}
  \end{equation}
  where
  \begin{equation*}
 \bigg(\begin{array}{ccc}{g_{2,\ast}}\\
 {g_{1,\ast}}\end{array}\bigg)=
\bigg(\begin{array}{ccc}
\frac{\h M_{+}^2-1}{\h u_+} & \frac1{\gamma-1} \\
\frac{\h M_{+}^2-1}{\gamma\h M_+^2} &  0
 \end{array}\bigg)^{-1}\bigg(\begin{array}{ccc}{g_2}\\
 {g_1}\end{array}\bigg).
 \end{equation*}
 Furthermore, it follows from the third equation in \eqref{6-9} that
 \begin{equation}\label{6-13-en}
  \ddot S_\ast(z_1,z_2)=g_{1,\ast}(z_2,\ddot {\bm {V}}_+,\ddot {\bm {V}}_-,
 \ddot\psi';\ddot\psi_\sharp),\ \  (z_1,z_2)\in\mn_+^z.
 \end{equation}
 \par Given $(\ddot {\bm {V}}_+,\ddot\psi')$, we solve the following linearized system:
 \begin{equation}\label{6-14}
\begin{cases}
\begin{aligned}
& (1
 - \h M_{+}^2)\p_{z_1}\ddot u_{1,\ast}-\h \rho_+\h u_+'  \ddot u_{2,\ast}+\h\rho_+ \h u_+ \p_{z_2}\ddot u_{2,\ast}=f_1(z_1,z_2,\ddot {\bm {V}}_+,\ddot\psi';\ddot\psi_\sharp),\ \ &{\rm{in}} \ \mn_+^z,\\
 &\p_{z_1} \ddot u_{2,\ast} -\h\rho_+ \h u_+ \p_{z_2}\ddot u_{1,\ast}+\bigg(-\h\rho_+ \h u_+'+
 \frac{ \beta\h u_+}{\h c_+^2}+\frac{\h\rho_+ \h u_+ \h S_+'}{\gamma}\bigg)\ddot u_{1,\ast}
 =f_{2}(z_1,z_2,\ddot {\bm {V}}_+,\ddot\psi';\ddot\psi_\sharp),\ \ &{\rm{in}} \ \mn_+^z,\\
\end{aligned}
\end{cases}
\end{equation}
where
\begin{equation*}
\begin{aligned}
&f_1(z_1,z_2,\ddot {\bm {V}}_+,\ddot\psi';\ddot\psi_\sharp)=-\frac{\ddot\psi_\sharp-\int_{z_2}^{\bar m}\ddot\psi'(s)\de s}{L-\psi(z_2)}\bigg(\left(1
 - (\e M_{1,+}^z)^2\right)\p_{z_1}\ddot u_{1,+}-\e M_{1,+}^z\e M_{2,+}^z\p_{z_1}\ddot u_{2,+}\bigg)\\
 &\quad+\frac{L-z_1}{L- \psi(z_2)}\ddot\psi'(z_2)\bigg(-\frac{\bar m}{m}\e\rho_+ \e u_{2,+}\p_{z_1} \ddot u_{1,+}+\frac{\bar m}{m}\e\rho_+ \e u_{1,+} \p_{z_1}\ddot u_{2,+}\bigg)\\
&\quad+ \left((\e M_{1,+}^z)^2-\h M_{+}^2\right)\p_{z_1} \ddot u_{1,+}+\e M_{1,+}^z\e M_{2,+}^z\p_{z_1}\ddot u_{2,+}+\bigg(\frac{\bar m}{m}-1\bigg)\e\rho_{+} \ddot u_{2,+}\p_{z_2} (\h u_{+}+\ddot u_{1,+})\\
&\quad+\bigg(\frac{\bar m}{m}-1\bigg)\e\rho_{+} (\h u_{+}+\ddot u_{1,+}) \p_{z_2} \ddot u_{2,+}
+\e\rho_{+} \ddot u_{2,+}\p_{z_2} (\h u_{+}+\ddot u_{1,+})-\h \rho_{+} \h u_{+}'\ddot u_{2,+}\\
&\quad-\e\rho (\h u_{+}+\ddot u_{1,+}) \p_{z_2}\ddot u_{2,+}
+\h \rho_{+} \h u_{+} \p_{z_2}\ddot u_{2,+},\ \ (z_1,z_2)\in \mn_+^z,\\
&f_{2}(z_1,z_2,\ddot {\bm {V}}_+,\ddot\psi';\ddot\psi_\sharp)=-\frac{\ddot\psi_\sharp-\int_{z_2}^{\bar m}\ddot\psi'(s)\de s}{L-\psi(z_2)}\p_{z_1} \ddot u_{2,+}+\frac{L-z_1}{L- \psi(z_2)}\ddot\psi'(z_2)\bigg(-\frac{\bar m}{m}\e \rho_+^z \ddot u_{2,+}\p_{z_1}\ddot u_{2,+} \\
&\quad-\frac{\bar m}{m}\e\rho_+^z (\h u_++\ddot u_{1,+})\p_{z_1}\ddot u_{1,+}-\frac{ e^{\e S_+^z}(\e\rho_+^z)^{\gamma}}{\gamma -1}   \frac{\bar m}{m}\p_{z_1}g_{1,\ast}\bigg)+\frac{\bar m}{m}\e \rho_{+}^z \ddot u_{2,+}\p_{z_2} \ddot u_{2,+}+\bigg(\frac{\bar m}{m}-1\bigg)\e\rho_+^z (\h u_{+} \p_{z_1}\ddot u_{1,+}\\
&\quad + \ddot u_{1,+}\h u_{+}'+ \ddot u_{1,+}\p_{y_2} \ddot u_{1,+})+\e\rho_+^z (\h u_{+} \p_{z_1}\ddot u_{1,+}+ \ddot u_{1,+}\h u_{+}'+ \ddot u_{1,+}\p_{z_2} \ddot u_{1,+})-\h\rho_{+}^z \h u_{+} \p_{z_2} \ddot u_{1,+}
 -\h\rho_{+}^z \h u_+'  \ddot u_{1,+}\\
 \end{aligned}\end{equation*}
\begin{equation*}
\begin{aligned}
&\quad+\frac{e^{\e S_+^z} (\e\rho_+^z)^{\gamma}}{\gamma -1}   \bigg(\frac{\bar m}{m}-1\bigg)\p_{z_2}g_{1,\ast}+\bigg(\frac{ e^{\e S_+^z}(\e\rho_+^z)^{\gamma}}{\gamma -1}-\frac{e^{\h S_+}\h\rho_+^{\gamma}}{\gamma -1}\bigg)\p_{z_2}g_{1,\ast}-\bigg(\frac{\bar m}{m}-1\bigg)\e\rho_+^z\sigma\e B_{en}'-(\e\rho_+^z-\h\rho_+)\sigma\e B_{en}'\\
  &\quad+\beta\bigg(\frac{\bar m}{m}-1\bigg)\bigg(\frac{\e\rho_{+}^z}{\h\rho_{+}}-1\bigg)
  +\beta\bigg(\frac{\e\rho_{+}^z}{\h\rho_{+}}-1+\frac{\h u_{+}}{\h c_{+}^2}\ddot u_{1,+}+\frac{g_{1,\ast}}{\gamma-1}-\frac{\sigma \e B_{en}}{\h c_{+}^2}\bigg)+\bigg(\frac{\bar m}{m}-1\bigg)\frac{\h S_{+}'}{\gamma-1}\e \rho_{+}^z\bigg(e^{\e S_+^z}(\e \rho_{+}^z)^{\gamma-1}
 \\
 &\quad-e^{\h S_+}\h \rho_{+}^{\gamma-1}\bigg)+\frac{\h S_{+}'}{\gamma-1}\e \rho_{+}^z\bigg(e^{\e S_+^z}(\e \rho_{+}^z)^{\gamma-1}-e^{\h S_+}\h \rho_{+}^{\gamma-1}\bigg)
 -\frac{\h S_{+}'}{\gamma-1}\h \rho_{+}^z\bigg(e^{\e S_+^z}(\e \rho_{+}^z)^{\gamma-1}-e^{\h S_+}\h \rho_{+}^{\gamma-1}\bigg)
\\
 &\quad +\frac{\h S_{+}'}{\gamma-1}\h \rho_{+}\bigg(
 e^{\e S_+^z} (\e \rho_{+}^z)^{\gamma-1}-e^{\h S_+}\h \rho_{+}^{\gamma-1}\bigg)+\frac{\h\rho_{+} \h u_{+} \h S_{+}'}{\gamma }\ddot u_{1,+}-\frac{\h\rho_{+}  \h S_{+}'}{\gamma }   \sigma \e B_{en}+\bigg(\frac{\beta}{\h c_{+}^2}+\frac{\h\rho_{+}  \h S_{+}'}{\gamma }\bigg)   \sigma \e B_{en}-\frac{\beta}{\gamma-1}g_{1,\ast}\\
&\quad +\frac{ e^{\h S_{+}}\h\rho_{+}^{\gamma}}{\gamma -1}\p_{z_2}g_{1,\ast}-\frac{\beta}{\gamma-1}g_{1,\ast}
-\h\rho_{+}\sigma \e B_{en}'+\bigg(\frac{\beta}{\h c_{+}^2}+\frac{\h\rho_{+}  \h S_{+}'}{\gamma }\bigg)\sigma \e B_{en}',\ \ (z_1,z_2)\in \mn_+^z.\\
   \end{aligned}
\end{equation*}
The boundary conditions on the nozzle walls are given by
\begin{equation}\label{6-15}
\begin{cases}
\begin{aligned}
&{\ddot u_{2,\ast}}(z_1,\bar m)=\sigma\bigg({\h u_{+}}(\bar m)+{\ddot u_{1,+}}(z_1,\bar m)\bigg)g'\bigg(z_1+\frac{L- z_1}{L-\bar \psi}\ddot\psi_\sharp\bigg):=g_{3,\ast}(z_1,\ddot {\bm {V}}_+
 ;\ddot\psi_\sharp),\  &z_1\in[\bar\psi,L],\\
 &\ddot u_{2,\ast}(z_1,0)=0,\  &z_1\in[\bar\psi,L].\\
 \end{aligned}
\end{cases}
\end{equation}
Furthermore, it follows from \eqref{5-15-ex} that the boundary condition on the exit is prescribed as
\begin{equation}\label{6-16}
{\ddot u_{1,\ast}}(L,z_2)=g_{4,\ast}(z_2,\ddot {\bm {V}}_+,\ddot {\bm {V}}_-,
 \ddot\psi';\ddot\psi_\sharp), \ \ z_2\in[0,\bar m],
 \end{equation}
 where
 \begin{equation*}
 \begin{aligned}
 &g_{4,\ast}(z_2;\ddot {\bm {V}}_+,\ddot {\bm {V}}_-,
 \ddot\psi';\ddot\psi_\sharp)
 =- \frac{\sigma}{(\h \rho_+ \h u_+)(z_2)}P_{ex}\bigg(\int_{0}^{z_2}\frac{1}{(\e\rho_+ (\h u_++ \ddot u_{1,+}))(L,s)}\de s\bigg)\\
&\quad -\frac{(\h P_+g_{1,\ast})(z_2)}{(\gamma-1)(\h \rho_+ \h u_+)(z_2)}+\frac{\sigma\e  B_{en}(z_2)}{ \h u_+(z_2)}+\frac{1}{(\h \rho_+ \h u_+)(z_2)}\bigg(\e P_+-\h P_++{\h \rho_+ \h u_+}\ddot u_{1,+}\\
 &\quad
 +\frac{\h P_+}{\gamma-1}
\ddot S_{+}-{\sigma\h \rho_+\e  B_{en}}\bigg)(L,z_2),\ \  z_2\in[0,\bar m].
\end{aligned}
 \end{equation*}
 \par In summary, the entropy is given by
  \begin{equation}\label{6-17-s}
 \ddot S_\ast=g_{1,\ast}(z_2,\ddot {\bm {V}}_+,\ddot {\bm {V}}_-,
 \ddot\psi';\ddot\psi_\sharp),\ \ {\rm{in}} \ \mn_+^z.\\
 \end{equation}
 the velocity field $(\ddot u_{1,\ast},\ddot u_{2,\ast})$ will be determined by solving the following boundary value problem:
 \begin{equation}\label{6-17}
\begin{cases}
\begin{aligned}
& \p_{z_1}( b_{1,+}(z_2)\ddot u_{1,\ast}) +\p_{z_2}( b_{2,+}(z_2)\ddot u_{2,\ast})=\frac{b_{2,+}(z_2)}{(\h\rho_+ \h u_+)(z_2)} f_1(z_1,z_2,\ddot {\bm {V}}_+,\ddot\psi';\ddot\psi_\sharp),\ \ &{\rm{in}} \ \mn_+^z,\\
 &\p_{z_1} ( b_{3,+}(z_2)\ddot u_{2,\ast}) - \p_{z_2}( b_{4,+}(z_2)\ddot u_{1,\ast})
 =b_{3,+}(z_2)f_2(z_1,z_2,\ddot {\bm {V}}_+,\ddot\psi';\ddot\psi_\sharp),\ \ &{\rm{in}} \ \mn_+^z,\\
 &\ddot u_{1,\ast}(\bar \psi,z_2)=g_{2,\ast}(z_2,\ddot {\bm {V}}_+,\ddot {\bm {V}}_-,
 \ddot\psi';\ddot\psi_\sharp),  &z_2\in[0,\bar m],\\
 &{\ddot u_{1,\ast}}(L,z_2)=g_{4,\ast}(z_2,\ddot {\bm {V}}_+,\ddot {\bm {V}}_-,
 \ddot\psi';\ddot\psi_\sharp),&z_2\in[0,\bar m],\\
 &\ddot u_{2,\ast}(z_1,0)=0,\  &z_1\in[\bar\psi,L],\\
 &{\ddot u_{2,\ast}}(z_1,\bar m)=g_{3,\ast}(z_1,\ddot {\bm {V}}_+
 ;\ddot\psi_\sharp),\  &z_1\in[\bar\psi,L],\\
\end{aligned}
\end{cases}
\end{equation}
where
\begin{equation*}
\begin{aligned}
& b_{1,+}(z_2)=\frac{1
 - \h M_{+}^2}{\h\rho_+ \h u_+} b_{2,+}, \ \  b_{2,+}(z_2)=\exp \left(\int_{0}^{z_2}-
  \frac{\h u_+'}{ \h u_+}(s)\de s\right),\  \  z_2\in[0,\bar m],\\
  & b_{3,+}(z_2)=\frac{
  1}{\h\rho_+ \h u_+}b_{4,+},\ \  b_{4,+}(z_2)=\exp \bigg( \int_{0}^{z_2}\bigg(\frac{ \h u_+'}{ \h u_+}-
 \frac{ \beta}{\h \rho_+\h c_+^2}-\frac{ \h S_+'}{\gamma e^{\h S_+}}\bigg)(s)\de s\bigg),\ \  z_2\in[0,\bar m].\\
 \end{aligned}
\end{equation*}
Furthermore, the shape of the shock front $\ddot\psi_\ast'$ is determined by
 \begin{equation}\label{6-18}
 \begin{aligned}
 \ddot\psi_\ast'(z_2)&=\frac { m }{\bar m[\h P](z_2)}\bigg(a_{0,+}\ddot{\bm V}_\ast(\bar\psi,z_2)+g_{0,\ast}(z_2,\ddot {\bm {V}}_+,\ddot {\bm {V}}_-,
 \ddot\psi';\ddot\psi_\sharp)\bigg), \ \
  z_2\in[0,\bar m],
  \end{aligned}
 \end{equation}
where
$$g_{0,\ast}(z_2,\ddot {\bm {V}}_+,\ddot {\bm {V}}_-,
 \ddot\psi';\ddot\psi_\sharp)=\frac {\bar m[\h P](z_2)}{ m }\ddot\psi'(z_2)-a_{0,+}\ddot{\bm V}_+(\bar\psi,z_2)+G_0(\e{\bm V}_-(\psi(z_2),z_2),\e{\bm V}_+^z(\bar\psi,z_2)),  \ \
  z_2\in[0,\bar m].$$
 \par To simplicity the notations, we denote
 \begin{equation*}
 \begin{aligned}
 \ddot f_i(z_1,z_2)&= f_i(z_1,z_2,\ddot {\bm {V}}_+,\ddot\psi';\ddot\psi_\sharp),\quad \ \ i=1,2, &(z_1,z_2)\in\mn_+^z,\\
 \ddot g_{j,\ast}(z_2)&= g_{j,\ast}(z_2,\ddot {\bm {V}}_+,\ddot {\bm {V}}_-,
 \ddot\psi';\ddot\psi_\sharp),\ \ j=0,1,2,4,&
  z_2\in[0,\bar m],\\
 \ddot g_{3,\ast}(z_1)& =g_{3,\ast}(z_1,\ddot {\bm {V}}_+
 ;\ddot\psi_\sharp),&  z_1\in[\bar\psi,L].
  \end{aligned}
 \end{equation*}
 \subsection{The well-posedness  for the linearized system}\noindent
 \par By Theorem \ref{th2}, there exists  a unique solution to \eqref{6-17}  if
 \begin{equation}\label{6-19}
 \begin{aligned}
&\int_0^{\bar m}(\ddot g_{4,\ast}-\ddot g_{2,\ast})(z_2)b_{1,+}(z_2)\de z_2
 +\int_{\bar \psi}^{L}b_{2,+}(\bar m)\ddot g_{3,\ast}(z_1)\de z_1
=\iint_{\mn_+^z} \frac{b_{2,+}(z_2)}{(\h\rho_+ \h u_+)(z_2)}\ddot f_1(z_1,z_2)\de z_1 \de z_2,
\end{aligned}
\end{equation}
which will be used to determine $\ddot\psi_\sharp$. If the condition \eqref{6-19} holds, it follows from \eqref{4-4} that
 \begin{equation}\label{4-4-sub}
\|(\ddot u_{1,\ast},\ddot u_{2,\ast})\|_{1,\alpha;\mn_+^z}^{(-\alpha)}\leq C\bigg(\|(\ddot f_1,\ddot f_2)\|_{0,\alpha;\mn_+^z}^{(1-\alpha)}
+\|(\ddot g_{2,\ast},\ddot g_{4,\ast})\|_{1,\alpha;(0,\bar m)}^{(-\alpha)}+\|\ddot g_{3,\ast}\|_{0,\alpha;[\bar \psi,L]}\bigg).
\end{equation}  This, together with  \eqref{6-13-en} and  \eqref{6-18}, yields that
\begin{equation}\label{4-5-sub}
\|\ddot\psi_\ast'\|_{1,\alpha;(0,\bar m)}^{(-\alpha)}\leq C\bigg(\|(\ddot f_1,\ddot f_2)\|_{0,\alpha;\mn_+^z}^{(1-\alpha)}
+\|(\ddot g_{0,\ast},\ddot g_{1,\ast},\ddot g_{2,\ast},\ddot g_{4,\ast})\|_{1,\alpha;(0,\bar m)}^{(-\alpha)}+\|\ddot g_{3,\ast}\|_{0,\alpha;[\bar \psi,L]}\bigg).\end{equation}
 Here  $C$ is a positive constant depending only on $(\h {\bm V}_\pm,L,\bar m)$ and $\alpha$.
 \par Next, we rewrite \eqref{6-19} as
\begin{equation}\label{6-20}
\mj(\ddot\psi_\sharp;\ddot {\bm {V}}_+,
 \ddot\psi',\ddot {\bm {V}}_-)=\mj_1+\mj_2+\mj_3=0,
 \end{equation}
 where
 \begin{gather*}
\mj_1=\int_{\bar \psi}^{L}b_{2,+}(\bar m)\ddot g_{3,\ast}(z_1)\de z_1,\ \
 \mj_2=\int_0^{\bar m}(\ddot g_{4,\ast}-\ddot g_{2,\ast})(z_2)b_{1,+}(z_2)\de z_2,\\
 \mj_3=
-\iint_{\mn_+^z} \frac{b_{2,+}(z_2)}{(\h\rho_+ \h u_+)(z_2)}\ddot f_1(z_1,z_2)\de z_1 \de z_2.
\end{gather*}
For any $\varepsilon>0$, we define the following space:
\begin{equation}\label{6-21}
\dot{\mathfrak N}_\varepsilon:=\bigg\{(\ddot {\bm {V}}_+,
 \ddot\psi'):\|\ddot {\bm {V}}_+-\dot {\bm {V}}_+\|_{1,\alpha;\mn_+^z}^{(-\alpha)}+
 \|\ddot\psi'-\dot\psi'\|_{1,\alpha;(0,\bar m)}^{(-\alpha)}\leq \varepsilon\bigg\}.
 \end{equation}
 \begin{lemma}\label{le-6-1}
 There exists a small constant $ \sigma_1>0$ such that for any $0<\sigma<\sigma_1$, if $(\ddot {\bm {V}}_+,
 \ddot\psi')\in \dot{\mathfrak N}_{\sigma^{3/2}}$, the equation \eqref{6-20} admits a unique solution $\ddot\psi_\sharp$ satisfying
 \begin{equation}\label{6-22}
 |\ddot\psi_\sharp|\leq \mc_1\sigma
 \end{equation}
 for some positive constant $\mc_1 $ depending only on $(\h {\bm V}_\pm,L,\bar m)$.
 \end{lemma}
\begin{proof}
It is easy to verify that
\begin{equation}\label{6-23}
\mj(0;0,
 0,\dot {\bm {V}}_-)=0.
  \end{equation}
  Indeed, \eqref{6-23} coincides with the equation \eqref{5-3-sub}. Next, we claim
\begin{equation}\label{6-24}
\frac{\p\mj}{\p \ddot\psi_\sharp}(0;0,
 0,\dot {\bm {V}}_-)\neq0.
  \end{equation}
  Then the implicit function implies that there exists a $\ddot\psi_\sharp$ to the equation \eqref{6-20}.
  \par To prove \eqref{6-24}, one needs to expand $\mj $ near $(0;0,
 0,\dot {\bm {V}}_-)$. Note that for any $(\ddot {\bm {V}}_+,
 \ddot\psi')\in \dot{\mathfrak N}_{\sigma^{3/2}}$, it follows from \eqref{5-6-sub-all} that
 \begin{equation}\label{6-25}
 \|\ddot {\bm {V}}_+\|_{1,\alpha;\mn_+^z}^{(-\alpha)}+
 \|\ddot\psi'\|_{1,\alpha;(0,\bar m)}^{(-\alpha)}\leq \sigma.
 \end{equation}
 In the following, we estimate $ \mj_i $. For the first term, one has
 \begin{equation}\label{6-26}
 \begin{aligned}
 \mj_1&=\sigma\int_{\bar \psi}^{L}b_{2,+}(\bar m)\bigg({\h u_{+}}(\bar m)+{\ddot u_{1,+}}(z_1,\bar m)\bigg)g'\bigg(z_1+\frac{L- z_1}{L-\bar \psi}\ddot\psi_\sharp\bigg)\de z_1\\
 &=\bigg(1+\frac{\ddot\psi_\sharp}{L-\bar \psi-\ddot\psi_\sharp}\bigg)\sigma\int_{\bar \psi+\ddot\psi_\sharp}^{L}b_{2,+}(\bar m)\bigg({\h u_{+}}(\bar m)+{\ddot u_{1,+}}\bigg(s+\frac{\ddot\psi_\sharp(s-L)}{L-\bar \psi-\ddot\psi_\sharp},\bar m\bigg)\bigg)g'(s)\bigg)\de s\\
 &=\bigg(1+\frac{\ddot\psi_\sharp}{L-\bar \psi-\ddot\psi_\sharp}\bigg)\sigma\int_{\bar \psi+\ddot\psi_\sharp}^{L}(b_{2,+}{\h u_{+}})(\bar m)g'(z_1)\de z_1+O(1)\sigma^2.
 \end{aligned}
  \end{equation}
\par  For the second term, it holds that
  \begin{equation*}
 \begin{aligned}
 \mj_2&=\int_0^{\bar m}(\ddot g_{4,\ast}-\ddot g_{2,\ast})(z_2)b_{1,+}(z_2)\de z_2\\
 &=\int_0^{\bar m}b_{1,+}(z_2) \frac{\sigma }{(\h \rho_+ \h u_+)(z_2) }P_{ex}\bigg(\int_{0}^{z_2}\frac{1}{(\e\rho_+ (\h u_++ \ddot u_{1,+})(L,s)}\de s\bigg)\de z_2\\
 &\quad-\int_0^{\bar m}\frac{(b_{1,+}\h P_+g_{1,\ast})(z_2)}{(\gamma-1)(\h \rho_+ \h u_+)(z_2)}\de z_2+\int_0^{\bar m}\frac{\sigma (b_{1,+}\e  B_{en})(z_2)}{ \h u_+(z_2)}\de z_2\\
&\quad
 +\int_0^{\bar m}\frac{1}{(\h \rho_+ \h u_+)(z_2)}\bigg(\e P_+-\h P_++{\h \rho_+ \h u_+}\ddot u_{1,+}
 +\frac{\h P_+}{\gamma-1}
\ddot S_{+}-{\sigma \h \rho_+\e  B_{en}}\bigg)(L,z_2)b_{1,+}(z_2)\de z_2\\
&\quad-\int_0^{\bar m}\ddot g_{2,\ast}(z_2)b_{1,+}(z_2)\de z_2=\mj_2^1+\mj_2^2+\mj_2^3+\mj_2^4+\mj_2^5.\\
\end{aligned}
  \end{equation*}
  Obviously,
  \begin{equation}\label{6-27}
  \mj_2^3=O(1)\sigma^2.
   \end{equation}
   For the first term $ \mj_2^1$, one  has
   \begin{equation*}
 \begin{aligned}
 & \mj_2^1=-\int_0^{\bar m}\bigg( \frac{\sigma}{(\h \rho_+ \h u_+)(z_2)}P_{ex}\bigg(\int_{0}^{z_2}\frac{1}{(\h\rho_+\h u_+)(s)}\de s\bigg)\bigg)b_{1,+}(z_2)\de z_2\\
 &\quad\quad -\int_0^{\bar m}\bigg( \frac{\sigma}{(\h \rho_+ \h u_+)(z_2)}P_{ex}\bigg(\int_{0}^{z_2}\frac{1}{(\e\rho_+ (\h u_++ \ddot u_{1,+}))(L,s)}\de s\bigg)\bigg)b_{1,+}(z_2)\de z_2\\
 \end{aligned}
  \end{equation*}
  \begin{equation}\label{6-28}
 \begin{aligned}
&\quad\quad +\int_0^{\bar m}\bigg( \frac{\sigma}{(\h \rho_+ \h u_+)(z_2)}P_{ex}\bigg(\int_{0}^{z_2}\frac{1}{(\h\rho_+\h u_+)(s)}\de s\bigg)\bigg)b_{1,+}(z_2)\de z_2\\
&=-\int_0^{\bar m}\bigg( \frac{\sigma}{(\h \rho_+ \h u_+)(z_2)}P_{ex}\bigg(\int_{0}^{z_2}\frac{1}{(\h\rho_+\h u_+)(s)}\de s\bigg)\bigg)b_{1,+}(z_2)\de z_2+O(1)\sigma^2.\\
 \end{aligned}
  \end{equation}
  For  $ \mj_2^2+\mj_2^5$, it follows from \eqref{6-12} that
  \begin{equation*}
  \begin{aligned}
 \bigg(\begin{array}{ccc}{g_2}\\
 {g_1}\end{array}\bigg)&=
\bigg(\begin{array}{ccc}
\frac{\h M_{+}^2-1}{\h u_+} & \frac1{\gamma-1} \\
\frac{\h M_{+}^2-1}{\gamma\h M_+^2} &  0
 \end{array}\bigg)\bigg(\begin{array}{ccc}{\ddot u_{1,+}}\\
 {\ddot S_{+}}\end{array}\bigg)-\bigg(\begin{array}{ccc}{G_1(\e{\bm V}_-(\psi(y_2),y_2),\e{\bm V}_+)}\\
 {G_2(\e{\bm V}_-(\psi(y_2),y_2),\e{\bm V}_+)}\end{array}\bigg)\\
 &=
\bigg(\begin{array}{ccc}
\frac{\h M_{+}^2-1}{\h u_+} & \frac1{\gamma-1} \\
\frac{\h M_{+}^2-1}{\gamma\h M_+^2} &  0
 \end{array}\bigg)\bigg(\begin{array}{ccc}{\ddot u_{1,+}}\\
 {\ddot S_{+}}\end{array}\bigg)-\bigg(\begin{array}{ccc}
\frac{\h M_{-}^2-1}{\h u_-} & \frac1{\gamma-1} \\
\frac{\h M_{-}^2-1}{\gamma\h M_-^2} &  0
 \end{array}\bigg)\bigg(\begin{array}{ccc}{\ddot u_{1,-}(\psi(z_2),z_2)}\\
 \sigma\e S_{en}(z_2)\end{array}\bigg)-\bigg(\begin{array}{ccc}{G_1(\e{\bm V}_-(\psi(y_2),y_2),\e{\bm V}_+)}\\
 {G_2(\e{\bm V}_-(\psi(y_2),y_2),\e{\bm V}_+)}\end{array}\bigg)\\
 &\quad+\bigg(\begin{array}{ccc}
\frac{\h M_{-}^2-1}{\h u_-} & \frac1{\gamma-1} \\
\frac{\h M_{-}^2-1}{\gamma\h M_-^2} &  0
 \end{array}\bigg)\bigg(\begin{array}{ccc}{\ddot u_{1,-}(\psi(z_2),z_2)-\dot u_{1,-}(\psi(z_2),z_2)}\\
 0\end{array}\bigg)\\
 &\quad+\bigg(\begin{array}{ccc}
\frac{\h M_{-}^2-1}{\h u_-} & \frac1{\gamma-1} \\
\frac{\h M_{-}^2-1}{\gamma\h M_-^2} &  0
 \end{array}\bigg)\bigg(\begin{array}{ccc}{\dot u_{1,-}(\psi(z_2),z_2)-\dot u_{1,-}(\bar \psi+\ddot\psi_\sharp,z_2)}\\
 0\end{array}\bigg)\\
 &\quad+\bigg(\begin{array}{ccc}
\frac{\h M_{-}^2-1}{\h u_-} & \frac1{\gamma-1} \\
\frac{\h M_{-}^2-1}{\gamma\h M_-^2} &  0
 \end{array}\bigg)\bigg(\begin{array}{ccc}{\dot u_{1,-}(\bar \psi+\ddot\psi_\sharp,z_2)}\\
 \sigma\e S_{en}(z_2)\end{array}\bigg)=\sum_{i=1}^4{\bf d}_i.\\
 \end{aligned}
 \end{equation*}
 Here we used $\ddot S_{-}(z_1,z_2)=\dot S_{-}(z_1,z_2)=\sigma\e S_{en}(z_2)$.
 Then it follows from \eqref{5-shock}  and \eqref{6-4}     that
 $$\sum_{i=1}^2{\bf d}_i=O(1)\sigma^2.$$
 Furthermore, one gets from \eqref{5-6} and  \eqref{6-25} that
\begin{equation*}
  \begin{aligned} {\bf d}_3&=\bigg(\begin{array}{ccc}
\frac{\h M_{-}^2-1}{\h u_-} & \frac1{\gamma-1} \\
\frac{\h M_{-}^2-1}{\gamma\h M_-^2} &  0
 \end{array}\bigg)\bigg(\begin{array}{ccc}{\dot u_{1,-}(\psi(z_2),z_2)-\dot u_{1,-}(\bar \psi+\ddot\psi_\sharp,z_2)}\\
0\end{array}\bigg)\\
&=\bigg(\begin{array}{ccc}
\frac{\h M_{-}^2-1}{\h u_-} & \frac1{\gamma-1} \\
\frac{\h M_{-}^2-1}{\gamma\h M_-^2} &  0
 \end{array}\bigg)\bigg(\begin{array}{ccc}{\int_{z_2}^{\bar m}\ddot\psi'(s)\de s\int_0^1\p_{z_1}\dot u_{1,-}(s\psi(z_2)+(1-s)\dot u_{1,-}(\bar \psi+\ddot\psi_\sharp,z_2)\de s}\\
0\end{array}\bigg)=O(1)\sigma^2.\end{aligned}
 \end{equation*}
 Combining the above estimates yields
  \begin{equation}\label{6-30-s}
  \begin{aligned}
 \bigg(\begin{array}{ccc}{g_{2,\ast}}\\
 {g_{1,\ast}}\end{array}\bigg)&=
\bigg(\begin{array}{ccc}
\frac{\h M_{+}^2-1}{\h u_+} & \frac1{\gamma-1} \\
\frac{\h M_{+}^2-1}{\gamma\h M_+^2} &  0
 \end{array}\bigg)^{-1}\bigg(\begin{array}{ccc}{g_2}\\
 {g_1}\end{array}\bigg)\\
 &=\bigg(\begin{array}{ccc}
\frac{\h M_{+}^2-1}{\h u_+} & \frac1{\gamma-1} \\
\frac{\h M_{+}^2-1}{\gamma\h M_+^2} &  0
 \end{array}\bigg)^{-1}\bigg(\begin{array}{ccc}
\frac{\h M_{-}^2-1}{\h u_-} & \frac1{\gamma-1} \\
\frac{\h M_{-}^2-1}{\gamma\h M_-^2} &  0
 \end{array}\bigg)\bigg(\begin{array}{ccc}{\dot u_{1,-}(\bar \psi+\ddot\psi_\sharp,z_2)}\\
 \sigma\e S_{en}(z_2)\end{array}\bigg)\\
 &=\bigg(\begin{array}{ccc}\mathfrak {a}_1(z_2),& 0\\
 \mathfrak {a}_2(z_2),& 1\end{array}\bigg)\bigg(\begin{array}{ccc}{\dot u_{1,-}(\bar \psi+\ddot\psi_\sharp,z_2)}\\
 \sigma\e S_{en}(z_2)\end{array}\bigg)+O(1)\sigma^2.\\
 \end{aligned}
 \end{equation}
 Thus
 \begin{equation}\label{6-30}
  \begin{aligned}
 \mj_2^2+\mj_2^5&=-\int_0^{\bar m}\frac{(b_{1,+}\h P_+)(z_2)}{(\gamma-1)(\h \rho_+ \h u_+)(z_2)}
\bigg(\mathfrak {a}_2(z_2)\dot u_{1,-}(\bar \psi+\ddot\psi_\sharp,z_2)+\sigma\e S_{en}(z_2)\bigg)\de z_2\\
&\quad-\int_0^{\bar m}b_{1,+}(z_2)\mathfrak {a}_1(z_2)\dot u_{1,-}(\bar \psi+\ddot\psi_\sharp,z_2)\de z_2+O(1)\sigma^2.
\end{aligned}
 \end{equation}
 By \eqref{6-27}-\eqref{6-30}, there holds
 \begin{equation}\label{6-31}
  \begin{aligned}
  \mj_2&=-\int_0^{\bar m}\bigg( \frac{\sigma}{(\h \rho_+ \h u_+)(z_2)}P_{ex}\bigg(\int_{0}^{z_2}\frac{1}{(\h\rho_+\h u_+)(s)}\de s\bigg)\bigg)b_{1,+}(z_2)\de z_2\\
  &\quad+\int_0^{\bar m}
\bigg(b_{1,+}(z_2)\mathfrak {a}_3(z_2)\dot u_{1,-}(\bar \psi+\ddot\psi_\sharp,z_2)-\frac{\sigma(b_{1,+}\h P_+\e S_{en})(z_2)}{(\gamma-1)(\h \rho_+ \h u_+)(z_2)}+\frac{\sigma(b_{1,+}\e B_{en})(z_2)}{\h u_+(z_2)}\bigg)\de z_2\\
&\quad-\int_0^{\bar m}b_{1,+}(z_2)\mathfrak {a}_1(z_2)\dot u_{1,-}(\bar \psi+\ddot\psi_\sharp,z_2)\de z_2+O(1)\sigma^2.
\end{aligned}
 \end{equation}
 \par For the last term, one derives
 \begin{equation}\label{6-32}
  \begin{aligned}
 \mj_3&=
-\iint_{\mn_+^z} \frac{b_{2,+}(z_2)}{(\h\rho_+ \h u_+)(z_2)} \ddot f_1(z_1,z_2)\de z_1 \de z_2\\
&=
-\iint_{\mn_+^z} \frac{b_{2,+}(z_2)}{(\h\rho_+ \h u_+)(z_2)}\bigg(-\frac{\ddot\psi_\sharp-\int_{z_2}^{\bar m}\ddot\psi'(s)\de s}{L-\psi(z_2)}\bigg(\left(1
 - (\e M_{1,+}^z)^2\right)\p_{z_1}\ddot u_{1,+}-\e M_{1,+}^z\e M_{2,+}^z\p_{z_1}\ddot u_{2,+}\bigg)\\
 &\qquad\qquad+\frac{L-z_1}{L- \psi(z_2)}\ddot\psi'(z_2)\bigg(-\frac{\bar m}{m}\e\rho_+ \e u_{2,+}\p_{z_1} \ddot u_{1,+}+\frac{\bar m}{m}\e\rho_+ \e u_{1,+} \p_{z_1}\ddot u_{2,+}\bigg)\\
&\qquad\qquad+ \left((\e M_{1,+}^z)^2-\h M_{+}^2\right)\p_{z_1} \ddot u_{1,+}+\e M_{1,+}^z\e M_{2,+}^z\p_{z_1}\ddot u_{2,+}+\bigg(\frac{\bar m}{m}-1\bigg)\e\rho_{+} \ddot u_{2,+}\p_{z_2} (\h u_{+}+\ddot u_{1,+})\\
&\qquad\qquad+\bigg(\frac{\bar m}{m}-1\bigg)\e\rho_{+} (\h u_{+}+\ddot u_{1,+}) \p_{z_2} \ddot u_{2,+}
+\e\rho_{+} \ddot u_{2,+}\p_{z_2} (\h u_{+}+\ddot u_{1,+})-\h \rho_{+} \h u_{+}'\ddot u_{2,+}\\
&\qquad\qquad-\e\rho (\h u_{+}+\ddot u_{1,+}) \p_{z_2}\ddot u_{2,+}
+\h \rho_{+} \h u_{+} \p_{z_2}\ddot u_{2,+}\bigg)\\
&=
\iint_{\mn_+^z} \frac{b_{2,+}(z_2)}{(\h\rho_+ \h u_+)(z_2)}\frac{\ddot\psi_\sharp}{L-\psi(z_2)}\bigg(\left(1
 - (\e M_{1,+}^z)^2(z_2)\right)\p_{z_1}\ddot u_{1,+}(z_1,z_2)\bigg)\de z_1\de z_2+O(1)\sigma^2.
\end{aligned}
 \end{equation}
One can decompose the first term in \eqref{6-32}  as
 \begin{equation*}
  \begin{aligned}
 &\frac{\ddot\psi_\sharp\left(1
 - (\e M_{1,+}^z)^2\right)}{L-\psi(z_2)}\p_{z_1}\ddot u_{1,+}\\
 &=\frac{\ddot\psi_\sharp\left(\h M_+^2
 - (\e M_{1,+}^z)^2\right)}{L-\psi(z_2)}\p_{z_1}\ddot u_{1,+}+\frac{\ddot\psi_\sharp\left(1-(\e M_{1,+}^z)^2\right)}{L-\psi(z_2)}\p_{z_1}(\ddot u_{1,+}-\dot u_{1,+})\\
&\quad+\frac{\ddot\psi_\sharp(1-\h M_+^2)}{L-\psi(z_2)}\p_{z_1}\dot u_{1,+}-\frac{\ddot\psi_\sharp(1-\h M_+^2)}{L-\psi_\sharp}\p_{z_1}\dot u_{1,+}+\frac{\ddot\psi_\sharp(1-\h M_+^2)}{L-\psi_\sharp}\p_{z_1}\dot u_{1,+}\\
 &=\frac{\ddot\psi_\sharp\left(\h M_+^2
 - (\e M_{1,+}^z)^2\right)}{L-\psi(z_2)}\p_{z_1}\ddot u_{1,+}+\frac{\ddot\psi_\sharp(1-\h M_+^2)}{L-\psi(z_2)}\p_{z_1}(\ddot u_{1,+}-\dot u_{1,+})\\
 &\quad+\frac{-\ddot\psi_\sharp\int_{z_2}^{\bar m}\ddot\psi'(s)\de s(1-\h M_+^2)}{(L-\psi(z_2))(L-\psi_\sharp)} \p_{z_1}\dot u_{1,+} +\frac{\ddot\psi_\sharp(1-\h M_+^2)}{L-\psi_\sharp}\p_{z_1}\dot u_{1,+}\\
 &=O(1)\sigma^2\ddot\psi_\sharp+O(1)\sigma^{\frac32}\ddot\psi_\sharp+\frac{\ddot\psi_\sharp(1-\h M_+^2)}{L-\psi_\sharp}\p_{z_1}\dot u_{1,+}.\\
 \end{aligned}
 \end{equation*}
 Combining the first equation in \eqref{5-9} and \eqref{5-10} yields that
 \begin{equation*}
  \begin{aligned}
 &\iint_{\mn_+^z} \frac{b_{2,+}(z_2)}{(\h\rho_+ \h u_+)(z_2)}\frac{\ddot\psi_\sharp}{L-\psi_\sharp}\bigg((1
 - \h M_{1,+}^2(z_2))\p_{z_1}\dot u_{1,+}(z_1,z_2)\bigg)\de z_1\de z_2\\
 &=\frac{\ddot\psi_\sharp}{L-\psi_\sharp}\iint_{\mn_+^z} \p_{z_1}\bigg(b_{1,+}(z_2)\dot u_{1,+}(z_1,z_2)\bigg)\de z_1\de z_2\\
 &=-\frac{\ddot\psi_\sharp}{L-\psi_\sharp}\iint_{\mn_+^z} \p_{z_2}\bigg(b_{2,+}(z_2)\dot u_{2,+}(z_1,z_2)\bigg)\de z_1\de z_2\\
 &=-\frac{\sigma\ddot\psi_\sharp}{L-\psi_\sharp}\int_{\bar \psi}^{L}(b_{2,+}{\h u_{+}})(\bar m)g'(z_1)\de z_1=-\frac{\sigma\ddot\psi_\sharp}{L-\psi_\sharp}(b_{2,+}{\h u_{+}})(\bar m)(g(L)-g(\bar \psi)).
 \end{aligned}
 \end{equation*}
 Thus it holds that
 \begin{equation}\label{6-33}
  \begin{aligned}
 \mj_3=O(1)\sigma^2\ddot\psi_\sharp+O(1)\sigma^{\frac32}\ddot\psi_\sharp
 -\frac{\sigma\ddot\psi_\sharp}{L-\psi_\sharp}\int_{\bar \psi}^{L}(b_{2,+}{\h u_{+}})(\bar m)g'(z_1)\de z_1.
 \end{aligned}
 \end{equation}
 Collecting \eqref{6-26}, \eqref{6-31} and \eqref{6-33}, one obtains
 \begin{equation*}
  \begin{aligned}
  &\mj(\ddot\psi_\sharp;\ddot {\bm {V}}_+,
 \ddot\psi',\ddot {\bm {V}}_-)\\
& =\bigg(1+\frac{\ddot\psi_\sharp}{L-\bar \psi-\ddot\psi_\sharp}\bigg)\sigma\int_{\bar \psi+\ddot\psi_\sharp}^{L}(b_{2,+}{\h u_{+}})(\bar m)g'(z_1)\de z_1\\
 &\quad-\int_0^{\bar m}\bigg( \frac{\sigma}{(\h \rho_+ \h u_+)(z_2)}P_{ex}\bigg(\int_{0}^{z_2}\frac{1}{(\h\rho_+\h u_+)(s)}\de s\bigg)\bigg)b_{1,+}(z_2)\de z_2\\
  &\quad+\int_0^{\bar m}
\bigg(b_{1,+}(z_2)\mathfrak {a}_3(z_2)\dot u_{1,-}(\bar \psi+\ddot\psi_\sharp,z_2)-\frac{\sigma(b_{1,+}\h P_+\e S_{en})(z_2)}{(\gamma-1)(\h \rho_+ \h u_+)(z_2)}+\frac{\sigma(b_{1,+}\e B_{en})(z_2)}{\h u_+(z_2)}\bigg)\de z_2\\
&\quad-\int_0^{\bar m}b_{1,+}(z_2)\mathfrak {a}_1(z_2)\dot u_{1,-}(\bar \psi+\ddot\psi_\sharp,z_2)\de z_2\\
&\quad-
\frac{\sigma\ddot\psi_\sharp}{L-\psi_\sharp}\int_{\bar \psi}^{L}(b_{2,+}{\h u_{+}})(\bar m)g'(z_1)\de z_1
+O(1)\sigma^2+O(1)\sigma^2\ddot\psi_\sharp+O(1)\sigma^{\frac32}\ddot\psi_\sharp.
\end{aligned}
 \end{equation*}
 Note that $\ddot\psi(\bar m)=\psi_\sharp=\bar\psi+\ddot\psi_\sharp$. Thus one has
 \begin{equation*}
 \begin{aligned}
 &\frac{\sigma\ddot\psi_\sharp}{L-\psi_\sharp}\int_{\bar \psi+\ddot\psi_\sharp}^{L}(b_{2,+}{\h u_{+}})(\bar m)g'(z_1)\de z_1-
\frac{\sigma\ddot\psi_\sharp}{L-\psi_\sharp}\int_{\bar \psi}^{L}(b_{2,+}{\h u_{+}})(\bar m)g'(z_1)\de z_1\\
&=\frac{\sigma\ddot\psi_\sharp}{L-\psi_\sharp}\int_{\bar \psi}^{\bar \psi+\ddot\psi_\sharp}(b_{2,+}{\h u_{+}})(\bar m)g'(z_1)\de z_1.
\end{aligned}
 \end{equation*}
 Therefore, there holds
 \begin{equation*}
  \begin{aligned}
  &\mj(\ddot\psi_\sharp;\ddot {\bm {V}}_+
 \ddot\psi',\ddot {\bm {V}}_-)\\& =\sigma\int_{\bar \psi+\ddot\psi_\sharp}^{L}(b_{2,+}{\h u_{+}})(\bar m)g'(z_1)\de z_1+\frac{\sigma\ddot\psi_\sharp}{L-\psi_\sharp}\int_{\bar \psi}^{\bar \psi+\ddot\psi_\sharp}(b_{2,+}{\h u_{+}})(\bar m)g'(z_1)\de z_1\\
 &\quad-\int_0^{\bar m}\bigg( \frac{\sigma}{(\h \rho_+ \h u_+)}P_{ex}\bigg(\int_{0}^{z_2}\frac{1}{(\h\rho_+\h u_+)(s)}\de s\bigg)\bigg)b_{1,+}(z_2)\de z_2\\
 \end{aligned}
  \end{equation*}
  \begin{equation}\label{6-34}
 \begin{aligned}
  &\quad+\int_0^{\bar m}
\bigg(b_{1,+}(z_2)\mathfrak {a}_3(z_2)\dot u_{1,-}(\bar \psi+\ddot\psi_\sharp,z_2)-\frac{\sigma(b_{1,+}\h P_+\e S_{en})(z_2)}{(\gamma-1)(\h \rho_+ \h u_+)(z_2)}+\frac{\sigma(b_{1,+}\e B_{en})(z_2)}{\h u_+(z_2)}\bigg)\de z_2\\
&\quad-\int_0^{\bar m}b_{1,+}(z_2)\mathfrak {a}_1(z_2)\dot u_{1,-}(\bar \psi+\ddot\psi_\sharp,z_2)\de z_2
+O(1)\sigma^2+O(1)\sigma^2\ddot\psi_\sharp+O(1)\sigma^{\frac32}\ddot\psi_\sharp.
\end{aligned}
 \end{equation}
Then it follows from  Lemma \ref{le6} that
 \begin{equation}\label{6-35}
  \begin{aligned}
  \frac{\p\mj}{\p \ddot\psi_\sharp}(0;0,
 0,\dot {\bm {V}}_-)&=\int_0^{\bar m}\bigg(\big(\mathfrak {a}_3-\mathfrak {a}_1\big)(z_2)b_{1,+}(z_2)\p_{\bar\psi}{\dot u_{1,-}}(\bar\psi,z_2)\bigg)\de z_2\\
 &\quad- \sigma(b_{2,+}{\h u_{+}})(\bar m)g'(\bar\psi)
+O(1)\sigma^2+O(1)\sigma^{\frac32}\\
&=\sigma J_1'(\bar\psi)+O(1)\sigma^2+O(1)\sigma^{\frac32}\neq 0.\\
 \end{aligned}
 \end{equation}
 The above  $O(1)$ is a constant depending only on $(\h {\bm V}_\pm,L,\bar m)$.
 \par By the implicit function theorem, for sufficiently small constant $\sigma$, one can obtain that there exists a $ \ddot\psi_\sharp $ such that \eqref{6-20} holds. Furthermore, \eqref{6-34} implies that the estimate \eqref{6-22}  holds. The proof is completed.
\end{proof}
\par In summary, we have the following theorem.
\begin{theorem}\label{th-6-2}
For any $(\ddot {\bm {V}}_+,
 \ddot\psi')\in \dot{\mathfrak N}_{\sigma^{3/2}}$, let $\ddot\psi_\sharp$ be given in Lemma \ref{le-6-1}. Then the linear boundary problem
 \eqref{6-17-s}-\eqref{6-18} admits a unique solution $(\ddot u_{1,\ast},\ddot u_{2,\ast},\ddot S_\ast;\ddot\psi_\ast')$ satisfying
 \begin{equation}\label{4-5-sub-sub}
 \begin{aligned}
&\|(\ddot u_{1,\ast},\ddot u_{2,\ast},\ddot S_\ast)\|_{1,\alpha;\mn_+^z}^{(-\alpha)}+\|\ddot\psi_\ast'\|_{1,\alpha;(0,\bar m)}^{(-\alpha)}\\
&\leq C\bigg(\|(\ddot f_1,\ddot f_2)\|_{0,\alpha;\mn_+^z}^{(1-\alpha)}
+\|(\ddot g_{0,\ast},\ddot g_{1,\ast},\ddot g_{2,\ast},\ddot g_{4,\ast})\|_{1,\alpha;(0,\bar m)}^{(-\alpha)}+\|\ddot g_{3,\ast}\|_{0,\alpha;[\bar \psi,L]}\bigg) \end{aligned}\end{equation}
   for some constant $C>0$  depending only on $(\h {\bm V}_\pm,L,\bar m)$ and $\alpha$.
\end{theorem}
\subsection{The contractive mapping}\noindent
\par Define a mapping
\begin{equation}\label{6-36}
\mt:(\ddot {\bm {V}}_+,
 \ddot\psi')\mapsto(\ddot {\bm {V}}_\ast,\ddot\psi_\ast').
 \end{equation}
 We will show that for sufficiently small $\sigma>0 $, $\mt$ maps the space $ \dot{\mathfrak N}_{\sigma^{3/2}}$ into itself.
 \begin{lemma}\label{lemma-1}
 Under the assumptions in Theorem \ref{th-6-2}, there exists a constant $ 0<\sigma_2<\sigma_1$ such that for any $ \sigma\in(0,\sigma_2)$, the mapping $\mt$ is well-defined in  $ \dot{\mathfrak N}_{\sigma^{3/2}}$.
 \end{lemma}
 \begin{proof}
 Note that
 $$\ddot {\bm {V}}_\ast-\dot {\bm {V}}_+=(\ddot u_{1,\ast}-\dot u_{1,+},\ddot u_{2,\ast}-\dot u_{2,+},\ddot S_\ast-\dot S_+,0) .$$
 By \eqref{5-shock-2}, we set
 $$\dot g_1(z_2)=\mathfrak {a}_2(z_2){\dot u_{1,-}}(\bar\psi,z_2)+\sigma \e S_{en}(z_2),\ \ z_2\in[0,\bar m].$$
  Then the difference of the entropy satisfies
 \begin{equation}\label{6-37}
 \begin{aligned}
 (\ddot S_\ast-\dot S_+)(z_1,z_2)&=
 (\ddot g_{1,\ast}-\dot g_1)(z_2)
 =\mathfrak {a}_2(z_2)\bigg(\dot u_{1,-}(\bar \psi+\ddot\psi_\sharp,z_2)-{\dot u_{1,-}}(\bar\psi,z_2)\bigg)+O(1)\sigma^2\\
 &=\mathfrak {a}_2(z_2)\ddot\psi_\sharp\int_0^1\p_{z_1}\dot u_{1,-}\bigg(s(\bar \psi+\ddot\psi_\sharp)+(1-s)\bar\psi,z_2\bigg)\de s\\
 &\quad+O(1)\sigma^2,\ \ (z_1,z_2)\in\mn_+^z,
 \end{aligned}
  \end{equation}
from which one gets
  \begin{equation}\label{6-37-dif}
 \begin{aligned}
 \|\ddot S_\ast-\dot S_+\|_{1,\alpha;\mn_+^z}^{(-\alpha)}\leq C\sigma^2.
 \end{aligned}
  \end{equation}
  Next, it follows from \eqref{5-16} and \eqref{6-17} that $(\ddot u_{1,\ast}-\dot u_{1,+},\ddot u_{2,\ast}-\dot u_{2,+})$ satisfies
   \begin{equation}\label{6-17-df}
\begin{cases}
\begin{aligned}
& \p_{z_1}( b_{1,+}(z_2)(\ddot u_{1,\ast}-\dot u_{1,+})) +\p_{z_2}( b_{2,+}(z_2)\ddot (u_{2,\ast}-\dot u_{2,+}))
=\frac{b_{2,+}(z_2)}{(\h\rho_+ \h u_+)(z_2)} \ddot  f_1, &{\rm{in}} \ \mn_+^z,\\
 &\p_{z_1} ( b_{3,+}(z_2)(u_{2,\ast}-\dot u_{2,+})) - \p_{z_2}( b_{4,+}(z_2)(\ddot u_{1,\ast}-\dot u_{1,+}))
 =b_{3,+}(z_2)\bigg(\ddot f_2
 -\dot f_2\bigg), &{\rm{in}} \ \mn_+^z,\\
 &({\ddot u_{1,\ast}}-{\dot u_{1,+}})(\bar \psi,z_2)=(\ddot g_{2,\ast}-\dot g_{2})(z_2),  &z_2\in[0,\bar m],\\
 &({\ddot u_{1,\ast}}-{\dot u_{1,+}})(L,z_2)=(\ddot g_{4,\ast}-\dot g_{4})(z_2),&z_2\in[0,\bar m],\\
 &({\ddot u_{2,\ast}}-{\dot u_{2,+}})(z_1,0)=0,\  &z_1\in[\bar\psi,L],\\
 &({\ddot u_{2,\ast}}-{\dot u_{2,+}})(z_1,\bar m)=(\ddot g_{3,\ast}-\dot g_{3})(z_1),\  &z_1\in[\bar\psi,L],\\
\end{aligned}
\end{cases}
\end{equation}
where
 \begin{equation*}
\begin{aligned}
&\dot f_2(z_1,z_2)=\bigg(\frac{e^{\h S_+} \h\rho_+^{\gamma}}{\gamma -1}\p_{z_2}\dot  S_+-\frac{\beta}{\gamma-1}\dot  S_+-\sigma\h\rho_{-}   \e  B_{en}'+\bigg(\frac{\beta}{\h c_{+}^2}+\frac{\h\rho_{+}  \h S_{+}'}{\gamma }\bigg)   \sigma  \e  B_{en}\bigg), &(z_1,z_2)\in\mn_+^z,\\
&\dot g_{2}(z_2)=\mathfrak {a}_1(z_2){\dot u_{1,-}}(\bar\psi,z_2),&z_2\in[0,\bar m],\\
&\dot g_{4}(z_2)=- \frac{\sigma}{(\h \rho_+ \h u_+)(z_2)}P_{ex}\bigg(\int_{0}^{z_2}\frac{1}{(\bar\rho_+ \bar u_+)(s)}\de s\bigg)-\frac{(\h P_+\dot g_{1})(z_2)}{(\gamma-1)(\h \rho_+ \h u_+)(z_2)}+\frac{\sigma\e  B_{en}(z_2)}{ \h u_+(z_2)},& z_2\in[0,\bar m],\\
&\dot g_{3}(z_1)=\sigma{\h u_{+}}(\bar m)g'(z_1), & z_1\in[\bar\psi,L].
\end{aligned}\end{equation*}
A direct computation shows that
 \begin{equation}\label{6-17-diff}
\begin{aligned}
\|(\ddot f_1,\ddot f_2)-(0,\dot f_2)\|_{0,\alpha;\mn_+^z}^{(1-\alpha)}+
\|(\ddot g_{2,\ast}-\dot g_2,\ddot g_{4,\ast}-\dot g_4)\|_{1,\alpha;(0,\bar m)}^{(-\alpha)}+\|\ddot g_{3,\ast}-\dot g_3\|_{0,\alpha;[\bar \psi,L]}\leq C\sigma^2.
\end{aligned}
\end{equation}
This, together with \eqref{4-5-sub-sub}, yields that
\begin{equation}\label{6-17-diff-1}
\begin{aligned}
\|(\ddot u_{1,\ast}-\dot u_{1,+},\ddot u_{2,\ast}-\dot u_{2,+})\|_{1,\alpha;\mn_+^z}^{(-\alpha)}\leq C\sigma^2.
\end{aligned}
\end{equation}
Finally, the difference of shock front is given by
\begin{equation}\label{6-18-s}
\begin{aligned}
 (\ddot\psi_\ast'-\dot \psi')(z_2)&=\frac { m }{\bar m[\h P](z_2)}\bigg(a_{0,+}(\ddot{\bm V}_\ast-\dot{\bm V}_+)(\bar\psi,z_2)+\ddot g_{0,\ast}(z_2)-a_{0,-}\dot{\bm V}_-(\bar \psi,z_2)\bigg)\\
 &=\frac { m }{\bar m[\h P](z_2)}\bigg(a_{0,+}(\ddot{\bm V}_\ast-\dot{\bm V}_+)(\bar\psi,z_2)+\frac {\bar m[\h P](z_2)}{ m }\ddot\psi'(z_2)-a_{0,+}\ddot{\bm V}_+(\bar\psi,z_2)\\
 &\qquad\qquad\quad-a_{0,-}\ddot{\bm V}_-(\psi(z_2),z_2)+G_0(\e{\bm V}_-(\psi(z_2),z_2),\e{\bm V}_+^z(\bar\psi,z_2))\\
&\qquad\qquad\quad +a_{0,-}(\ddot{\bm V}_--\dot{\bm V}_-)(\psi(z_2),z_2)\\
&\qquad\qquad\quad+a_{0,-}(\dot{\bm V}_-(\psi(z_2),z_2)-\dot{\bm V}_-(\bar\psi+\ddot\psi_\sharp,z_2))\\
&\qquad\qquad\quad+a_{0,-}(\dot{\bm V}_-(\bar\psi+\ddot\psi_\sharp,z_2)-\dot{\bm V}_-(\bar\psi,z_2)
\bigg),\ \ z_2\in[0,\bar m].\\
 \end{aligned}
 \end{equation}
 Thus there holds
 \begin{equation}\label{6-18-s-diff}
\|\ddot\psi_\ast'-\dot \psi'\|_{1,\alpha;(0,\bar m)}^{(-\alpha)}\leq C\sigma^2.
\end{equation}
 The above constant $C>0$  depends only on $(\h {\bm V}_\pm,L,\bar m)$ and $\alpha$. Collecting the estimates \eqref{6-37-dif}, \eqref{6-17-diff-1} and \eqref{6-18-s-diff}, one can conclude that for sufficiently small $\sigma$,
$$\|\ddot {\bm {V}}_\ast-\dot {\bm {V}}_+\|_{1,\alpha;\mn_+^z}^{(-\alpha)}+ \|\ddot\psi_\ast'-\dot \psi'\|_{1,\alpha;(0,\bar m)}^{(-\alpha)}\leq C\sigma^2\leq \sigma^{\frac32}.$$
 \end{proof}
 \par In the following, we show that the mapping $\mt $ is contractive.
 \begin{lemma}\label{l0}
 There exists a constant $0<\sigma_3<\sigma_2$ such that for any $ \sigma\in(0,\sigma_3)$, the mapping $\mt$ defined in \eqref{6-36} is contractive.
 \end{lemma}
 \begin{proof}
 Let $ (\ddot {\bm {V}}^j,
 (\ddot\psi^j)')\in  \dot{\mathfrak N}_{\sigma^{3/2}}$, $j=1,2$. Lemma \ref{le-6-1} guarantees the unique existence of the corresponding  $\ddot\psi_{\sharp}^j$.
 Then there exists a unique solution define $ (\ddot {\bm {V}}_{\ast}^j,(\ddot\psi_{\ast}^j)')$ to the linear boundary problem
 \eqref{6-17-s}-\eqref{6-18}, where $(\ddot {\bm {V}}_+,\ddot\psi';\ddot\psi_\sharp) $ is replaced by
 $ (\ddot {\bm {V}}_{+}^j,
 (\ddot\psi^j)';\ddot\psi_{\sharp}^j)$, respectively.
 \par Next, it follows from \eqref{6-37} that
 \begin{equation}\label{6-37-dd}
 \begin{aligned}
 (\ddot S_{\ast}^1- \ddot S_{\ast}^2)(z_1,z_2)&=(\ddot g_{1,\ast}^1-\ddot g_{1,\ast}^2)(z_2),\ \ (z_1,z_2)\in\mn_+^z.\\
 \end{aligned}
  \end{equation}
  By \eqref{6-17-df},  $(\ddot u_{1,\ast}^1-\ddot u_{1,\ast}^2,\ddot u_{2,\ast}^1-\ddot u_{2,\ast}^2)$ satisfies
   \begin{equation}\label{6-17-df-did}
\begin{cases}
\begin{aligned}
& \p_{z_1}( b_{1,+}(z_2)(\ddot u_{1,\ast}^1-\ddot u_{1,\ast}^2)) +\p_{z_2}( b_{2,+}(z_2) (\ddot u_{2,\ast}^1-\ddot u_{2,\ast}^2))
=\frac{b_{2,+}(z_2)}{(\h\rho_+ \h u_+)(z_2)} (\ddot f_1^1-\ddot f_1^2), &{\rm{in}} \ \mn_+^z,\\
 &\p_{z_1} ( b_{3,+}(z_2)(\ddot u_{2,\ast}^1-\ddot u_{2,\ast}^2)) - \p_{z_2}( b_{4,+}(z_2)(\ddot u_{1,\ast}^1-\ddot u_{1,\ast}^2))
 =b_{3,+}(z_2)(\ddot f_2^1-\ddot f_2^2), &{\rm{in}} \ \mn_+^z,\\
 &(\ddot u_{1,\ast}^1-\ddot u_{1,\ast}^2)(\bar \psi,z_2)=(\ddot g_{2,\ast}^1-\ddot g_{2,\ast}^2)(z_2),  &z_2\in[0,\bar m],\\
 &(\ddot u_{1,\ast}^1-\ddot u_{1,\ast}^2)(L,z_2)=(\ddot g_{4,\ast}^1-\ddot g_{4,\ast}^2)(z_2),&z_2\in[0,\bar m],\\
 &(\ddot u_{2,\ast}^1-\ddot u_{2,\ast}^2)(z_1,0)=0,\  &z_1\in[\bar\psi,L],\\
 &(\ddot u_{2,\ast}^1-\ddot u_{2,\ast}^2)(z_1,\bar m)=(\ddot g_{3,\ast}^1-\ddot g_{3,\ast}^2)(z_1),\  &z_1\in[\bar\psi,L],\\
\end{aligned}
\end{cases}
\end{equation}
where $ (\ddot f_1^j,\ddot f_2^j,\ddot g_{1,\ast}^j, \ddot g_{2,\ast}^j,\ddot g_{3,\ast}^j,\ddot g_{4,\ast}^j)$ $ (j=1,2)$ are defined as in \eqref{6-17-s} and \eqref{6-17} with $(\ddot {\bm {V}}_+,\ddot\psi';\ddot\psi_\sharp) $  replaced by
 $ (\ddot {\bm {V}}_{+}^j,
 (\ddot\psi^j)';\ddot\psi_{\sharp}^j)$.  Furthermore, it follows from \eqref{6-18-s} that
 \begin{equation}\label{6-18-s-d}
\begin{aligned}
 \left((\ddot\psi_\ast^1)'-(\ddot\psi_\ast^2)'\right)(z_2)&=\frac { m }{\bar m[\h P](z_2)}\bigg(a_{0,+}(\ddot{\bm V}_\ast^1-\ddot{\bm V}_\ast^2)(\bar\psi,z_2)+(\ddot g_{0,\ast}^1-\ddot g_{0,\ast}^2)(z_2)\bigg),\ \ z_2\in[0,\bar m].
 \end{aligned}
\end{equation}
Then the estimate \eqref{4-5-sub-sub} implies that
\begin{equation}\label{4-5-sub-sub-1 that}
 \begin{aligned}
&\|(\ddot u_{1,\ast}^1-\ddot u_{1,\ast}^2,\ddot u_{2,\ast}^1-\ddot u_{2,\ast}^2,\ddot S_\ast^1-\ddot S_\ast^2)\|_{1,\alpha;\mn_+^z}^{(-\alpha)}+\|(\ddot\psi_\ast^1)'
-(\ddot\psi_\ast^2)'\|_{1,\alpha;(0,\bar m)}^{(-\alpha)}\\
&\leq C\bigg(\|(\ddot f_1^1-\ddot f_1^2,\ddot f_2^1-\ddot f_2^2)\|_{0,\alpha;\mn_+^z}^{(1-\alpha)}
+\|(\ddot g_{0,\ast}^1-\ddot g_{0,\ast}^2,\ddot g_{1,\ast}^1-\ddot g_{1,\ast}^2,\ddot g_{2,\ast}^1-\ddot g_{2,\ast}^2,\ddot g_{4,\ast}^1-\ddot g_{4,\ast}^2)\|_{1,\alpha;(0,\bar m)}^{(-\alpha)}\\
&\qquad+\|\ddot g_{3,\ast}^1-\ddot g_{3,\ast}^2\|_{0,\alpha;[\bar \psi,L]}\bigg).
 \end{aligned}\end{equation}
 \par It remains to estimate $ \ddot\psi_{\sharp}^1-\ddot\psi_{\sharp}^2$. Note that
 \begin{equation}\label{6-20-dd}
  \begin{aligned}
0&=\mj(\ddot\psi_\sharp^1;\ddot {\bm {V}}_+^1,
 (\ddot\psi')^1,\ddot {\bm {V}}_-)-\mj(\ddot\psi_\sharp^2;\ddot {\bm {V}}_+^2,
 (\ddot\psi')^2,\ddot {\bm {V}}_-)\\
 &=(\ddot\psi_\sharp^1-\ddot\psi_\sharp^2)\int_0^1\frac{\p\mj}{\p \ddot\psi_\sharp}
 \bigg(s\ddot\psi_\sharp^1+(1-s)\ddot\psi_\sharp^2,\ddot {\bm {V}}_+^2,
 (\ddot\psi')^2,\ddot {\bm {V}}_-\bigg)\de s\\
 &\quad +(\ddot {\bm {V}}_+^1-\ddot {\bm {V}}_+^2)\cdot\int_0^1 \n_{\ddot {\bm {V}}_+}\mj\bigg(\ddot\psi_\sharp^2;s\ddot {\bm {V}}_+^1+(1-s)\ddot {\bm {V}}_+^2,
 (\ddot\psi')^2,\ddot {\bm {V}}_-\bigg)\de s\\
  &\quad+\left((\ddot\psi')^1-(\ddot\psi')^2\right)\int_0^1\frac{\p\mj}{\p\ddot\psi'}\mj\bigg(\ddot\psi_\sharp^2;\ddot {\bm {V}}_+^2,
 s(\ddot\psi')^1+(1-s)(\ddot\psi')^2,\ddot {\bm {V}}_-\bigg)\de s.
 \end{aligned}
 \end{equation}
 Then similar computations as in Lemma \ref{le-6-1} lead to
 \begin{equation*}
  \begin{aligned}
  &\frac{\p\mj}{\p \ddot\psi_\sharp}
 \bigg(s\ddot\psi_\sharp^1+(1-s)\ddot\psi_\sharp^2,\ddot {\bm {V}}_+^2,
 (\ddot\psi')^2,\ddot {\bm {V}}_-\bigg)=O(1)\sigma, \\
 &\n_{\ddot {\bm {V}}_+}\mj\bigg(\ddot\psi_\sharp^2;s\ddot {\bm {V}}_+^1+(1-s)\ddot {\bm {V}}_+^2,
 (\ddot\psi')^2,\ddot {\bm {V}}_-\bigg)=O(1)\sigma, \\
 &\frac{\p\mj}{\p\ddot\psi'}\mj\bigg(\ddot\psi_\sharp^2;\ddot {\bm {V}}_+^2,
 s(\ddot\psi')^1+(1-s)(\ddot\psi')^2,\ddot {\bm {V}}_-\bigg)=O(1)\sigma, \\
 \end{aligned}
 \end{equation*}
 where $O(1)$ is a constant depending only on $(\h {\bm V}_\pm,L,\bar m)$. Thus it follows from \eqref{6-20-dd} that
 \begin{equation*}
 |\ddot\psi_\sharp^1-\ddot\psi_\sharp^2|\leq C\bigg(\|(\ddot u_{1,\ast}^1-\ddot u_{1,\ast}^2,\ddot u_{2,\ast}^1-\ddot u_{2,\ast}^2\|_{1,\alpha;\mn_+^z}^{(-\alpha)}+
 \|(\ddot\psi')^1-(\ddot\psi')^2\|_{1,\alpha;(0,\bar m)}^{(-\alpha)}\bigg).
 \end{equation*}
 This, together with \eqref{4-5-sub-sub-1 that}, yields
 \begin{equation}\label{4-5-sub-sub-1 that-1}
 \begin{aligned}
&\|(\ddot u_{1,\ast}^1-\ddot u_{1,\ast}^2,\ddot u_{2,\ast}^1-\ddot u_{2,\ast}^2,\ddot S_\ast^1-\ddot S_\ast^2)\|_{1,\alpha;\mn_+^z}^{(-\alpha)}+\|(\ddot\psi_\ast^1)'
-(\ddot\psi_\ast^2)'\|_{1,\alpha;(0,\bar m)}^{(-\alpha)}\\
&\leq C\sigma\bigg(\|(\ddot u_{1,+}^1-\ddot u_{1,+}^2,\ddot u_{2,+}^1-\ddot u_{2,+}^2,\ddot S_+^1-\ddot S_+^2)\|_{1,\alpha;\mn_+^z}^{(-\alpha)}+\|(\ddot\psi')^1-(\ddot\psi')^2\|_{1,\alpha;(0,\bar m)}^{(-\alpha)}\bigg).
\end{aligned}
 \end{equation}
 Then the contraction of the mapping $\mt $ follows directly from \eqref{4-5-sub-sub-1 that-1} for sufficiently small $ \sigma$, which completes the proof of Lemma \ref{l0}.
 \end{proof}
 \par {\bf Acknowledgement.}   The research of Zihao Zhang is supported by  the Postdoctoral Fellowship Program of CPSF under Grant Number GZB20250719.
\par {\bf Data availability.} No data was used for the research described in the article.
    \par {\bf Conflict of interest.} This work does not have any conflicts of interest.

\end{document}